\documentclass[11pt]{article}
\usepackage{amsmath,amsthm,amssymb}
\usepackage[usenames,dvipsnames]{xcolor}
\usepackage{enumerate}
\usepackage{graphicx}
\usepackage{cite}
\usepackage{comment}
\usepackage{oands}
\usepackage{tikz}
\usepackage{changepage}
\usepackage{bbm}
\usepackage{mathtools}
\usepackage[margin=1in]{geometry}
\usepackage[pagewise,mathlines]{lineno}
\usepackage{appendix}
\usepackage{stmaryrd} 
\usepackage{multicol}
\usepackage{microtype}
\usepackage[colorinlistoftodos]{todonotes}
\usepackage{enumitem}
\usepackage{subfigure}

\usepackage[pdftitle={External diffusion-limited aggregation on a spanning-tree-weighted random planar map},
  pdfauthor={Ewain Gwynne and Joshua Pfeffer},
colorlinks=true,linkcolor=NavyBlue,urlcolor=RoyalBlue,citecolor=PineGreen,bookmarks=true,bookmarksopen=true,bookmarksopenlevel=2,unicode=true,linktocpage]{hyperref}

\theoremstyle{plain}
\newtheorem{thm}{Theorem}[section]
\newtheorem{cor}[thm]{Corollary}
\newtheorem{lem}[thm]{Lemma}
\newtheorem{prop}[thm]{Proposition}

\newtheorem{prob}{Problem}[section]

\def\@rst #1 #2other{#1}
\newcommand\MR[1]{\relax\ifhmode\unskip\spacefactor3000 \space\fi
  \MRhref{\expandafter\@rst #1 other}{#1}}
\newcommand{\MRhref}[2]{\href{http://www.ams.org/mathscinet-getitem?mr=#1}{MR#2}}

\theoremstyle{definition}
\newtheorem{defn}[thm]{Definition}

\newtheorem{remark}[thm]{Remark}

\newtheorem{notation}[thm]{Notation}

  \newcounter{lst}
\newenvironment{lst}{%
\refstepcounter{lst}%
\begin{center}
\begin{minipage}{.9\textwidth}}{%
\end{minipage}%
\makebox[.1\textwidth][r]{(\thelst)}%
\end{center}}

\makeatletter
\newcommand{\neutralize}[1]{\expandafter\let\csname c@#1\endcsname\count@}
\makeatother

\numberwithin{equation}{section}

\newcommand{\dsb}{\begin{adjustwidth}{2.5em}{0pt}
\begin{footnotesize}}
\newcommand{\dse}{\end{footnotesize}
\end{adjustwidth}}

\newcommand{\ssb}{\begin{adjustwidth}{2.5em}{0pt}}
\newcommand{\sse}{\end{adjustwidth}}

\newcommand{\aryb}{\begin{eqnarray*}}
\newcommand{\arye}{\end{eqnarray*}}
\def\alb#1\ale{\begin{align*}#1\end{align*}}
\def\allb#1\alle{\begin{align}#1\end{align}}
\newcommand{\eqb}{\begin{equation}}
\newcommand{\eqe}{\end{equation}}
\newcommand{\eqbn}{\begin{equation*}}
\newcommand{\eqen}{\end{equation*}}

\newcommand{\BB}{\mathbbm}
\newcommand{\ol}{\overline}
\newcommand{\ul}{\underline}
\newcommand{\op}{\operatorname}

\newcommand{\eqD}{\overset{\mathcal{L}}{=}}
\newcommand{\ep}{\epsilon}
\newcommand{\rta}{\rightarrow}

\newcommand{\wt}{\widetilde}
\newcommand{\wh}{\widehat} 
\newcommand{\mcl}{\mathcal}

\newcommand{\bdy}{\partial}

\let\originalleft\left
\let\originalright\right
\renewcommand{\left}{\mathopen{}\mathclose\bgroup\originalleft}
\renewcommand{\right}{\aftergroup\egroup\originalright}

\title{External diffusion-limited aggregation on a spanning-tree-weighted random planar map}

\date{  }
\author{
\begin{tabular}{c} Ewain Gwynne\\[-5pt]\small Cambridge \end{tabular}
\begin{tabular}{c} Joshua Pfeffer\\[-5pt]\small MIT \end{tabular} 
}

\begin{document}

\maketitle

\begin{abstract}  
Let $M$ be the infinite spanning-tree-weighted random planar map, which is the local limit of finite random planar maps sampled with probability proportional to the number of spanning trees they admit. 
We show that a.s.\ the $M$-graph-distance diameter of the external diffusion-limited aggregation (DLA) cluster on $M$ run for $m$ steps is of order $m^{2/d + o_m(1)}$, where $d$ is the metric ball volume growth exponent for $M$ (which was shown to exist by Ding-Gwynne, 2018). 
By known bounds for $d$, one has $0.55051\dots \leq 2/d \leq 0.563315\dots$. 

Along the way, we also prove that loop-erased random walk (LERW) on $M$ typically travels graph distance $m^{2/d  + o_m(1)}$ in $m$ units of time and that the graph-distance diameter of a finite spanning-tree-weighted random planar map with $n$ edges, with or without boundary, is of order $n^{1/d+o_n(1)}$ except on an event with probability decaying faster than any negative power of $n$.

Our proofs are based on a special relationship between DLA and LERW on spanning-tree-weighted random planar maps as well as estimates for distances in such maps which come from the theory of Liouville quantum gravity. 
\end{abstract}

\tableofcontents

\section{Introduction}
\label{sec-intro}
 
\label{Overview}
\label{sec-overview}

This paper studies a random growth process on a connected graph called \textbf{external diffusion-limited aggregation} (abbrv. DLA).  
DLA describes a process in which new edges are randomly added to a growing cluster according to harmonic measure viewed from some target vertex.

\begin{defn}[External DLA]
Let $G$ be a connected graph, and let $v_0,v_*$ be vertices in $G$.  We define \emph{external DLA} with \emph{initial vertex} $v_0$ and \emph{target vertex} $v_*$ as a finite growing sequence $X_m$ of random subgraphs of $G$, started with $X_0 := \{v_0\}$ and defined inductively as follows:
\begin{itemize}
\item
If $X_{m-1}$ does not contain $v_*$ for some positive integer $m$, then we consider the set of edges $G\setminus X_{m-1}$ with exactly one endpoint in $X_{m-1}$, and we sample one of these edges according to harmonic measure from $v_*$. In other words, we run a simple random walk on $G$ started from $v_*$ conditioned to hit $X_{m-1}$---this condition automatically holds if simple random walk on $G$ is recurrent---and we choose the last edge it traverses before it hits $X_{m-1}$.  We then define $X_m$ as the union of the subgraph $X_{m-1}$, the edge we just sampled, and the endpoint of that edge that is not already in $X_{m-1}$. 
\item
If $X_{m-1}$ does contain $v_*$, then the process terminates.
\end{itemize}
We can also define \emph{external DLA targeted at infinity}, with harmonic measure from $v_*$ replaced by harmonic measure from infinity, on infinite graphs for which this measure is well-defined.
\label{defn-dla}
\end{defn}

DLA was originally introduced by Witten and Sander in 1981 to describe random growths of ``dust balls, agglomerated soot, and dendrites'' in nature~\cite{witten-dla1,witten-dla2}, and the process has been studied widely by physicists using simulations. See the review articles~\cite{sander-dla-survey,halsey-dla-survey}  for a survey of this vast literature from a physics perspective.

By contrast, mathematical results about DLA are rather limited. Kesten~\cite{kesten-arms, kesten-dla2} showed that the diameter of the DLA cluster on ${\mathbb{Z}}^n$ (w.r.t.\ the ambient graph metric on ${\mathbb{Z}}^n$) after $m$ steps grows asymptotically no faster than $m^{2/3}$  in dimension $n=2$, no faster than $(m \log{m})^{1/2}$ in dimension $n=3$, and no faster than $m^{2/(n+1)}$ in dimensions $n>3$.  Kesten's techniques have been extended to a more general class of graphs; see, e.g.,~\cite{benjamini-yadin-dla}, and similar bounds have been obtained for DLA on the half-plane~\cite{pz-half-plane-dla} defined using the so-called stationary harmonic measure. But Kesten's bounds are far from optimal, and neither sharper upper bounds nor any non-trivial lower bounds at all have been proven rigorously since Kesten's work. It is not even known whether there exists an exponent that describes the growth of the diameter of the external DLA cluster in $\BB Z^n$ for $n \geq 2$.

There is also a substantial literature concerning generalizations and variants of DLA, such as the \emph{dielectric breakdown model}~\cite{niem-dialectic} the closely related \emph{Hastings-Levitov} model~\cite{hastings-levitov}, but so far these models remain poorly understood for the parameter values which are expected to correspond to DLA. 
We will not attempt to survey this literature in its entirety, but see~\cite{cm-aggregation,rz-laplacian-growth,nt-hastings-levitov,nst-planar-aggregation,qle,stv-laplacian-growth} for some representative results on these models. 

In this paper, we consider DLA in a \textit{random} environment called the \emph{uniform infinite spanning-tree-weighted random planar map} (abbrv.\ UITM), which is defined as follows.

\begin{defn}
\label{defn-UITM}
The \emph{uniform infinite spanning-tree-weighted random planar map} (abbrv.\ UITM) $(M , e_0  )$ is the Benjamini-Schramm local limit~\cite{benjamini-schramm-topology} of finite random planar maps sampled with probability proportional to the number of spanning trees they admit, rooted at a uniformly random oriented edge. See~\cite{shef-burger,chen-fk} for a proof that this local limit exists. We write $v_0$ for the terminal endpoint of $e_0$ and call $v_0$ the \emph{root vertex}.
\end{defn}

\begin{figure}[ht!] \centering
\includegraphics[width=0.6\textwidth]{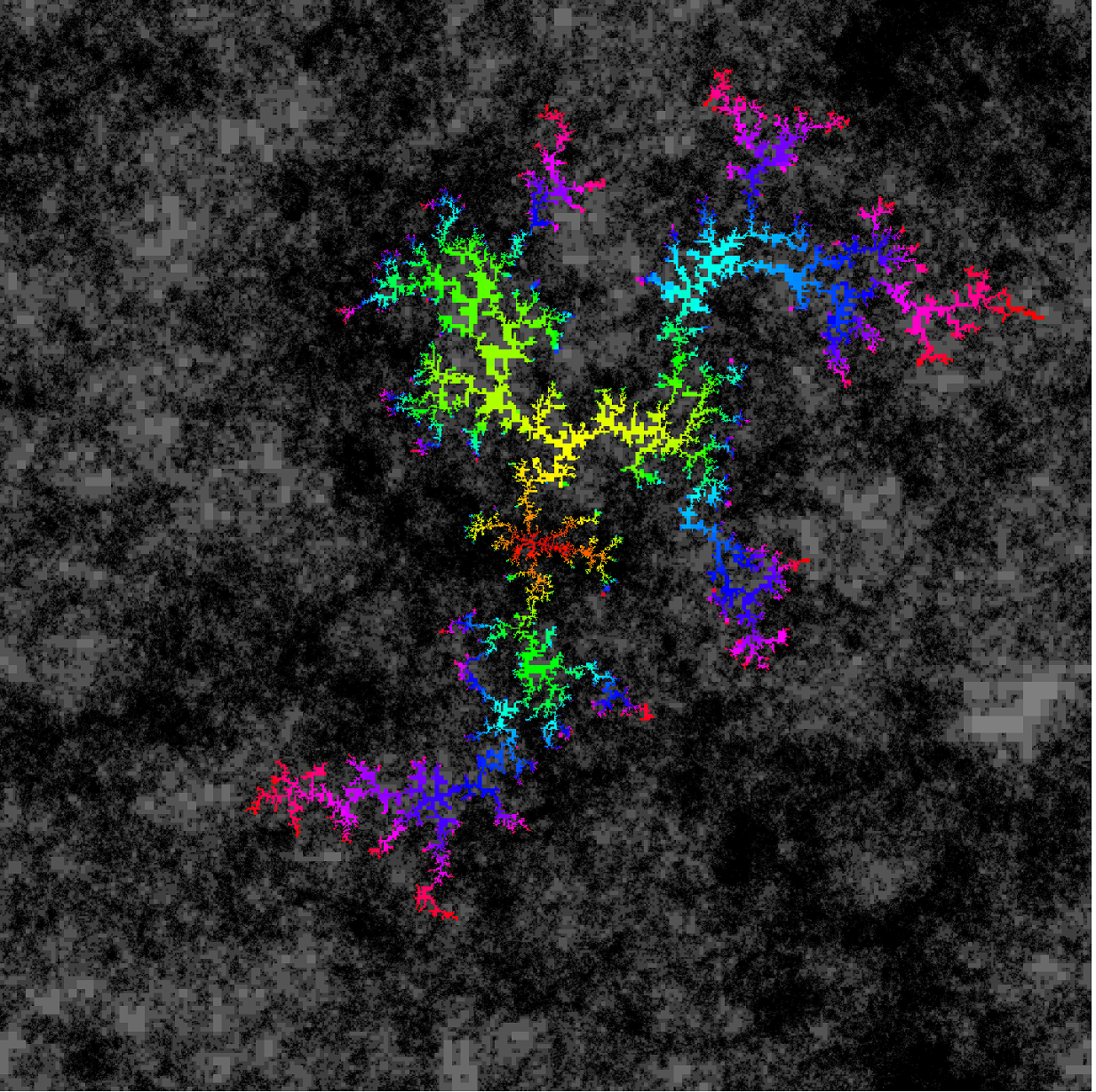}
\caption{Simulation of external DLA on a graph approximation of $\sqrt 2$-LQG made by Jason Miller. Given the connection between the UITM and  $\sqrt 2$-LQG that we describe in Section~\ref{sec-intro-tool2}, we expect, though we do not prove, that this random growth process is in the same universality class as DLA on the UITM, provided that we embed the UITM into the plane via a so-called ``discrete conformal embedding'' such as the circle packing embedding or the Tutte embedding.
}
\label{fig-dla}
\end{figure}

For reasons that we will describe further in Sections~\ref{sec-intro-tool1} and~\ref{sec-intro-tool2}, we find that DLA is much more tractable on the UITM than on, say, the Euclidean lattice ${\BB{Z}}^2$. 
Our main result identifies the growth exponent of the graph-distance diameter of external DLA on $M$ in terms of another growth exponent associated to $M$, which we now define.

\begin{defn}[Ball volume exponent]
\label{def-d}
We define the volume exponent of a metric ball in the UITM centered at the root as the limit
\begin{equation} \label{d} 
d := \lim_{r \rightarrow \infty} \frac{\log \#B_r^M(v_0)}{\log r} ,
\end{equation}
where $B_r^M(v_0)$ is the graph-distance ball of radius $r$ in $M$ centered at $v_0$ and $\#B_r^M(v_0)$ is its cardinality. 
\end{defn}

The existence of the limit in~\eqref{d} is established in~\cite[Theorem 1.6]{dg-lqg-dim}, building on~\cite{ghs-map-dist,dzz-heat-kernel}.  (Note that $d$ is referred to as $d_{\sqrt 2}$ in~\cite{dg-lqg-dim}.)   The exponent $d$ also describes the graph distance traveled by simple random walk on $M$: it was shown in~\cite{gh-displacement} that this distance after $n$ steps is typically of order $n^{1/d + o_n(1)}$. The exponent $d$ also arises in the theory of Liouville quantum gravity (LQG), see Remark~\ref{remark-d}.  We will say more about LQG and its relationship to the UITM in Section~\ref{sec-intro-tool2}.

We now state our main result. For the statement, we note that harmonic measure from $\infty$ on the UITM is well-defined (see Proposition~\ref{prop-graph-conditions2} below), which allows us to define external DLA targeted at $\infty$ on the UITM.

\begin{thm}[DLA growth exponent]
Let $(M,e_0)$ be the UITM and let $\{X_m\}_{m\in\BB N}$ be the clusters of external DLA on $M$ started from the root vertex $v_0$ and targeted at $\infty$. With $d$ as in Definition~\ref{def-d}, we have
\eqb
\lim_{m \rightarrow \infty} \frac{\log \op{diam} (X_m ; M) }{\log m} = \frac{2}{d} \qquad\text{ a.s.,}
\eqe
where $\op{diam}(\cdot ; M)$ denotes diameter with respect to the graph distance on $M$. 
\label{main3}
\end{thm}

It is shown in~\cite[Corollary 2.5]{gp-lfpp-bounds}, building on~\cite[Theorem 1.2]{dg-lqg-dim}, that $2 ( 9   + 3 \sqrt 5 - \sqrt{3}) ( 4 - \sqrt{15}) \leq d \leq \frac13(6 + 2\sqrt 6)$ (one has $\gamma = \sqrt 2$ in our setting), which implies that the exponent $2/d$ in Theorem~\ref{main3} satisfies the upper and lower bounds
\eqb \label{eqn-dla-bounds}
0.55051 \approx \frac{3}{3+ \sqrt 6} \leq \frac{2}{d} \leq  \frac{4+\sqrt{15}}{9   + 3 \sqrt 5 - \sqrt{3}} \approx 0.563315 .
\eqe
We do not expect that either the upper or lower bound in~\eqref{eqn-dla-bounds} is optimal. See~\cite[Section 1.3]{dg-lqg-dim} for some speculation concerning the numerical value of $d$.

It is natural to wonder whether Theorem~\ref{main3} tells us anything about external DLA on $\BB Z^2$ via some version of the KPZ formula~\cite{kpz-scaling,shef-kpz}. As far as we know, it does not, even at a heuristic level, since the scaling limits of external DLA on $M$ (see Footnote~\ref{footnote-dla-scaling-limit}) and on $\BB Z^2$ are not expected to agree in law. At a heuristic level, this can be seen since the simulation of DLA on $\sqrt 2$-LQG in Figure~\ref{fig-dla} looks qualitatively different from simulations of DLA on $\BB Z^2$ (e.g., in the sense that the lengths of the ``arms" are much less uniform). See also~\cite{dla-lattice-dependence} for some numerical evidence that the scaling limit of DLA should be lattice-dependent.  

We devote the rest of this introductory section to describing \emph{why} we are able to derive a precise growth exponent for external DLA specifically for the UITM.  The central idea is that, if we decorate the UITM by a uniform spanning tree on the map, then the joint law of the map and the tree is uniform on the set of such pairs. This obvious property of the UITM is actually quite powerful, and as we explain in the following two subsections, it gives us two tools for studying DLA on the UITM:
\begin{enumerate}
\item A link on the UITM between external DLA and \emph{loop-erased random walk}.
\item A link between UITM and \emph{Liouville quantum gravity} via the so-called \emph{mated-CRT map}.
\end{enumerate}
Our proof of Theorem~\ref{main3} rests on applying these two tools; we now elaborate on each of them in turn.

\subsection{Tool 1: DLA and loop-erased random walk on the UITM}
\label{sec-intro-tool1}

The first tool that we use is a relationship between DLA and the following related growth process on a graph.

\begin{defn}
We define \emph{loop-erased random walk} (LERW) on a graph $G$ from a vertex $v_0$ to a vertex $v_*$ as a random path in $G$ started at $v_0$, in which we add each sucessive edge of the walk by sampling an edge adjacent to the tip of the path according to harmonic measure from $v_*$ at the tip (with harmonic measure defined as in Definition~\ref{defn-dla}). We define LERW from $v_0$ to infinity the same way, with harmonic measure from $v_*$ replaced by harmonic measure from infinity, on graphs for which the latter measure is well-defined.
\end{defn}

We call this path a loop-erased random walk because we can generate it by sampling a simple random walk from $v_0$ run until it hits $v_*$, and then erasing all the loops of this random walk in chronological order.  {Note also that one can view LERW as the Laplacian-$b$ random walk defined in~\cite{lawler-laplacian-walk} with $b=1$.}

The key combinatorial fact that links external DLA to LERW on the UITM is \emph{Wilson's algorithm}~\cite{wilson-algorithm}, a famous algorithm in combinatorics that uses loop-erased random walk on a finite graph $G$ to generate a uniform spanning tree on $G$. If $G$ is infinite, then it is not \emph{a priori} clear how to define a ``uniform spanning tree'' on $G$.  The next proposition asserts that, for recurrent graphs, the uniform spanning tree can be defined as the limit of uniform spanning trees on finite graphs, and that the resulting tree is equivalent to the tree obtained on the infinite graph via Wilson's algorithm.

\begin{prop}
\label{prop-graph-conditions}
Suppose that $G$ is an infinite graph for which simple random walk is recurrent. Then, if $G_n$ is an increasing sequence of subgraphs of $G$ whose union is all of $G$, the uniform spanning trees on $G_n$---viewed as measures on the product space $2^{E}$, for $E$ the set of edges of $G$---converge weakly as $n \rta \infty$ to a tree on $G$ that we call the \emph{uniform spanning tree} on $G$.  Moreover, this tree is the same as the tree obtained by applying Wilson's algorithm on $G$.  This means that, for any vertices $v_0$ and $v_*$ in $G$, sampling an edge adjacent to $v_0$ according to harmonic measure viewed from $v_*$ is equivalent to sampling a uniform spanning tree $T$ on $G$, and choosing the first edge on the path in $T$ from $v_0$ to $v_*$.

Finally, we can apply this result to the UITM almost surely since simple random walk on the UITM is a.s.\ recurrent.
\end{prop}
\begin{proof}
The result for infinite recurrent graphs follows from~\cite[Proposition 5.6]{blps-usf}.   The fact that simple random walk on the UITM is a.s.\ recurrent follows from~\cite[Theorem 1.1]{gn-recurrence} (see also~\cite{chen-fk}).
\end{proof}

Thus, for infinite recurrent graphs $G$ (including the UITM), the law of LERW from $v_0$ to $v_*$ is the law of the unique branch from $v_0$ to $v_*$  of a uniform spanning tree on $G$.  In particular, since simple random walk on the UITM is a.s.\ recurrent (this follows from~\cite[Theorem 1.1]{gn-recurrence}---see also~\cite{chen-fk}), we can a.s.\ sample a uniform spanning tree on the UITM.  

Moreover, if the uniform spanning tree on the infinite graph is \emph{almost surely one-ended}, then we get a similar characterization of harmonic measure from infinity (and LERW to infinity) in terms of the tree:

\begin{prop}
\label{prop-graph-conditions2}
Suppose that $G$ is an infinite graph for which simple random walk is recurrent and the uniform spanning tree is a.s.\ one-ended.  Then, for any subgraph $A$ of $G$, harmonic measure on $A$ from infinity is well-defined as the limit of harmonic measure on $A$ from a vertex $v$ as the graph distance between $v$ and $A$ tends to infinity.  Moreover,  for any vertex $v_0$  in $G$, sampling an edge adjacent to $v_0$ according to harmonic measure viewed from infinity is equivalent to sampling a uniform spanning tree $T$ on $G$, and choosing the first edge on the (unique) path in $T$ from $v_0$ to infinity.
\end{prop}
\begin{proof}
This result is~\cite[Theorem 14.2]{blps-usf}.
\end{proof}

Thus, on this large class of graphs $G$---which we show in Section~\ref{sec-lerw-dla} includes the UITM---we can describe the law of LERW from $v_0$ to infinity as the law of the unique branch from $v_0$ to  infinity of a uniform spanning tree on $G$. 

As we will prove in Section~\ref{sec-lerw-dla}, this result yields an exact relation between the laws of DLA and LERW on the UITM---a relation that does \emph{not} appear to be satisfied for any other types of random planar maps (e.g., uniform maps or planar maps with other weighting).  Briefly, suppose we run $m$ steps of a LERW from the root to infinity on a rooted infinite planar map $M$. We can ``cut" along the edges of this length-$m$ path---replacing each such edge by a pair of edges---to obtain an infinite planar map $M^{(m)}$ with finite boundary of length $2m$. We can perform the same cutting operation on a DLA cluster run for $m$ steps to obtain another infinite random planar map $\wt M^{(m)}$ with finite boundary of length $2m$.  The key fact that we observe is that, if we take $M$ to have the law of a UITM, then \emph{$M^{(m)}$ and $\wt M^{(m)}$ have the same law}. See Figure~\ref{fig-cutting} for an illustration of the cutting procedure. We note that similar cutting procedures for various processes on random planar maps have been used elsewhere in the literature, e.g., in the case of a self-avoiding walk~\cite{dup-kos-saw,caraceni-curien-saw,gwynne-miller-saw}, a collection of loops~\cite{bbg-recursive-approach}, and a (non-spanning) tree~\cite{fs-tree-decorated}.

\begin{figure}[ht!] \centering
\begin{tabular}{ccc} 
\includegraphics[width=0.33\textwidth]{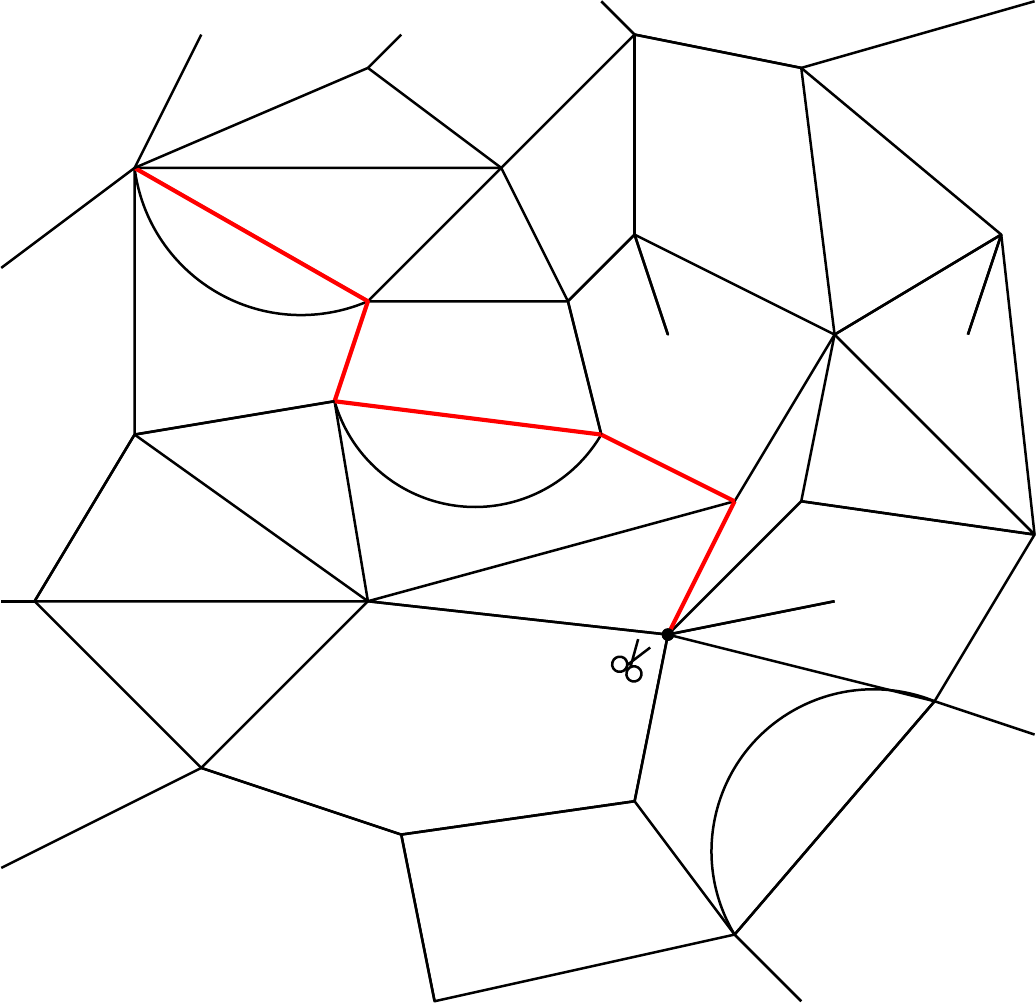}
&
\includegraphics[width=0.33\textwidth]{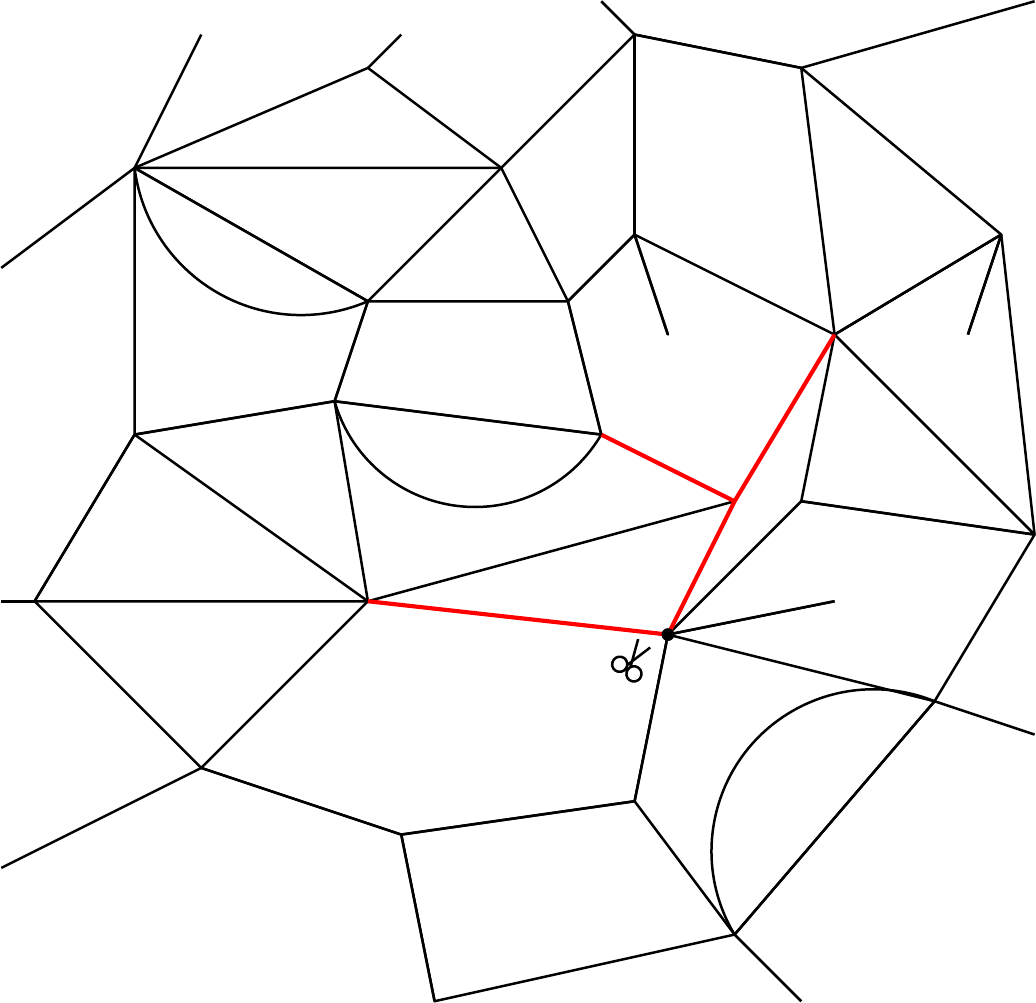}
\\
\includegraphics[width=0.33\textwidth]{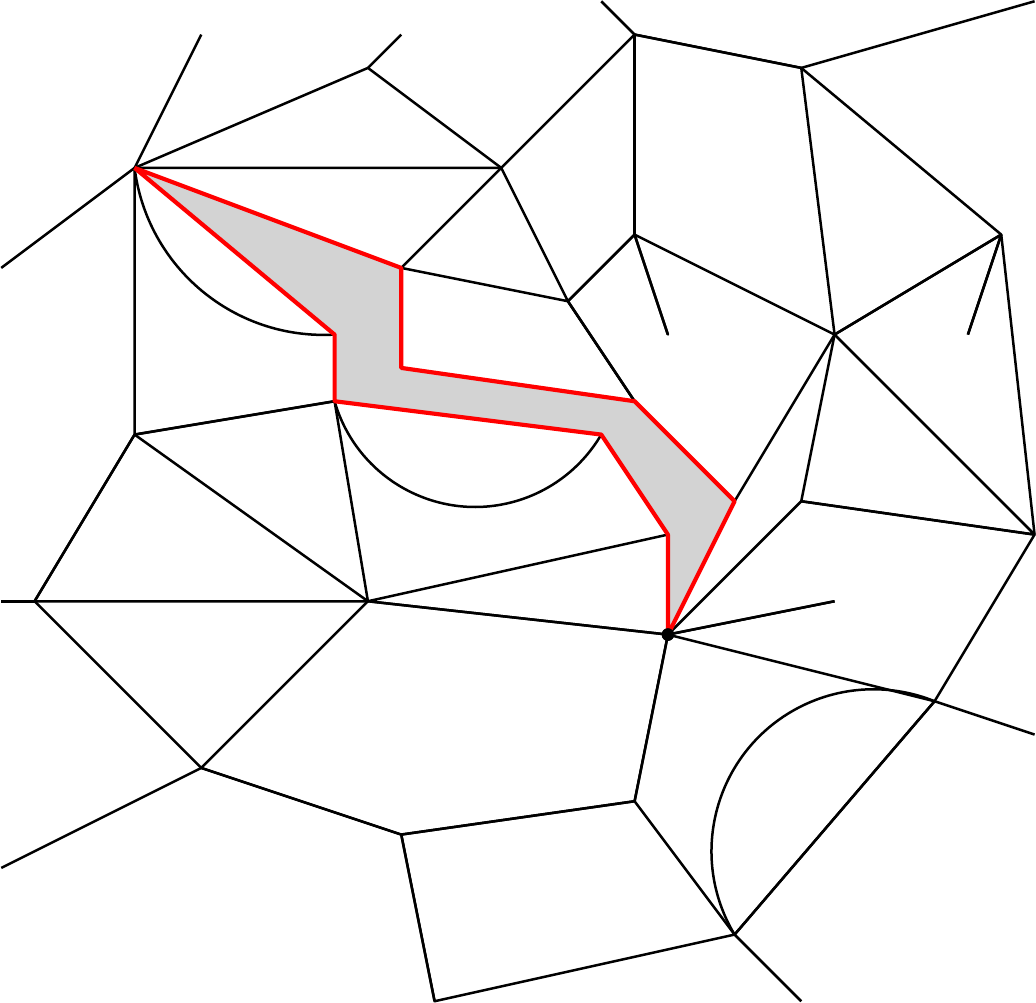}
&
\includegraphics[width=0.33\textwidth]{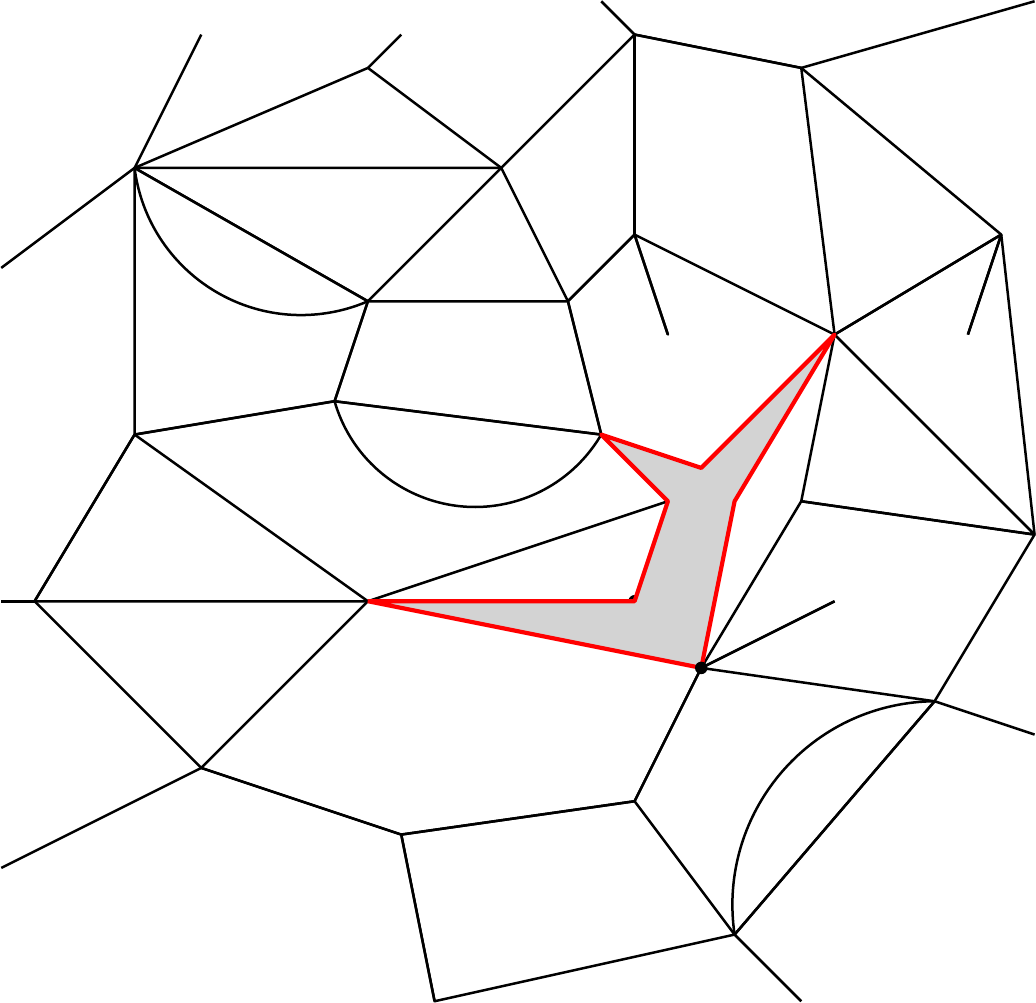}
\end{tabular}
\caption{\textbf{Top Left:} We color in red the first four edges of a LERW from the black vertex to infinity on an infinite random planar map $M$.    \textbf{Bottom Left:} The infinite random planar map $M^{(4)}$ with finite boundary of length $8$ obtained by cutting along the edges of the first four edges of the LERW. \textbf{Top/Bottom Right:} We do the same for the first four edges of edge-based external DLA started from the black vertex and targeted at infinity to obtain the map $\wt M^{(4)}$.
}
\label{fig-cutting}
\end{figure}

The above property is closely related to the representation of DLA as ``re-shuffled loop-erased random walk" in~\cite[Section 2.3]{qle}.\footnote{\label{footnote-dla-scaling-limit} Miller and Sheffield~\cite{qle} use a continuum version of this ``re-shuffled loop-erased random walk" to construct a candidate for the scaling limit of external DLA on a spanning-tree-weighted random planar map, namely the \emph{quantum Loewner evolution} with $\gamma^2 = 2$ and $\eta = 1$ (denoted QLE$(2,1)$). More specifically, they construct the QLE$(2,1)$ process  by randomly ``re-shuffling" SLE$_2$ curves on a $\sqrt 2$-Liouville quantum gravity surface and taking a (subsequential) limit as the time increments between re-shuffling operations goes to zero. Currently, there are no known rigorous relationships between QLE$(2,1)$ and the discrete objects studied in the present paper.}

Using the connection we just described between DLA and LERW, we will prove Theorem~\ref{main3} by proving the same growth exponent for loop-erased random walk on $M$.

\begin{thm}[Loop erased random walk growth exponent] \label{thm-lerw}
Consider a loop-erased random walk on the UITM $M$ from $v_0$ to $\infty$; equivalently, consider the branch from $v_0$ to $\infty$ of the uniform spanning tree on $M$.  For $m\in\BB N$, let $\op{LERW}_m$ be the set of edges which it traverses in its first $m$ steps. 
Almost surely,
\eqb
\lim_{m\rta\infty} \frac{\log \op{diam} (\op{LERW}_m  ; M) }{\log m} = \frac{2}{d} .
\eqe
In particular, by Theorem~\ref{main3} and the discussion preceding it, the growth exponent for loop erased random walk on $M$ is the same as the growth exponent for external DLA and is twice the growth exponent for simple random walk. 
\end{thm}

By applying the ``cutting operation'' just described, we are able to reduce the task of proving the growth exponent for DLA and LERW on the UITM to proving certain estimates for distances in spanning-tree-weighted maps with boundary.  We state one such estimate that we obtain in the course of our proof as a theorem.

\begin{thm}[Diameter of finite tree-weighted planar maps] \label{thm-finite-map}
Let $M_n$ be a finite spanning-tree-weighted random planar map with $n$ total edges, either without boundary or with a simple boundary cycle of specified length $\ell \leq n^{1/2}$; in the latter case, the spanning tree has wired boundary conditions. For each $\zeta \in (0,1)$, it holds except on an event of probability decaying faster than any negative power of $n$ (at a rate which does not depend on the particular choice of $\ell$) that the graph-distance diameter of $M_n$ is between $n^{1/d -\zeta}$ and $n^{1/d+\zeta}$. 
\end{thm}

\subsection{Tool 2: Coupling UITM with the mated-CRT map and $\sqrt{2}$-LQG}
\label{sec-intro-tool2}

Even with the combinatorial bijection we described in Section~\ref{sec-intro-tool1} at our disposal, we still need to derive quantitative estimates on distances in spanning-tree-weighted maps in order to prove Theorems~\ref{main3},~\ref{thm-lerw}, and~\ref{thm-finite-map}.  Understanding distances in spanning-tree-weighted maps is highly non-trivial since there is no known way to estimate such distances directly. In the case of \emph{uniform} random planar maps, we have several tools to study distances, such as the Schaeffer bijection and its generalizations~\cite{schaeffer-bijection,bdg-bijection} and peeling~\cite{angel-peeling}.  However, we cannot apply these tools to spanning-tree-weighted maps.

Instead, we will apply a different combinatorial bijection that we have in the special setting of spanning-tree-weighted maps, known as the \emph{Mullin bijection}~\cite{mullin-maps,bernardi-maps,shef-burger}.  As we describe in Section~\ref{mapintro}, the Mullin bijection allows us to construct a spanning-tree-weighted random planar map by gluing together two independent discrete random trees.  This characterization of the spanning-tree-weighted random planar map is useful because it has a direct analogue in the continuum setting, in a theory of continuum random surfaces called \emph{Liouville quantum gravity} (LQG).  Namely, if we instead glued together two independent \emph{continuum random trees} (CRTs), then by the ``mating of trees'' theorem of Duplantier, Miller and Sheffield~\cite{wedges}, we get a certain type $\sqrt 2$-LQG surface called a \emph{$\sqrt 2$-quantum cone}.

Therefore, the Mullin bijection gives us a direct connection between spanning-tree-weighted random planar maps and $\sqrt{2}$-LQG.  The works~\cite{ghs-dist-exponent,ghs-map-dist} used this connection to develop a general technique for applying quantitative estimates from the LQG setting to prove results for spanning-tree-weighted maps. We note that their method can be applied to any random planar map model that can be constructed by gluing together a pair of discrete random trees.  However, we will describe the  technique only for the case of spanning-tree-weighted maps. 

The central principle of the technique in~\cite{ghs-dist-exponent,ghs-map-dist} is that we can naturally couple both the spanning-tree-weighted map and the corresponding $\sqrt{2}$-LQG  surface to a third model: a random planar map that we construct by gluing together two independent \emph{discretized} CRTs.  We call this model the \emph{mated-CRT map}, and we couple it to the spanning-tree-weighted map and to $\sqrt{2}$-LQG as follows:
\begin{itemize}
\item
We can use a strong coupling between random walk and Brownian motion to couple the spanning-tree-weighted map and the mated-CRT map so that the corresponding trees are close~\cite{ghs-map-dist,zaitsev-kmt}.
\item
The ``mating of trees" theorem from~\cite{wedges} implies that we can couple the mated-CRT map to $\sqrt{2}$-LQG by equivalently defining the mated-CRT map as a ``discretization'' of a $\sqrt{2}$-LQG surface.  This discretization is defined in terms of an independent Schramm-Loewner evolution curve with parameter 8 (SLE$_8$)~\cite{schramm0} sampled on the surface.  

We will explain this last point in more detail in Section~\ref{sec-mated-crt} after defining the tools from LQG needed to understand it.  For the reader already familiar with those tools, here is a more precise statement of the definition of the mated-CRT map in terms of LQG. Let $h$ be the variant of the whole-plane Gaussian free field corresponding to the so-called $\sqrt 2$-quantum cone and let $\mu_h$ its associated $\sqrt 2$-LQG measure as defined in~\cite{shef-kpz}. 
Let $\eta$ be an SLE$_8$ from $\infty$ to $\infty$ sampled independently from $h$ and parametrize $\eta$ so that $\mu_h(\eta([s,t]) = t-s$ whenever $s,t\in \BB R$ with $s < t$. Then the mated-CRT map agrees in law with the adjacency graph of unit LQG mass ``cells" $\eta([x-1,x])$ for $x\in \BB Z$. 
\end{itemize}

\begin{remark} \label{remark-d}
The exponent $d$ of Definition~\ref{def-d} also describes several quantities related to $\sqrt{2}$-LQG.  (The reader not familiar with LQG should feel free to skip the rest of this remark.)  For example, $d$ appears in the so-called Liouville heat kernel~\cite{dzz-heat-kernel} and in various continuum approximations (i.e., approximations defined in terms of the LQG surface) of $\sqrt{2}$-LQG distances such as Liouville graph distance and Liouville first passage percolation (LFPP)~\cite{dg-lqg-dim,dzz-heat-kernel}.  
The recent papers~\cite{dddf-lfpp,local-metrics,lqg-metric-estimates,gm-confluence,gm-uniqueness,gm-coord-change}
constructed a metric (distance function) on an LQG surface as the limit of Liouville first passage percolation. It is shown in~\cite{gp-kpz} that $d$ is the Hausdorff dimension of a $\sqrt 2$-LQG surface equipped with this metric.
\end{remark}

\medskip 
\noindent\textbf{Acknowledgments.} We thank two anonymous referees for helpful comments on an earlier version of this paper.
J.P.\ was partially supported by the National Science Foundation Graduate Research Fellowship under Grant No. 1122374.

\subsection{Outline}
\label{sec-outline}

In Section~\ref{mapintro}, we review the two important combinatorial properties of the UITM that we use in this paper: its encoding in terms of a pair of independent discrete random trees via the Mullin bijection that we described in Section~\ref{sec-intro-tool2}; and the relationship between DLA and LERW on the UITM that we described in Section~\ref{sec-intro-tool1}.

In Section~\ref{sec-proof-conditional} we apply these two combinatorial properties of the UITM to prove Theorems~\ref{main3},~\ref{thm-lerw} and~\ref{thm-finite-map} conditional on a relationship between two exponents associated with distances in the UITM, which we state as Theorem~\ref{thm-chi}. The first of these exponents is the ball volume exponent $d$ (Definition~\ref{def-d}). The other exponent, which we call $\chi$, describes the ``internal" graph-distance diameter of certain submaps of the UITM (i.e., the graph distance along paths required to stay in the submap) and is shown to exist in~\cite{ghs-dist-exponent,ghs-map-dist}. 
Theorem~\ref{thm-chi} asserts that $\chi = 1/d$, and we use it to relate distances in the UITM to distances in tree-weighted maps with boundary---in particular, we apply it to the map obtained by cutting along the edges of a LERW or DLA cluster as in Figure~\ref{fig-cutting}). 

The proof of Theorem~\ref{thm-chi} occupies most of the paper. 
To prove that $\chi = 1/d$, we will need several estimates for graph distances in spanning-tree weighted planar maps with boundary which we prove using the coupling with mated-CRT maps and Liouville quantum gravity that we described in Section~\ref{sec-intro-tool2} above. In Section~\ref{LQGreview}, we review the tools from the theory of SLE and LQG that we need in the proof. 
We then prove Theorem~\ref{thm-chi} in Sections~\ref{outline}-~\ref{proofpart2} by establishing various estimates for the mated-CRT map, then transferring to spanning-tree weighted maps via the strong coupling results of~\cite{ghs-map-dist}. See the beginnings of the individual subsections for more detail on the arguments involved.

Finally, in Section~\ref{sec-open-problems} we discuss some open problems.

\subsection{Basic notation and terminology}
\label{sec-notation}

\noindent
We write $\BB N = \{1,2,3,\dots\}$ and $\BB N_0 = \BB N \cup \{0\}$. 
\medskip

\noindent
For $a < b$, we define the discrete interval $[a,b]_{\BB Z}:= [a,b]\cap\BB Z$. 
\medskip

\noindent
If $f  :(0,\infty) \rta \BB R$ and $g : (0,\infty) \rta (0,\infty)$, we say that $f(\ep) = O_\ep(g(\ep))$ (resp.\ $f(\ep) = o_\ep(g(\ep))$) as $\ep\rta 0$ if $f(\ep)/g(\ep)$ remains bounded (resp.\ tends to zero) as $\ep\rta 0$. We similarly define $O(\cdot)$ and $o(\cdot)$ errors as a parameter goes to infinity. 
\medskip

\noindent
If $f,g : (0,\infty) \rta [0,\infty)$, we say that $f(\ep) \preceq g(\ep)$ if there is a constant $C>0$ (independent from $\ep$ and possibly from other parameters of interest) such that $f(\ep) \leq  C g(\ep)$. We write $f(\ep) \asymp g(\ep)$ if $f(\ep) \preceq g(\ep)$ and $g(\ep) \preceq f(\ep)$. 
\medskip

\noindent
Let $\{E^\ep\}_{\ep>0}$ be a one-parameter family of events. We say that $E^\ep$ occurs with
\begin{itemize}
\item \emph{polynomially high probability}  as $\ep\rta 0$ if there is a $p > 0$ (independent from $\ep$ and possibly from other parameters of interest) such that  $\BB P[E^\ep] \geq 1 - O_\ep(\ep^p)$. 
\item \emph{superpolynomially high probability}  as $\ep\rta 0$ if $\BB P[E^\ep] \geq 1 - O_\ep(\ep^p)$ for every $p>0$. 
\end{itemize}
We similarly define events which occur with polynomially or superpolynomially high probability as a parameter tends to $\infty$.  
\medskip

\noindent
If $G$ is a graph and $A,B\subset G$ are sets of vertices and/or edges, we write $\op{dist} (A,B ; G)$ for the $G$-graph distance between $A$ and $B$ (the $G$-graph distance between two edges is the minimum of the $G$-graph distance between their endpoints). 
We also write $\op{diam} (A ; G) := \sup_{x,y \in A} \op{dist} (x,y ; G)$ and $\op{diam}(G) = \op{diam}(G; G)$.
\medskip

\noindent
Finally, we use the following terminology for planar maps:
\begin{itemize}
\item
A \emph{planar map with boundary} is a planar map $M$ with a distinguished face $f_\infty$, called the \emph{external face}. 
\item
The \emph{boundary} $\partial M$ of $M$ is the subgraph of $M$ consisting of the vertices and edges on the boundary of $f_\infty$. 
\item
$M$ is said to have \emph{simple boundary} if $\bdy M$ is a simple cycle (equivalently, each vertex of $\bdy M$ corresponds to a single prime end of $f_\infty$). 
\end{itemize}
All of our planar maps with boundary will have simple boundary.

\section{Two key combinatorial properties of the UITM} 
\label{mapintro}

In this section, we will describe in greater detail the two key combinatorial properties of the UITM that, as we described in Sections~\ref{sec-intro-tool1} and~\ref{sec-intro-tool2}, provide us with the two tools we need to analyze external DLA on the UITM:
\begin{enumerate}
\item
The relationship between DLA and LERW on the UITM, which allows us to reduce the task of deriving Theorems~\ref{main3},~\ref{thm-lerw} and~\ref{thm-finite-map} to proving certain estimates for distances in spanning-tree-weighted maps with boundary.
\item
The \emph{Mullin bijection}, which will allow us to prove those distance estimates by coupling the UITM to $\sqrt{2}$-LQG via the mated-CRT map.
\end{enumerate} 

In Section~\ref{sec-mullin}, we review the Mullin bijection; and, in Section~\ref{sec-lerw-dla}, we describe and prove the relationship between DLA and LERW on the UITM. (We have organized the two subsections in this order because we will need a result in the second subsection that can easily be seen by applying the Mullin bijection.)
 
\subsection{The Mullin bijection}
\label{sec-mullin}

In this section, we review the Mullin bijection, a combinatorial bijection that encodes a spanning-tree-decorated planar map by a nearest-neighbor walk on $\BB Z^2$. This bijection was first discovered in~\cite{mullin-maps} and is explained more explicitly in~\cite{bernardi-maps,shef-burger}.   We will describe the infinite-volume version of the Mullin bijection, which describes a correspondence between the UITM and a standard nearest-neighbor simple random walk in $\BB{Z}^2$. At the end of the subsection, we will briefly describe the version of the Mullin bijection for finite spanning-tree-decorated planar maps (possibly with boundary), since we will use it in the proof of Theorem~\ref{thm-finite-map}.  

We follow closely the exposition in~\cite{ghs-map-dist}.  We have also included an illustration of the bijection in Figure~\ref{fig-mullin-bijection}.

\begin{figure}[ht!]
\begin{center}
\includegraphics[scale=.6]{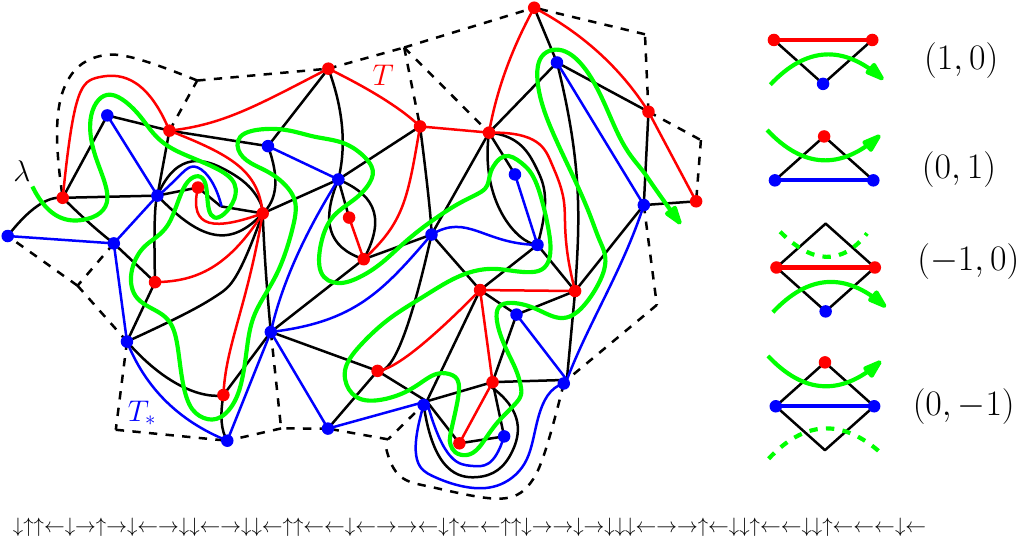} 
\caption[Infinite-volume Mullin bijection]{\label{fig-mullin-bijection} \textbf{Left:} The subset of the triangulation $\mcl Q \cup T\cup T_*$ consisting of the triangles in $\lambda([a,b]_{\BB Z})$, with edges of $\mcl Q$ (resp.\ $T$, $T_*$) shown in black (resp.\ red, blue). The green curve indicates the order in which the triangles are hit by $\lambda$. The edges of triangles of $\mcl T$ which do not belong to $\lambda([a,b]_{\BB Z})$, but which are adjacent to triangles in $\lambda([a,b]_{\BB Z})$, are shown as dotted lines. The vertices of $M|_{[a,b]}$ are the red vertices and the edges of $M|_{[a,b]}$ are the red edges shown in the figure and the edges of $M$ which cross the blue edges in the figure (not shown). \textbf{Right:} The correspondence between $M$ and $\mcl Z$. If $i\in \BB Z$ and $\lambda(i)$ has an edge in the red tree $ T$, then $\mcl Z_i - \mcl Z_{i-1}$ is equal to $(1,0)$ or $(-1,0)$ according to whether the other triangle which shares this same edge of $T$ is hit by $\lambda$ before or after time $i$. 
The other coordinate of $\mcl Z$ is defined symmetrically. 
\textbf{Bottom:} The steps of $\mcl Z$ corresponding to this segment of $\lambda$ if we assume that each of the exterior triangles (i.e., those with dotted edges) is hit by $\lambda$ before each of the non-dotted triangles. 
}
\end{center}
\end{figure}

Let $(M,e_0,T)$ be the UITM decorated by a uniform spanning tree on the map (see Proposition~\ref{prop-graph-conditions}). We begin by defining several maps in terms of $(M,e_0,T)$:
\begin{itemize}
\item
Let $M_*$ be the dual map of $M$, i.e., the map whose vertices correspond to faces of $M$ and with two vertices joined by an edge iff the corresponding faces of $M$ share an edge.  
\item Let $T_*$ be the dual spanning tree of $T$, i.e., the tree whose edges are the set of edges of $M_*$ which cross edges in $M \backslash T$.  
\item
Let $\mcl Q = \mcl Q(M)$ be the \emph{radial quadrangulation}, whose vertex set is the union of the vertex sets of $M$ and $M_*$, with two vertices connected by an edge iff one is a vertex of $M_*$ and the other is  a vertex of $M$ incident to the face of $M$ that the first vertex represents. We declare that the root edge of $\mcl Q$ is the edge $\BB e_0 $ with the same initial endpoint of $e_0$ and which is the first edge in $\mcl Q$ with this initial endpoint moving in clockwise order from $e_0$.
\end{itemize}

Each face of $\mcl Q$ is a quadrilateral with one diagonal an edge in $M$ and one an edge in $M_*$; exactly one of these diagonals corresponds to an edge in $T \cup T_*$.  This implies that  the union of $\mcl Q$, $T$ and $T_*$ forms a triangulation with the same vertex set as $\mcl Q$. Let $\mcl T$ be the planar dual of this triangulation, so that $\mcl T$ is the adjacency graph on triangles of $\mcl Q\cup T\cup T_*$, where two triangles of $\mcl Q\cup T\cup T_*$ are adjacent if they share an edge. 
We declare that the root edge of $\mcl T$ is the edge of $\mcl T$ that crosses $\BB e_0$, oriented so that the initial endpoint of $\BB e_0$ is to its left.

Let $\lambda$ be the unique path from $\BB Z$ onto the set of vertices of $\mcl T$ (i.e., the set of triangles of $\mcl Q\cup T\cup T_*$) such that
\begin{itemize}
\item $\lambda(0)$ and $\lambda(1)$ are the initial and terminal points of the root edge of $\mcl T$, respectively; and,
\item for each $j \in \BB{Z}$, the triangles corresponding to $\lambda(j)$ and $\lambda(j-1)$ share an edge of $\mcl Q$.
\end{itemize}
In other words, the path  $\lambda$ passes through each triangle of $\mcl Q\cup T\cup T_*$ without crossing any of the edges of either $T$ or $T_*$.

We use the path $\lambda$ to define a walk $\mcl Z  = (\mcl L , \mcl R)$ on $\BB Z^2$, parametrized by $\BB Z$ and with increments in the set $\{(0,1) , (1,0) , (-1,0) , (0,-1)\}$, as follows. Set $\mcl Z_0 := 0$. For $j\in\BB Z$, we set $\mcl Z_j - \mcl Z_{j-1}$ equal to $(1,0)$   (resp.\ $(-1,0)$) if the triangle corresponding to $\lambda(j)$ shares an edge of $T$ with a triangle hit by $\lambda$ after  (resp. before) time $j$; and $(0,1)$   (resp.\ $(0,-1)$) if the triangle corresponding to $\lambda(j)$ shares an edge of $T_*$ with a triangle hit by $\lambda$ after  (resp. before) time $j$.

Since $(M,e_0,T)$ is random, the walk $\mcl Z$ that we have constructed is a random walk on $\BB Z^2$. 
The law of this walk is known explicitly (see, e.g.,~\cite[Section 4.2]{shef-burger} for a proof). 

\begin{thm}[The Mullin bijection for the UITM] \label{thm-mullin}
The walk $\mcl Z$ constructed above from the UITM $(M,e_0,T)$ has the law of a standard nearest-neighbor simple random walk in $\BB Z^2$.
\end{thm}

Given the walk $\mcl Z$, we can a.s.\ recover $(M,e_0,T)$ by first building $\mcl Q\cup T\cup T_*$ one triangle at a time via a \emph{sewing procedure}. This direction of the bijection is often called the \emph{inverse Mullin bijection}. Roughly speaking, given the submap traced by $\lambda$ in the time interval $(-\infty,t]$, the increment $\lambda(t+1)-\lambda(t)$ determines how to construct the submap traced in the time interval $(-\infty,t+1]$.  See Figure~\ref{fig-mullin-inverse} for a demonstration of how this procedure works.  We refer the reader to~\cite[Section 3.1.2]{ghs-map-dist} and~\cite[Section 4.1]{chen-fk} for further details on the inverse Mullin bijection.

\begin{figure}[ht!]
\begin{center}
\includegraphics[scale=.8]{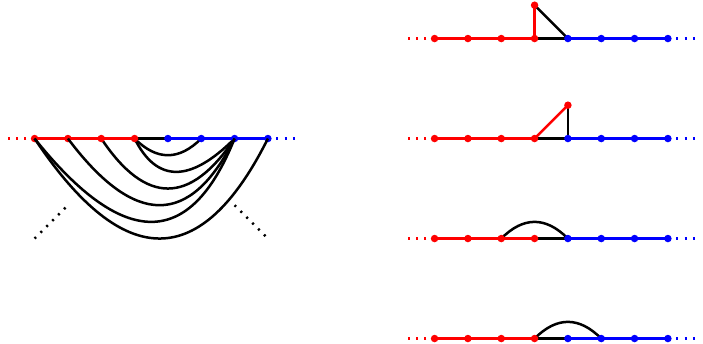} 
\caption[The inverse Mullin bijection]{\label{fig-mullin-inverse} An illustration of the sewing procedure in the inverse Mullin bijection, modeled closely after~\cite[Figure 7]{ghs-map-dist}.  \textbf{Left:} We have drawn (part of) the submap traced by $\lambda$ in the time interval $(-\infty,t]$ so that the boundary of the submap is a horizontal line, and the interior of the submap lies below that line.  Note that the boundary of the submap consists of an infinite ray of edges of $T$ (red) and an infinite ray of edges of $T_*$ (blue), joined together by a single edge of $Q$ (black).  \textbf{Right:} The value of the increment $\lambda(t+1)-\lambda(t)$ determines how to add a triangle to the submap to obtain the submap traced by $\lambda$ in the time interval $(-\infty,t+1]$.  We have demonstrated this for each of the four possible values of the increment $\lambda(t+1)-\lambda(t)$.
}
\end{center}
\end{figure}

There is a variant of the Mullin bijection for finite spanning-tree-weighted planar maps, with or without boundary, which corresponds to restricting the walk $\mcl Z$ to a finite interval and conditioning on the appropriate event.  In the setting of the infinite-volume Mullin bijection described above, we denote by $M|_{[a,b]}$ (for any two real numbers $a<b$) the submap of $M$ consisting of the vertices belonging to $M$ which lie on triangles of $Q \cup T \cup T_*$ traced by the curve $\lambda$ during the time interval $[a,b]_{ \BB Z}$, along with all edges in $M$ connecting pairs of such vertices. We define the Mullin bijection for finite spanning-tree-weighted planar maps in terms of the infinite-volume version as follows:
\begin{itemize}
\item (Without boundary) Suppose $n\in\BB N$ and we condition on the event that the walk $\mcl Z|_{[0,2n]_{\BB Z}}$ stays in $\BB N_0^2$ and satisfies $\mcl Z_{2n} = (0,0)$. 
Let $M_n = M_{[0,2n]}$ and let $T_n :=  T\cap M_n $.
Then the conditional law of the decorated map $(M_n , e_0 , T_n)$ is uniform on the set of triples consisting of a planar map with $n$ edges together with an oriented root edge and a spanning tree.
\item (With boundary) Suppose $n  \in \BB N$ and $\ell \in \BB N_0$ and we condition on the event that the walk $\mcl Z|_{[0,2n]_{\BB Z}}$ stays in $\BB N_0^2$ and satisfies $\mcl Z_n = (\ell,0)$. Let $M_n = M_{[0,2n]}$ and let $T_n := (T\cap M_n) \cup \{e_0\}$.
Then the conditional law of the decorated map $(M_n , e_0 , T_n)$ is uniform on the set of triples consisting of a planar map with $n- \ell$ interior edges and a simple boundary cycle of length $\ell$, together with a boundary root edge oriented counterclockwise along the boundary, and a spanning tree with wired boundary conditions (i.e., the boundary edges are counted as part of the tree). 
\end{itemize}
The case of finite maps without boundary is described, e.g., in~\cite[Section 4]{shef-burger}. The case of finite maps with boundary may be deduced from the boundary-free case, since we can compare a map with boundary to the map without boundary that we get by including the external face in the map.

\subsection{DLA and loop-erased random walk on the UITM}
\label{sec-lerw-dla}

We now describe the relationship between DLA and LERW on the UITM.  To do so, we first check that the uniform spanning tree on the UITM is a.s.\ one-ended, which will then allow us to apply Proposition~\ref{prop-graph-conditions2} to define harmonic measure from infinity and characterize it in terms of a uniform spanning tree on the UITM.

\begin{lem}
\label{lem-uitwm-conditions}
Let $(M,e_0)$ be the UITM.  Then the uniform spanning tree on $M$ is a.s.\ one-ended.  
\end{lem}

\begin{proof}
To see that $T$ is a.s.\ one-ended, we apply the (infinite-volume) Mullin bijection.  On the one hand, the curve $\lambda$ does not cross the tree $T$; on the other hand, $\lambda$ hits all of the triangles of $\mcl T$.  This is possible only if $T$ is one-ended.
\end{proof}

By Lemma~\ref{lem-uitwm-conditions} and Proposition~\ref{prop-graph-conditions2}, we can describe the laws of LERW and DLA on the UITM by sampling uniform spanning trees on the UITM.  We will apply these characterizations of LERW and DLA on the UITM to prove a rigorous link between the two.  
We first prove the variant of this link for finite random planar maps.

\begin{lem} 
Let $(M,e_0)$ be a spanning-tree weighted random planar map with $n$ edges and let $w$ be a uniformly chosen vertex of $M$. {Fix a nonnegative integer $m < n$.} Define the (unrooted) random planar maps $M^{(m)}$ and $\wt M^{(m)}$ as follows: 
\begin{itemize}
\item
Sample a LERW from $v_0$ to $w$ in $M$.  On the {(positive probability)} event $E^{(m)}$ that this LERW has length $>m$, let $M^{(m)}$ be the random planar map obtained by cutting along  the first $m$ edges of the LERW from $v_0$ to $w$.  (We illustrate this cutting procedure in Figure~\ref{fig-bijection}.) Also, let $w^{(m)}$ be the vertex in $M^{(m)}$ corresponding to $w$. 
\item
Sample a DLA in $M$ with initial vertex $v_0$ and targeted at $w$.  On the {(positive probability)} event $\wt E^{(m)}$ that the DLA contains $>m$ edges, let $\wt M^{(m)}$ be the random planar map obtained by cutting along all of the edges of the time $m$ DLA cluster.  Also, let $\wt w^{(m)}$ be the vertex in $M^{(m)}$ corresponding to $w$.
\end{itemize}

Then the conditional joint law of  $M^{(m)}$ and $w^{(m)}$ given $E^{(m)}$ is equal to the conditional joint law of  $\wt M^{(m)}$ and  $\wt w^{(m)}$ given $\wt E^{(m)}$.  Both of these laws are that of a spanning-tree-weighted random planar map with $n+m$ total edges and a simple boundary cycle of length $2m$, decorated by a uniform interior vertex.
\label{lem-bijection-cutting}
\end{lem}

We emphasize that Lemma~\ref{lem-bijection-cutting} does \emph{not} describe the conditional law of $\wt M^{(m)}$ given the DLA cluster, only the conditional law given $\wt E^{(m)}$.

\begin{proof}[Proof of Lemma~\ref{lem-bijection-cutting}] 
\noindent\textit{Step 1: describing the law of $(M^{(m)} ,w^{(m)})$.}
By Proposition~\ref{prop-graph-conditions}, we can couple the LERW, and therefore the event $E^{(m)}$ and the pair $(M^{(m)}, w^{(m)})$, to a uniform spanning tree $T$ on $M$.  We can take the first $m$ edges of the LERW to be the first $m$ edges in the path from $v_0$ to $w$ in $T$.  Now, since $(M,e_0)$ is a UITM and $w$ is a uniform vertex of $M$, the joint law of the 4-tuple $(M,T , e_0 , w)$ is \emph{uniform} on all $4$-tuples $(M,T , e_0 , w)$ satisfying the following conditions.
\begin{lst}
 \label{4-tuple-1}
\begin{itemize}
\item $M$ is a planar map without boundary having $n$ total edges, 
\item $T$ is a spanning tree of $M$, 
\item $e_0$ is an edge in $M$ (whose terminal vertex we denote by $v_0$),  and
\item $w$ is a vertex of $M$ at graph distance $>m$ in the tree $T$ from $v_0$.
\end{itemize}
\end{lst}
Given a 4-tuple $(M,T , e_0 , w)$ satisfying the conditions~\eqref{4-tuple-1}, {if we \textbf{fix} a nonnegative integer $k \leq m$}, we can cut along the first $k$ edges of the path from $v_0$ to $w$ in the tree $T$.  We can describe the image of $(M,T , e_0 , w)$ under this  ``cutting procedure'': we get a 4-tuple $(M',T', e_0' , w')$ satisfying the following conditions.
\begin{lst}
 \label{4-tuple-2}
\begin{itemize}
\item
$M'$ is a planar map with $n-k$ interior edges and $2k$ boundary edges,
\item
$T'$ is a spanning tree of $M'$ with wired boundary conditions, 
\item 
$w'$ is a vertex in $M'$ at positive distance from the boundary of $M'$ w.r.t.\ the graph distance in the tree $T'$, and 
\item
$e_0'$ is an edge in $M'$ whose terminal vertex $v_0' \in \bdy M'$ is described as follows. If $p$ is the point of $\bdy M'$ which is hit by the branch of $T'$ started from $w'$, then $v_0'$ is the ``antipodal" point of $p$ on $\bdy M'$, i.e., the point on $\bdy M'$ which is separated from $p$ by two arcs of $\bdy M'$ of length $k$.  
\end{itemize}
\end{lst}
Conversely, we can define a ``gluing procedure'' that takes a 4-tuple $(M',T' , e_0', w')$ satisfying the conditions~\eqref{4-tuple-2} and produces a 4-tuple $(M,T , e_0 , w)$ satisfying the conditions~\eqref{4-tuple-1}.  Given a 4-tuple $(M',T' , e_0', w')$ satisfying~\eqref{4-tuple-2}, we  identify pairs of vertices along the boundary of $M'$ which lie at equal distance (along $\bdy M'$) from $v_0'$, and we identify the corresponding pairs of boundary edges of $M'$.  After making these identifications, we obtain a 4-tuple $(M,T , e_0 , w)$ satisfying the conditions~\eqref{4-tuple-1}. 

Since the cutting and gluing procedures are inverses of each other, we have defined a bijection between 4-tuples $(M,T , e_0 , w)$ satisfying the conditions~\eqref{4-tuple-1} and 4-tuples $(M',T' , e_0', w')$ satisfying the conditions~\eqref{4-tuple-2}.  See Figure~\ref{fig-bijection} for an illustration of this bijection.

\begin{figure}[ht!] \centering
\begin{tabular}{ccc} 
\includegraphics[width=0.3\textwidth]{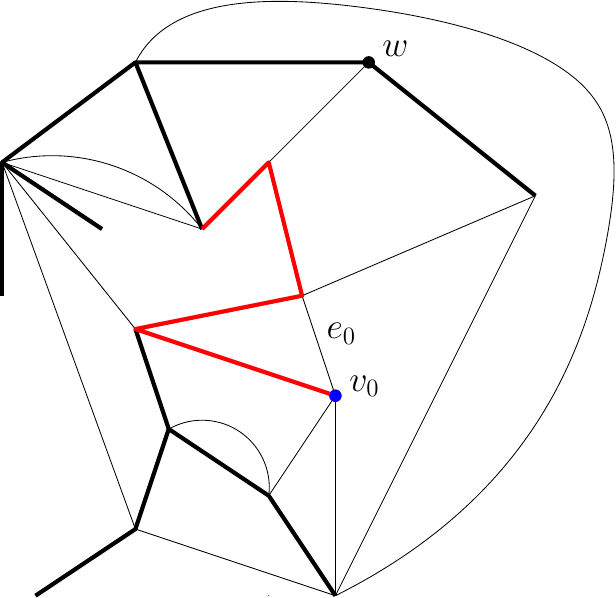}
&
\includegraphics[width=0.3\textwidth]{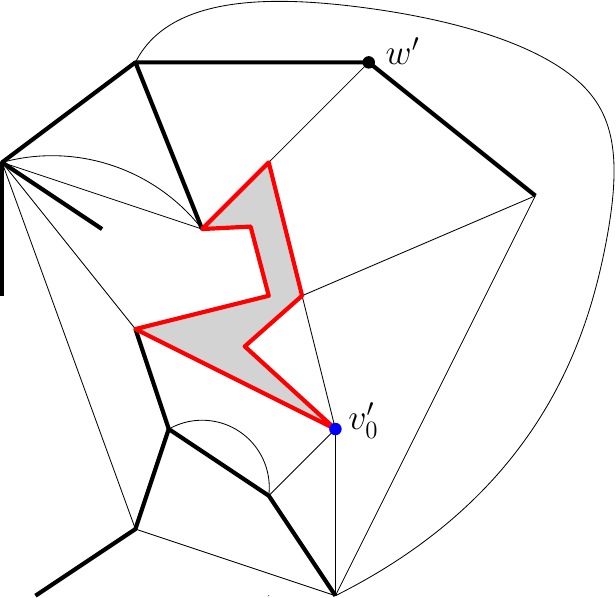}
&
\includegraphics[width=0.3\textwidth]{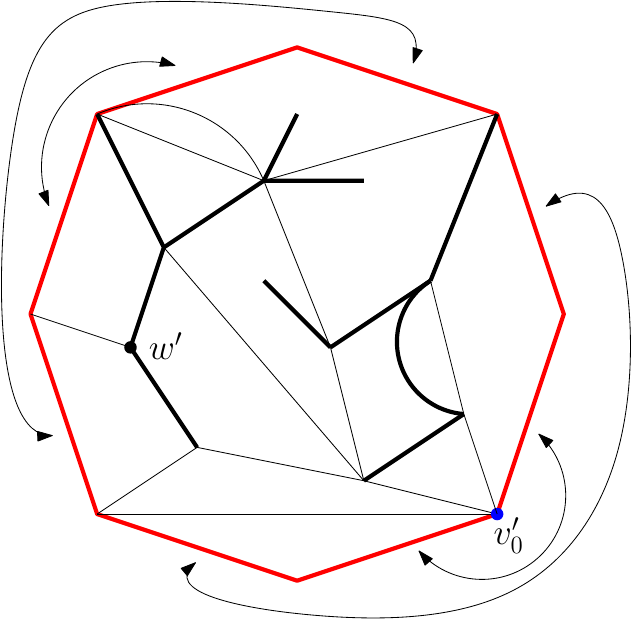}
\end{tabular}
\caption{An illustration of the cutting-gluing bijection in the proof of Lemma~\ref{lem-bijection-cutting} with $n=26$ and $k = 4$. \textbf{Left:} A 4-tuple $(M,T , e_0 , w)$ satisfying the conditions~\eqref{4-tuple-1}, with the tree $T$ in bold black, the root vertex $v_0$ in blue.  We have colored in red the first $k$ steps of the LERW from $v_0$ to $w$ coupled to $T$.  \textbf{Center:} Cutting along those $k$ edges yields a 4-tuple $(M',T' , e_0', w')$ satisfying the conditions~\eqref{4-tuple-2}.  \textbf{Right:} A nicer embedding of this 4-tuple $(M',T' , e_0', w')$.   We can reverse the cutting operation and recover $(M,T , e_0 , w)$ by pasting together boundary edges as indicated.}
\label{fig-bijection}
\end{figure}

The above bijection implies that if $(M,T,e_0,w)$ is sampled uniformly from the set of 4-tuples satisfying~\eqref{4-tuple-1}, then the corresponding $(M',T',e_0',w')$ is sampled uniformly from the set of 4-tuples satisfying~\eqref{4-tuple-2}. From this {and} the definition of $(M^{(k)} , w^{(k)})$ in the lemma statement, we obtain that the conditional joint law of the   pair $(M^{(k)}, w^{(k)})$ conditioned on the event $E^{(k)}$ is exactly the marginal joint law of $(M',w')$ for $(M',T' , e_0', w')$ a uniform 4-tuple satisfying the conditions~\eqref{4-tuple-2}.  Setting $k = m$, this proves that the pair  $(M^{(m)}, w^{(m)})$ has the desired conditional joint law given $E^{(m)}$.    \medskip

\noindent\textit{Step 2: comparing LERW and DLA.}
It remains to show that the pair $(\wt M^{(m)},  \wt w^{(m)})$ has the same conditional joint law given $\wt E^{(m)}$.  The idea here is that we can describe the DLA growth process as a modified version of LERW in which we ``reshuffle'' the tip of the path at each integer time, so that instead of sampling according to harmonic measure at the tip of the path, we sample from harmonic measure on the entire cluster.  We can think of this ``reshuffling'', roughly speaking, as a resampling of the uniform spanning tree that we use to build the LERW.  This means that if, at some step of constructing the LERW, we cut along the path as above and ``forget'' both the tip and the tree, the law of the next edge that we add is the same as if we were constructing a DLA cluster.

To make this idea precise, we first observe that we can strengthen our last result to get, for each $k$ and conditioned on the event $E^{(k)}$, a coupling of 
\begin{itemize}
\item the pair $(M^{(k)},  w^{(k)})$, 
\item the {indicator random variable associated to the} event $E^{(k+1)}$, and 
\item the pair $(M^{(k+1)},  w^{(k+1)})$. 
\end{itemize} 
In this coupling, we identify the pair $(M^{(k)},  w^{(k)})$ with  $(M',w')$ for $(M',T' , e_0', w')$ a uniform 4-tuple satisfying the conditions~\eqref{4-tuple-2}.  {Recalling the definition of the point $p$ in conditions~\eqref{4-tuple-2}, we identify $\mathbf{1}_{E^{(k+1)}}$ as the indicator of the event that $p$ and $w'$ are not adjacent in $T'$.} And we identify the pair $(M^{(k+1)},  w^{(k+1)})$ as the image of $(M',w')$ under the operation of cutting along {the first edge in the path in $T'$ from $p$ to $w'$}.   See Figure~\ref{fig-jointlaw} for an illustration of this coupling.

\begin{figure}[t!]
    \centering
    \subfigure[]{\label{fig:c}\includegraphics[width=0.3\textwidth]{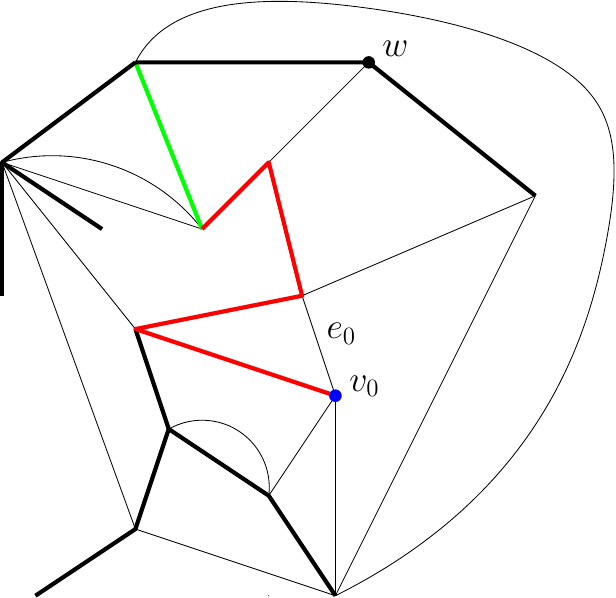}}
    \quad
    \subfigure[]{\label{fig:d}\includegraphics[width=0.3\textwidth]{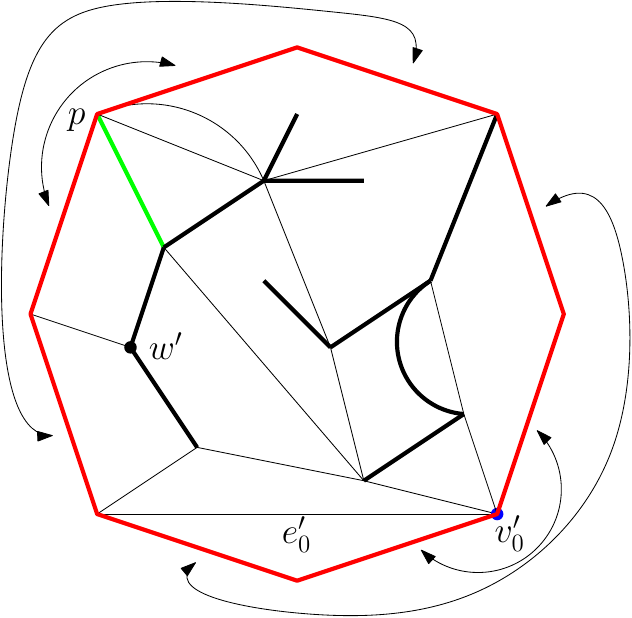}}
\caption{ {An illustration of the coupling described in Step 2 of the proof of Lemma~\ref{lem-bijection-cutting}.  Setting $n=26$ and $k = 4$, we start in (a) with a spanning-tree-weighted map $(M,e_0)$ decorated by a uniform vertex $w$ and uniform spanning tree $T$, and we condition on the event $E^{(k)}$ that the LERW from $v_0$ to $w$ has more than $k$ edges.  We have colored the first $k$ edges of the LERW in red, and the $(k+1)$-st edge in green.   Based on the cutting procedure we illustrated in Figure~\ref{fig-bijection}, we can produce the conditional joint law of $(M^{(k+1)},  w^{(k+1)})$ given $E^{(k+1)}$ by conditioning on $E^{(k+1)}$ and cutting along the first $k+1$ edges of the LERW (the edges colored in red and green).  Alternatively, we can divide this cutting procedure into two stages.  In the first stage, we cut along just the first $k$ edges (colored in red) to produce the map in (b).  The law of this map is the conditional joint law of $(M^{(k)},  w^{(k)})$ given $E^{(k)}$.  In the second stage, we condition on the event $E^{(k+1)}$ and cut along the edge in (b) colored in green, which corresponds to the green edge in (a).  The correspondence between (a) and (b) yields a coupling, on the event $E^{(k)}$, of  $(M^{(k)},  w^{(k)})$ with $\mathbf{1}_{E^{(k+1)}}$ and $(M^{(k+1)},  w^{(k+1)})$.  Under this coupling, $\mathbf{1}_{E^{(k+1)}}$ corresponds to the indicator of the event that the green edge in (b) does not contain the vertex $w'$, and $(M^{(k+1)},  w^{(k+1)})$ corresponds to the map obtained from the map $(M^{(k)},  w^{(k)})$ in (b) by cutting along the green edge.}}
\label{fig-jointlaw}
\end{figure}

We use this coupling to prove the desired result by finite induction on $k$.  The case $k=0$ is trivial, so assume that for some $k$, the conditional law of the pair $(\wt M^{(k)},  \wt w^{(k)})$ given  $\wt E^{(k)}$ is also the marginal joint law of $(M',w')$ for $(M',T' , e_0', w')$ a uniform 4-tuple satisfying the conditions~\eqref{4-tuple-2}.  Remember that we sample the $(k+1)$-th edge of the DLA cluster according to harmonic measure from $w$ on the boundary of the cluster.  By the cutting-gluing bijection, this means that we can couple both this $(k+1)$-th edge and the pair $(\wt M^{(k)},  \wt w^{(k)})$ to the uniform 4-tuple  $(M',T' , e_0', w')$, by taking the $(k+1)$-th edge to be the {first edge of the path in $T'$ from $p$ to $w'$}.  (Here we have applied Wilson's algorithm to say that, if we ``forget'' the tree $T'$ and the edge $e_0'$, the law of this first edge given the triple $(M',w')$ is harmonic measure from $w'$ on the boundary of $M'$.) Hence, we have the same coupling that we had in the LERW case for the pair $(\wt M^{(k)},  \wt w^{(k)})$, the pair $(\wt M^{(k+1)},  \wt w^{(k+1)})$, and the event $\wt E^{(k+1)}$,  conditional on $\wt E^{(k)}$.  This proves the inductive step.
\end{proof}

By taking the spanning-tree-weighted random planar map in  Lemma~\ref{lem-bijection-cutting} to be arbitrarily large and using the fact that the UITM is a Benjamini-Schramm limit of finite spanning-tree-weighted random planar maps, we get a relationship between DLA and LERW on the UITM.

\begin{lem}
Let $(M,e_0)$ be a UITM.  The law of the  planar map obtained by cutting along  the first $m$ edges of an external DLA growth process started at $v_0$ and targeted at infinity is equal to the law of the planar map obtained by cutting along the first $m$ edges of a loop-erased random walk from $v_0$ to infinity.
\label{complementinfinite}
\end{lem}

Lemma~\ref{complementinfinite} follows directly from the following assertion, which roughly states that we can approximate harmonic measure from infinity on the DLA and LERW by harmonic measure on these clusters viewed from a uniform vertex in a sufficiently large finite spanning-tree-weighted random planar map. 

\begin{lem}
Let $(M,e_0)$ be the UITM, and let $(M^n,e_0^n)$ be a sequence of edge-rooted finite maps that converge a.s.\ to $(M,e_0)$ in the Benjamini-Schramm sense.  Also, fix $m \in \BB N$ and $\ep >0$.  Then, for all $n$ sufficiently large, we have the following. For $n\in\BB N$, let $w^n$ be sampled uniformly from the vertex set of $M^n$. For every connected set $A$ of $m$ edges in $M$ that includes the root vertex, harmonic measure on $A$ from infinity in $M$ is $<\ep$-apart in total variation from harmonic measure on $A$ from $w^n$.  
\label{lem-coupling-benjamini-schramm}
\end{lem}

Note that, in the statement of Lemma~\ref{lem-coupling-benjamini-schramm}, we use $A$ to denote both an edge set in $M$ and the corresponding edge set in $M^n$.  This abuse of notation makes sense since, for large enough $n$, we have the isomorphism $B_m^M(v_0) \equiv B_m^{M^n}(v_0^n)$ due to the Benjamini-Schramm convergence.  


\begin{proof}[Proof of Lemma~\ref{lem-coupling-benjamini-schramm}]
From Proposition~\ref{prop-graph-conditions2} and Lemma~\ref{lem-uitwm-conditions}, we deduce that with probability at least $1-\ep$, the UITM $(M,e_0)$ satisfies the following two properties, for every choice of $A$ and some pair of integers $N < N'$ chosen to be sufficiently large:
 \begin{enumerate}
 \item
 \label{enum-property1}
 Harmonic measure on $A$ from infinity is well-approximated by harmonic measure on $A$ from \emph{any} vertex $w$ in $M$ at distance $ \geq N$ from the root.  Specifically, the two measures are a.s. $<\ep$-apart in total variation for all such $w$.
 \item
  \label{enum-property2}
 A random walk started in $B_N^M(v_0)$ remains in $B_{N'}^M(v_0)$ before hitting $v_0$ with probability at least $1 - \ep$. (Here we use the recurrence of the random walk on the UITM.)
 \end{enumerate}
 
Thus, it suffices to prove the conditions of the lemma for the UITM $(M,e_0)$ on the event that it satisfies these two properties.  For the rest of the proof, we fix $(M,e_0)$ sampled from this conditional law, and let $(M^n,e_0^n)$ be the corresponding sequence of approximating maps. We also fix the choice of cluster $A$.

Since $(M^n , e_0^n) \rta (M , e_0)$ in the Benjamini-Schramm sense, we can choose $n$ large enough so that $B_{N'}^M(v_0) $ and $ B_{N'}^{M^n}(v_0^n)$ are isomorphic as graphs.
We henceforth fix an isomorphism and use it to identify $B_{N'}^M(v_0) $ and $ B_{N'}^{M^n}(v_0^n)$. 
We divide the proof into two steps: 
\begin{enumerate}
\item
We describe a way of choosing a (non-uniform) random vertex $\wt w$ in $M$ such that harmonic measure on $A$ from $\wt w$ in $M$ is at most $\ep$-apart in total variation from harmonic measure on $A$ from a uniform vertex $w^n$   in $M^n$. 
\item
From our choice of vertex $\wt w$ in $M$, we show that (for large enough $n$) harmonic measure on $A$ from $\wt w$ in $M$ is at most $\ep$-apart in total variation from harmonic measure on $A$ from infinity in $M$. 
\end{enumerate}
 
We begin with the first step.  We can describe harmonic measure on $A$ from $w^n$ by sampling a random walk started at $w^n$ and run until it hits $A$.  Now, let $\wt w^n$ be the first vertex the random walk hits in the ball $B_N^{M^n}(v_0^n)$.  If $\op{dist} (w^n,v_0 ; M^n) \leq N$, then $\wt w^n = w^n$; otherwise, $w^n$ is some point on the boundary of the ball.  By the strong Markov property of random walk, the harmonic measure on $A$ as viewed from $\wt w^n$ is the same as the harmonic measure  on $A$ as viewed from $w^n$.  By property~\ref{enum-property2} above, this random walk will remain in the larger ball $B_{N'}^{M^n}(v_0^n)$ in $M^n$ with probability $> 1-\ep$. 
Let $\wt w$ be the vertex of $B_{N'}^M(v_0)$ corresponding to $\wt w^n$ under the isomorphism $B_{N'}^{M^n}(v_0^n) \equiv B_{N'}^M(v_0)$. 
Then the harmonic measure on $A$ in $M$ as viewed from $\wt w$ is at most $\ep$ apart in total variation from the harmonic measure on $A$ in $M^n$ as viewed from $\wt w^n$.
This completes the first step.

Now, for the second step, remember that $\wt w$ is a random vertex in $M$ chosen so that the distance between $\wt w$ and the root is exactly $N$  as long as $w^n$ is not itself in the radius-$N$ ball $B_{N}^{M^n}(v_0^n)$. 
By making $n$ larger if necessary, we can ensure that $w^n$ lies outside this metric ball with probability $>1-\ep$.  Thus, applying property~\ref{enum-property1} above completes the second step.
\end{proof}

\begin{proof}[Proof of Lemma~\ref{complementinfinite}]
Let $\ep > 0$, and let $N$ be a positive integer greater than $m$.  By Skorohod's representation theorem, we can couple the maps $(M^n,e_0^n)$ such that they converge a.s.\ to $(M,e_0)$.  Let $w^n$ be a uniform vertex of $M^n$. Choose $n$ large enough that the balls $B_N^M(v_0)$ and $B_N^{M^n}(v_0^n)$  are isomorphic. By Lemma~\ref{lem-coupling-benjamini-schramm} and the definitions of LERW and DLA in terms of harmonic measure, we can choose $n$ large enough that  the following two pairs of random growths in $M$ and $M^n$ agree with probability at least $1-\ep$:
\begin{itemize}
\item the first $m$ steps of a DLA process started at the vertex $v_0$ and targeted at infinity on $M$, and the first $m$ steps of DLA started at the vertex $v_0^n$ and targeted at $w^n$ on $M^n$;
\item the first $m$ steps of a LERW started at the vertex $v_0$ and targeted at infinity on $M$, and the first $m$ steps of LERW started at the vertex $v_0^n$ and targeted at $w^n$ on $M^n$.
\end{itemize}
Note that, as in the statement of Lemma~\ref{lem-coupling-benjamini-schramm}, we are comparing the clusters in $M$ and $M^n$ under the isomorphism $B_N^M(v_0) \equiv B_N^{M^n}(v_0^n)$.

By applying Lemma~\ref{lem-bijection-cutting} to the LERW and DLA on $M^n$, and then transferring the result to $M$, we deduce that the laws of the following two random planar maps are at most $\ep$-apart in total variation:
\begin{itemize}
\item
the distance-$(N-m)$-neighborhood of the boundary in the infinite planar map with boundary that we obtain by cutting along the edges of the first $m$ steps of DLA on $M$;
\item
the analogous map for the first $m$ steps of loop-erased random walk on $M$.
\end{itemize}
The result follows by taking $\ep$ arbitrarily small and $N$ arbitrarily large.
\end{proof}

\section{Proof of main results conditional on a relationship between exponents}
\label{sec-proof-conditional}

Both the Mullin bijection and the connection between LERW and DLA on the UITM in Lemma~\ref{complementinfinite} are special properties of the UITM.  The second property allows us to prove the growth exponent for DLA in this random environment by first proving the growth exponent for LERW and then transferring the result to DLA. The key step in our analysis of LERW on the UITM is a relationship between two exponents associated with distances in the UITM that we describe in Theorem~\ref{thm-chi} below. 

We first describe the two exponents that we will need to compare.  The first is the metric ball volume growth exponent $d$ (Definition~\ref{def-d}).  To define the second exponent, we let $(M, T,e_0)$ be the UITM with its distinguished spanning tree and consider the finite submap $M|_{[0,n]} \subset M$ for $n\in\BB N$, using the notation we defined in Section~\ref{sec-mullin}.
It is proven in~\cite[Theorem 1.7]{ghs-map-dist} (building on~\cite[Theorem 1.12]{ghs-dist-exponent}) that there is an exponent $\chi > 0$ such that for each $\delta \in (0,1)$, 
\eqb \label{eqn-dist-upper0}
\BB P\left[ \op{diam} \left( M|_{[0,n]} \right) \leq  n^{\chi + \delta} \right] \geq 1  - O_n(n^{-p}), \quad \forall p > 0 
\eqe 
and 
\eqb \label{eqn-dist-lower0}
\BB P\left[ \op{diam} \left( M|_{[0,n]}  \right)      \geq   n^{\chi - \delta}  \right] \geq n^{-o_n(1)} .
\eqe  
We emphasize that~\eqref{eqn-dist-upper0} and~\eqref{eqn-dist-lower0} concern the diameter of $M|_{[0,n]}$ w.r.t.\ its \emph{internal} graph distance, \emph{not} w.r.t.\ the ambient graph distance on $M$. We also note that the lower bound~\eqref{eqn-dist-lower0} does not show that the event in question holds with probability tending to 1 as $n\rta\infty$, only with probability decaying slower than any negative power of $n$. We will eventually show (see Theorem~\ref{thm-chi} just below) that the $n^{-o_n(1)}$ can be replaced by $1 - O_n(n^{-p})$ for any $p > 0$. 

The exponent $\chi$ describes the diameter of a submap of $M$ with approximately $n$ vertices, so it is natural to expect that $\chi = 1/d$ (see also~\cite[Conjecture 1.13]{ghs-dist-exponent}). The main step in the proofs of Theorems~\ref{main3},~\ref{thm-lerw}, and~\ref{thm-finite-map} is showing that this is indeed the case.

\begin{thm} \label{thm-chi}
One has $\chi = 1/d$. In fact, for each $\delta \in (0,1)$ it holds with superpolynomially high probability as $n\rta\infty$ that 
\eqb \label{eqn-chi}
n^{1/d -\delta} \leq \op{diam} \left( M|_{[0,n]} \right) \leq  n^{1/d + \delta} .
\eqe
\end{thm}

One direction of the equality $\chi = 1/d$ is obvious. 

\begin{lem}
\label{lem-obvious-direction}
The inequality $\chi \geq \frac{1}{d}$ holds.  
\end{lem}


\begin{proof}[Proof of Lemma~\ref{lem-obvious-direction}]
Let $\delta > 0$. By Definition~\ref{def-d}, a.s.\ for all $n$ sufficiently large, 
\eqb
\label{eqn-obvious-direction-step}
\frac{\log \#B_{n^{\chi + \delta}}^M(v_0)}{(\chi + \delta) \log n} \leq d+\delta
\eqe
Since $\#B_r^M(v_0) \geq n$ for $r = \op{diam} \left( M|_{[0,n]} \right)$,~\eqref{eqn-dist-upper0} implies that the left-hand side of~\eqref{eqn-obvious-direction-step} is $\geq \frac{1}{\chi + \delta}$ with probability at least $1  - O_n(n^{-p})$ for all $p > 0$.  Taking $\delta \rta 0$ proves the lemma.
\end{proof}

In the rest of this section, we will assume Theorem~\ref{thm-chi} and deduce Theorems~\ref{main3},~\ref{thm-lerw}, and~\ref{thm-finite-map}. The remaining sections of the paper will be devoted to proving Theorem~\ref{thm-chi}, using SLE/LQG techniques. 

We first prove Theorems~\ref{main3} and~\ref{thm-lerw}.  Both these theorems are consequences of the following result about LERW on the UITM. 

\begin{prop} \label{prop-lerw-complement}
Let $(M,e_0)$ be the UITM and for $m\in\BB N$, let $M^{(m)}$ be the infinite random planar map with boundary obtained by cutting along the first $m$ steps of a LERW on $M$ from $v_0$ to $\infty$. 
For each $\delta >0$, it holds with superpolynomially high probability as $m\rta \infty$ that
\eqb \label{eqn-diam-upper}
\op{diam} \left( \bdy M^{(m)} ; M^{(m)} \right) \leq m^{2/d + \delta}   . 
\eqe
Furthermore, if $\alpha \in (0,2)$ and $\delta \in (0,\alpha/100)$, then with with superpolynomially high probability as $m\rta\infty$, there are at least $m^{1 + \alpha/2 -\delta}$ vertices at graph distance $\leq m^{\alpha/d +\delta}$ from $\bdy M^{(m)}$ in $M^{(m)}$.
\end{prop}
\begin{proof}
Let $T$ be a uniform spanning tree on $M$ (i.e., the conditional law of $T$ given $M$ is that of the UST on $M$) and let $\mcl Z  = (\mcl L , \mcl R ) : \BB Z\rta \BB Z^2$ be the associated encoding walk for $(M, T , e_0)$ under the Mullin bijection. 
Also let $m\in\BB N$ and let $\op{LERW}_m$ denote the subgraph of $M$ consisting of the first $m$ edges of the branch of $T$ from $v_0$ to $\infty$ together with their endpoints, so that $\op{LERW}_m$ has the law of the first $m$ steps of a LERW from $v_0$ to $\infty$. 
Let $t_1,\ldots,t_m$ be the positive times  which correspond under the Mullin bijection to triangles of $Q\cup Q_* \cup T$ which have an edge in $\op{LERW}_m$. 
Equivalently, $t_j$ for $j=1,\dots,m$ is the smallest $t\geq 0$ for which $\mcl L_t = -j$.  

We first prove~\eqref{eqn-diam-upper}.
The key idea is to consider a collection of time intervals of length $m^{2}$  such that the submaps of $M$ encoded by  $\mcl Z$ in these time intervals together cover $\op{LERW}_m$.  We then apply Theorem~\ref{thm-chi} to each such interval. The collection of intervals is given by
\[
[\tau_k,\tau_k+m^{2}] , \qquad k \in [0,K-1]_{\BB Z}
\]
where 
\begin{itemize}
\item the sequence of stopping times $\tau_0,\tau_1,\ldots$ is defined inductively by setting $\tau_0 = 0$ and setting $\tau_{k+1}$ equal to the first time at or after $\tau_k + m^{2}$ at which $\mcl L$ achieves a running minimum; and
\item $K$ is equal to the smallest positive integer satisfying $\tau_K > t_m$.
\end{itemize}

We claim that the number of intervals $K$ is stochastically dominated by a geometric random variable with mean bounded above by a universal constant. To see why this is true, observe that 
\eqbn
\BB P\left[ \tau_k > t_m  |  \tau_{k-1} < t_m \right] 
\geq \BB P\left[ \min_{t \in [\tau_{k-1} +1  , \tau_{k-1} + m^2]_{\BB Z} } \mcl L_t \leq -m \right] 
\geq \BB P\left[ \min_{t  \in [1,m^2]_{\BB Z} } \mcl L_t \leq - m \right] 
\eqen
which is bounded below by a universal constant by the convergence of simple random walk to Brownian motion. In particular, for each $\delta >0$ it holds with superpolynomially high probability as $m\rta \infty$ that $K \leq m^{\delta/2}$.   

By Theorem~\ref{thm-chi} (applied with $n = m^2$) and a union bound, it holds with superpolynomially high probability as $m\rta\infty$ that the submaps $ M|_{[\tau_k,\tau_k+m^{2}]}$ all have internal diameter at most $m^{2/d + \delta/2}$.  Since the union of these submaps contains exactly half of the edges of $\bdy M^{(m)}$ (namely, one of the two edges corresponding to each edge of $\text{LERW}_m$), and since one can run the same argument to deal with the other half of $\bdy M^{(m)}$, the triangle inequality implies~\eqref{eqn-diam-upper}. 
 
We now prove the last part of the proposition. We consider a collection of time intervals of length $m^{\alpha}$ such that the corresponding submaps of $M$ intersect only along their boundaries and each intersects $\op{LERW}_m$. The collection is given by
\begin{equation}
[\tau_k,\tau_k+m^{\alpha}] , \qquad k \in [0,K-1]_{\BB Z}
\label{collection2}
\end{equation}
where 
\begin{itemize}
\item the sequence of stopping times $\tau_0,\tau_1,\ldots$ is defined inductively by setting $\tau_0 = 0$ and setting $\tau_{k+1}$ equal to the first time $\geq \tau_k + m^{\alpha}$ at which $\mcl L$ achieves a running minimum; and
\item 
 $K$ is the smallest positive integer satisfying $\tau_K > t_m$.
\end{itemize}

We have $\mcl L_{\tau_K} \leq -m$ (by the definition of $t_m$), so if $K < m^{1-(\alpha +\delta)/2}$, then there must be a $k\in [1,m^{1-(\alpha+\delta)/2}]_{\BB Z}$ such that $\min_{t \in [\tau_k   ,  \tau_k + m^\alpha]_{\BB Z} } (\mcl L_t - \mcl L_{\tau_k}) \leq  -m^{(\alpha + \delta)/2}$.
Hence  
\alb
&\BB P\left[ K < m^{1 - (\alpha + \delta)/2} \right]  \\
&\qquad \leq \BB P\left[ \exists k \in [0, m^{(\alpha + \delta)/2} ]_{\BB Z}  \:\: \text{s.t.} \:\: \min_{t \in [\tau_k   ,  \tau_k + m^\alpha]_{\BB Z} } (\mcl L_t - \mcl L_{\tau_k}) \leq  -m^{(\alpha + \delta)/2}  \right] 
\ale
which decays faster than any negative power of $m$ by a standard estimate for simple random walk and a union bound. 
Hence the collection of intervals~\eqref{collection2} contains at least $m^{1 - (\alpha + \delta)/2}$ intervals with superpolynomially high probability as $m \rightarrow \infty$. 
 
By Theorem~\ref{thm-chi} (and since trivially $K \leq m$), it holds with superpolynomially high probability as $m\rta\infty$ that the internal diameter of each $M|_{[\tau_k , \tau_k+m^\alpha]}$ for $k = 0,\dots , K-1$ is at most $m^{\alpha/d + \delta}$. On the event that these diameter upper bounds hold, the submaps' vertices are all at distance $\leq m^{\alpha/d + \delta}$ from $\op{LERW}_m$.  Since the sets of vertices in $M|_{[\tau_k , \tau_k+m^\alpha]} \setminus \bdy  M|_{[\tau_k , \tau_k+m^\alpha]}$ are all disjoint, we deduce that, with superpolynomially high probability as $m\rta\infty$, the number of vertices at distance $\leq m^{\alpha/d +\delta}$ from $\bdy M^{(m)}$ in $M^{(m)}$ is at least 
\eqb \label{eqn-vertex-sum}
\sum_{k=0}^{\lfloor m^{1- (\alpha + \delta) /2} \rfloor } \#\mcl V\left( M|_{[\tau_k , \tau_k+m^\alpha]} \setminus \bdy  M|_{[\tau_k , \tau_k+m^\alpha]} \right) .
\eqe 
We will now argue that with superpolynomially high probability as $m\rta\infty$, we have 
\eqb \label{eqn-vertex-count}
\#\mcl V\left( M|_{[\tau_k , \tau_k+m^\alpha]} \setminus \bdy  M|_{[\tau_k , \tau_k+m^\alpha]} \right) \geq m^{\alpha -\delta/2}, \quad\forall k \in [0,K-1]_{\BB Z} .
\eqe 
To show~\eqref{eqn-vertex-count}, we prove two types of estimates: a probabilistic lower bound for the number of vertices in each $M|_{[\tau_k , \tau_k+m^\alpha]}$, and a probabilistic upper bound for the number of vertices in each of the boundaries $\bdy M|_{[\tau_k , \tau_k+m^\alpha]}$.
\begin{itemize}
\item
Recall the space-filling curve $\lambda$ from the Mullin bijection. Whenever $\mcl L$ increases, $\lambda$ traces a triangle of $Q\cup T\cup T_*$ which includes a vertex of $M$ which is not part of any of the previously traced triangles. Therefore, we can bound the number of vertices in $M|_{[\tau_k , \tau_k+m^\alpha]}$ from below by the number of times $\mcl L$ increases in the corresponding time interval.   Since $\mcl Z$ is a simple random walk on $\BB Z^2$, the events that $\mcl L$ increases at each of the times in $[\tau_k , \tau_k+m^\alpha]$ are independent with probability $1/4$.  Thus, by applying Hoeffding's equality applied to $\lfloor m^\alpha \rfloor$ i.i.d.\ Bernoulli random variables with parameter $1/4$, we find that with superpolynomially high probability as $m\rta\infty$, each  $M|_{[\tau_k , \tau_k+m^\alpha]}$ has at least $m^{\alpha - \delta/2}$ vertices.
\item
 On the other hand, we can bound from above the number of vertices $\bdy M|_{[\tau_k , \tau_k+m^\alpha]}$. 
To do so, we examine what it means for $v$ to be on the boundary of the submap $M|_{[\tau_k , \tau_k+m^\alpha]}$.  We have $v \in \bdy M|_{[\tau_k , \tau_k+m^\alpha]}$ if $v$ is part of a triangle that the curve $\lambda$ traces in the time interval $[\tau_k , \tau_k+m^\alpha]$, \emph{and} part of a triangle traced at a time \emph{not} in $[\tau_k , \tau_k+m^\alpha]$. Equivalently, $v$ is on the (unique) path in $T$ consisting of edges traced exactly once in the interval $[\tau_k , \tau_k+m^\alpha]$.  The length of this path is exactly the displacement of the walk $\mcl L$ in the time interval $[\tau_k , \tau_k+m^\alpha]$.  By a basic tail estimate for simple random walk, with superpolynomially high probability as $m\rta\infty$, each $\bdy M|_{[\tau_k , \tau_k+m^\alpha]}$ has at most $m^{\alpha/2 + \delta/2}$ vertices. 
\end{itemize}
Combining these two bounds yields~\eqref{eqn-vertex-count}. Plugging~\eqref{eqn-vertex-count} into~\eqref{eqn-vertex-sum} concludes the proof. 
\end{proof}

We are now ready to prove Theorems~\ref{main3} and~\ref{thm-lerw}.
\begin{proof}[Proof of Theorems~\ref{main3} and~\ref{thm-lerw} assuming Theorem~\ref{thm-chi}]
We will prove Theorem~\ref{thm-lerw}, which gives the growth exponent for LERW; Theorem~\ref{main3} then follows by the connection we established between DLA and LERW on the UITM in Lemma~\ref{complementinfinite}.
For $m\in\BB N$, let $M^{(m)}$ be the infinite planar map with boundary of length $2m$ obtained by cutting along the edges of the LERW $\op{LERW}_m$.  By Proposition~\ref{prop-lerw-complement}, 
 for each fixed $\delta >0$, it holds with superpolynomially high probability as $m\rta \infty$ that $\op{diam} (\op{LERW}_m ; M) \leq m^{2/d+\delta}$. By the Borel-Cantelli lemma, this implies the upper bound in Theorem~\ref{thm-lerw}. 

By the last statement of Proposition~\ref{prop-lerw-complement} and the same argument as above, for each $\alpha \in (0,2)$ and $\delta \in (0,\alpha/100)$, it holds with superpolynomially high probability as $m\rta\infty$ that the set of vertices at $M$-graph-distance at most $m^{\alpha/d +\delta}$  from $\op{LERW}_m$ contains at least $m^{1+\alpha/2 - \delta}$ vertices. By the Borel-Cantelli lemma this is a.s.\ the case for large enough $m\in\BB N$. 

Let $r_m$ be the smallest $r > 0$ for which $\op{LERW}_m\subset B_r^M(v_0)$. Note that $r_m \rta\infty$ as $m\rta\infty$. 
By Definition~\ref{def-d}, a.s.\ $\# B^M_r(v_0) \leq r^{d + o_r(1)}$. Furthermore, the $m^{\alpha+\delta}$-neighborhood of $\op{LERW}_m$ is contained in $B_{r_m + m^{\alpha/d+\delta}}^M(v_0)$, so a.s.\ for large enough $m$, 
\eqb
m^{1+\alpha/2-\delta} 
\leq \# B_{r_m + m^{\alpha/d+\delta}}^M(v_0)
\leq \left( r_m + m^{\alpha/d+\delta}   \right)^{d + o_m(1)} .
\eqe
By re-arranging and sending $\delta \rta 0$, we get that a.s.\ 
\eqbn
r_m \geq m^{1/d + \alpha/(2d)   - o_m(1)} - m^{\alpha/d + o_m(1) } \geq m^{1/d + \alpha/(2d)   - o_m(1)}
\eqen
where in the last line we use that $1/d + \alpha/(2d)  > \alpha/d$ due to the fact that $\alpha < 2$. Sending $\alpha \rta 2$ now gives the lower bound in Theorem~\ref{thm-lerw}. 
\end{proof}

Finally, we prove our result for the diameters of finite random planar maps.

\begin{proof}[Proof of Theorem~\ref{thm-finite-map} assuming Theorem~\ref{thm-chi}]
Let $\mcl Z = (\mcl L , \mcl R)$ be the simple random walk on $\BB Z^2$ which encodes $(M, T , e_0)$ via the Mullin bijection for finite spanning-tree-weighted random planar maps possibly with boundary. 
If $M_n$ and $\ell$ are as in the theorem statement (we take $\ell = 0$ if $M_n$ has no boundary), then, as we described in Section~\ref{sec-mullin}, the Mullin bijection shows that the law of $M_n$ is the same as the law of $M|_{[0,2n]}$ conditioned on the event that $\mcl L_t , \mcl R_t \geq 0$ for each $t = 0,\dots, 2n$ and $\mcl Z_n = (\ell , 0)$. 
By applying a local central limit theorem to $Z_n/\sqrt{n}$, we see that probability of the event $\mcl Z_n = (\ell , 0)$ \emph{without} conditioning on $\mcl L_t , \mcl R_t \geq 0$  is at least $n^{-1/2}$ times some constant that is uniform in $\ell \in [0,n^{1/2}]_{\BB Z}$.  
By the reflection principle for simple random walk, the probability that both $\mcl Z_n = (\ell , 0)$ and $\mcl L_t , \mcl R_t \geq 0$ for each $t = 0,\dots, 2n$  is bounded from below, uniformly in $\ell \in [0,n^{1/2}]_{\BB Z}$, by $C n^{-3/2}$ for some universal constant $C$. 
Combining this with Theorem~\ref{thm-chi} gives the desired result. 
\end{proof}

\section{Liouville quantum gravity and the one-sided mated-CRT map}
\label{LQGreview}

The rest of the paper is devoted to the proof of Theorem~\ref{thm-chi}.   To prove the theorem, we will analyze another random planar map called the one-sided mated-CRT map which is directly connected to LQG and SLE.  The exponents $\chi$ and $d$ also describe distances in this map, so we will prove Theorem~\ref{thm-chi} by proving appropriate upper and lower bounds on these distances.  This section is devoted to defining the one-sided mated-CRT map, and to reviewing the  LQG and SLE theory needed to formulate this definition.

First, we will define LQG surfaces in general in Section~\ref{sec-lqg-def}; then, in Section~\ref{sec-wedge-def} we will define the LQG surface that we will use in the definition of the one-sided mated-CRT map.  In Section~\ref{sec-sle-def}, we recall the definitions of two types of SLE processes and their encodings in terms of Gaussian free fields given by the ``imaginary geometry'' machinery from~\cite{ig4}. For the reader unfamiliar with imaginary geometry, we will introduce the minimum amount of tools from this theory that we need to describe this encoding and prove our desired results. Finally,  we define the one-sided mated-CRT map in Section~\ref{sec-mated-crt}, and we develop tools for studying distances in this map in Section~\ref{sec-metric-def}.

\subsection{Liouville quantum gravity (LQG) surfaces}
 \label{sec-lqg-def}

Let $\gamma \in (0,2)$, let $\mcl D \subset \BB C$, and let $h$ be a variant of the Gaussian free field (GFF) on $\mcl D$ (see~\cite{shef-gff,ss-contour,ig1,ig4} for more on the GFF). 
Heuristically speaking, a $\gamma$-LQG surface is the random surface parametrized by $\mcl D$ with Riemannian metric tensor $e^{\gamma h} \, (dx^2 + dy^2)$, where $dx^2  + dy^2$ is the Euclidean Riemannian metric tensor on $\mcl D$. Such surfaces arise as the scaling limits of random planar maps in various topologies.
In particular, the spanning-tree-weighted random planar map corresponds to $\gamma = \sqrt 2$, so, for simplicity, we will restrict ourselves to this value of $\gamma$ in our exposition of LQG in this paper.

The above definition of LQG surfaces does not make literal sense since $h$ is a distribution, not a function, so cannot be exponentiated. 
However, one can give a rigorous definition of LQG surfaces following~\cite{shef-kpz,shef-zipper,wedges}, as follows.

\begin{defn} \label{def-lqg-surface}
A \emph{$\sqrt{2}$-LQG surface} is an equivalence class of pairs $(\mcl D,h)$, where $\mcl D \subset \BB{C}$ is an open domain and $h$ is a distribution on $\mcl D$ (typically some variant of the Gaussian free field), with two pairs $(\mcl D,h)$ and $(\wt{\mcl D},\tilde{h})$ considered to be equivalent if there exists a conformal map $\psi:\wt{\mcl D} \rightarrow \mcl D$ such that $\tilde{h}=h \circ \psi + Q \log|\psi'|$ for $Q = \frac{2}{\sqrt{2}} + \frac{\sqrt{2}}{2}$.  
More generally, a \emph{$\sqrt{2}$-LQG surface with $k$ marked points} is an equivalence class of $(k+2)$-tuples $(\mcl D,h,x_1,\dots,x_k)$ with $(\mcl D,h)$ as above and $x_1,\dots,x_k \in \mcl D\cup \bdy \mcl D$, with two such $(k+2)$-tuples declared to be equivalent if there is a conformal map $\psi$ as above which also takes the marked points for one $(k+2)$-tuple to the corresponding marked points of the other. 
\end{defn}

One thinks of two equivalent pairs $(\mcl D,h)$ and $(\wt{\mcl D} , \wt h)$ as representing two different parametrizations of the same surface.
Duplantier and Sheffield~\cite{shef-kpz} constructed the volume form associated with a $\sqrt{2}$-LQG surface, which is a measure $\mu_h$ that can be defined as the limit of regularized versions of $e^{\sqrt{2} h(z)} \,dz$ (where $dz$ denotes Lebesgue measure). The volume form is well-defined on these equivalence classes: if $(\mcl D,h)$ and $(\wt{\mcl D},\tilde{h})$ are equivalent with $\psi:\wt{\mcl D} \rightarrow \mcl D$ as in Definition~\ref{def-lqg-surface}, then $\mu_h(\psi(A)) = \mu_{\wt{h}}(A)$ for each Borel subset $A$ of $\mcl D$.
 
In a similar vein, one can define the $\sqrt{2}$-LQG length measure $\nu_h$ on certain curves in $\ol{\mcl D}$, including $\bdy \mcl D$ and SLE$_{2}$-type curves (or equivalently the outer boundaries of SLE$_8$-type curves, by SLE duality~\cite{zhan-duality1,zhan-duality2,dubedat-duality,ig1,ig4}) which are independent from $h$. 
The measure $\nu_h$ is well-defined on equivalence classes in the same sense as $\mu_h$.
The $\sqrt{2}$-LQG length measure can be defined in various ways, e.g., using semi-circle averages of a GFF on a domain with smooth boundary and then confomally mapping to the complement of an SLE$_2$ curve~\cite{shef-kpz,shef-zipper} or directly as a Gaussian multiplicative chaos measure with respect to the Minkowski content measure on the SLE$_2$ curve~\cite{benoist-lqg-chaos}. 
See also~\cite{rhodes-vargas-review,berestycki-gmt-elementary,aru-gmc-survey} for expository results concerning of a more general theory of regularized measures of this form (called \emph{Gaussian multiplicative chaos}) which dates back to Kahane~\cite{kahane}.

\subsection{A key example: the $Q$-quantum wedge}
\label{sec-wedge-def}

One example of a $\sqrt{2}$-quantum surface which appears frequently in this paper is the \textit{$Q$-quantum wedge}. As in Definition~\ref{def-lqg-surface} above, $Q = \frac{2}{\sqrt{2}} + \frac{\sqrt{2}}{2}$; we will use this notation throughout the rest of the paper.   

We define the $Q$-quantum wedge in terms of its parametrization by the upper-half plane. Let  $\mcl{H}(\BB{H})$ denote the Hilbert space completion of the subspace of $C^{\infty}$ functions $f : \BB H\rta \BB R$ satisfying $\int_{\BB H}|\nabla f(z)|^2 \,dz < \infty$ with finite Dirichlet energy with respect to the Dirichlet inner product, $(f,g)_\nabla = \int_{\BB H} \nabla f(z) \cdot\nabla g(z)\,dz$. 
Since the Dirichlet norm of a constant function is zero, we view functions in $\mcl H(\BB H)$ as being defined modulo global additive constant.
We recall that the free-boundary GFF is the standard Gaussian random variable on $\mcl H(\BB H)$; see~\cite{shef-zipper}.

 We can decompose the space $\mcl{H}(\BB{H})$  into the orthogonal sum $\mcl{H}_1(\BB{H}) \oplus \mcl{H}_2(\BB{H})$, where $\mcl{H}_1(\BB{H})$ is the subspace of $\mcl{H}(\BB{H})$ consisting of functions that are constant on the semicircles $\bdy B_{e^{-t}}(0) \cap \BB{H}$ for all $t \in \BB{R}$, and $\mcl{H}_2(\BB{H})$ is the subspace of $\mcl{H}(\BB{H})$ consisting of functions that have mean zero on these semicircles. The following definition is taken from~\cite{wedges} (note that we are giving the definition in the critical case $\alpha =Q$, which differs from the usual case $\alpha  <Q$ considered in~\cite{wedges}).

\begin{defn}
The $Q$-quantum wedge is the doubly marked $\sqrt{2}$-LQG surface $(\BB H , h , 0, \infty)$ where $h$ is the distribution on $\BB H$ whose projections $h_1$ and $h_2$ onto  $\mcl{H}_1(\BB H)$ and $\mcl{H}_2(\BB H)$, respectively, can be described as follows:
\begin{itemize}
\item
$h_1$ is the function in $\mcl{H}_1(\BB{H})$ whose common value on $\bdy B_{e^{-t}}(0) \cap \BB{H}$ is equal to $B_{-2t} + Q t$ for $t<0$ and to $-X_{2t} + Q t$ for $t>0$, where $X_t$ is a dimension 3 Bessel process started at the origin and $B_{t}$ is a standard linear Brownian motion independent of $X_t$.
\item
$h_2$ is independent from $h_1$ and has the law of the projection of a free boundary GFF on $\BB{H}$ onto $\mcl{H}_2(\BB H)$. 
\end{itemize}
\label{wedgedef}
\end{defn}

Since the quantum wedge has only two marked points, one can obtain a different equivalence class representative of the quantum surface $(\BB H , h , 0, \infty)$ by replacing $h$ by $h(C\cdot) + Q\log C$ for $C>0$. The particular choice of $h$ appearing in Definition~\ref{wedgedef} is called the \emph{circle average embedding} of $h$ and is determined by the condition that $1 = \inf\{r > 0 : h_r(0) + Q\log r = 0\}$, where $h_r(0)$ is the average of $h$ over $\bdy B_r(0) \cap \BB H$.

\subsection{Whole-plane SLE$_8$ from $\infty$ to $\infty$ and its relation to chordal SLE$_8$}
\label{sec-sle-def}

The description of the one-sided mated-CRT map (given in the next subsection) involves both a $\sqrt{2}$-LQG surface and an independent SLE$_8$.   The two variants of SLE$_8$ that we will consider in this paper are chordal SLE$_8$ and whole-plane SLE$_8$ from $\infty$ to $\infty$.  The latter is just a two-sided version of chordal SLE$_8$, and is defined, e.g., in the discussion and footnote before \cite[Theorem 1.9]{wedges}. One way to construct a whole-plane SLE$_8$ from $\infty$ to $\infty$ is as follows:
\begin{enumerate}
\item First sample a whole-plane SLE$_2$ curve $\eta^L$ from 0 to $\infty$.
\item Conditional on $\eta^L$, sample a chordal SLE$_2(-1 ; -1)$ curve $\eta^R$ from 0 to $\infty$ in $\BB C\setminus \eta^R$ with force points immediately to the left and right of its starting point (see~\cite{ig1} for basic properties of chordal SLE$_\kappa(\rho^L;\rho^R)$ curves). 
\item Conditional on $\eta^L$ and $\eta^R$, sample a chordal SLE$_8$ from $\infty$ to $0$ in the connected component of $\BB C\setminus (\eta^L \cup \eta^R)$ lying to the left of $\eta^L$ and a chordal SLE$_8$ from 0 to $\infty$ in the connected component of $\BB C\setminus (\eta^L\cup \eta^R)$ lying to the right of $\eta^R$. Then concatenate these two chordal SLE$_8$'s. 
\end{enumerate} 
By~\cite[Theorems 1.1 and 1.11]{ig4}, the curves $\eta^L$ and $\eta^R$ can equivalently be described as the flow lines of a whole-plane GFF started from 0 with angles $\pi/2$ and $-\pi/2$, respectively. In particular, the joint law of $(\eta^L, \eta^R)$ is symmetric under swapping the order of the two curves.

In our proofs, we will want to transfer a quantitative probabilistic estimate for a  whole-plane SLE$_8$ from $\infty$ to $\infty$ to one for a chordal SLE$_8$.  To do this, we will use an encoding of these curves in terms of Gaussian free fields developed by Miller and Sheffield in their theory of imaginary geometry.  By~\cite[Theorem 1.6]{ig4} (see also~\cite[Theorem 1.1]{ig1}), a whole-plane SLE$_8$ from $\infty$ to $\infty$ can be constructed from a whole-plane GFF  $h^{\op{IG}}_1$.\footnote{Technically,~\cite{ig4} works with a whole-plane GFF defined modulo a global additive multiple of $2\pi\chi$, where $\chi = \sqrt 2 - 1/\sqrt 2$. In this paper, we always assume (somewhat arbitrarily) that the additive constant for our whole-plane GFF is chosen so that its circle average over $\bdy\BB D$ is zero. This particular distribution of course determines the equivalence class modulo global additive multiples of $2\pi \chi$.} 
A chordal SLE$_8$ from $0$ to $\infty$ in $\BB H$ can be constructed in the same way from a GFF $h^{\op{IG}}_2$ on $\BB H$ with boundary data $ -\pi/\sqrt 2$ (resp. $\pi/\sqrt 2$) on the negative (resp. positive) real axis. Moreover, both of these variants of SLE$_8$ are a.s. locally determined by the associated imaginary geometry field in the following precise sense. 

\begin{lem}
\label{lem-gms-harmonic-2.4}
Suppose that $\eta$ is a whole-plane SLE$_8$ from $\infty$ to $\infty$ parametrized by Lebesgue measure, with $h^{\op{IG}}_1$ the associated imaginary geometry field defined above.  Let $U \subset \BB C$ be an open set, and for $z$ in $U \cap \BB{Q}^2$, let
\[
T_z = \text{last time $\eta$ enters $U$ before hitting $z$}
\]
and
\[
S_z = \text{first time $\eta$ exits $U$ after hitting $z$}
\]
Then, for each $\delta > 0$, the restriction of  $h^{\op{IG}}_1$ to the $\delta$-Euclidean neighborhood of $U$ almost surely determines the collection of curve segments
\eqb
(\eta( \cdot + T_z)|_{[0,S_z - T_z]})_{z \in U \cap \BB{Q}^2}.
\label{eqn-curve-segments}
\eqe
The same holds for $\eta$ a chordal SLE$_8$ from $0$ to $\infty$ for a subset $U \subset \BB{H}$, with  $h^{\op{IG}}_1$ replaced by  $h^{\op{IG}}_2$.
\end{lem}

\begin{proof}
For $\eta$ a whole-plane SLE$_8$ from $\infty$ to $\infty$, this is~\cite[Lemma 2.4]{gms-harmonic}.  The proof is easily adapted to the chordal case.
\end{proof}

Lemma~\ref{lem-gms-harmonic-2.4} is the only fact from imaginary geometry which we will need for this paper. 


We can use Lemma~\ref{lem-gms-harmonic-2.4} to transfer quantitative probabilistic estimates for whole-plane SLE$_8$ from $\infty$ to $\infty$ to analogous estimates for chordal SLE$_8$. All we need to make this work is a quantitative Radon-Nikodym derivative estimate comparing the corresponding imaginary geometry fields:

\begin{lem}
Suppose $U$ is an open set whose closure is contained in $\BB{H}$. Define the fields $h^{\op{IG}}_1$ and $h^{\op{IG}}_2$ as above.  Then, the law of $ h^{\op{IG}}_2|_{\overline{U}}$  is absolutely continuous w.r.t.\ the law of $ h^{\op{IG}}_1|_{\overline{U}}$, with the Radon-Nikodym derivative having a finite $q$-th moment for some $q>1$.  
\label{lem-nikodym}
\end{lem}

\begin{proof}
Choose an open set $V$ such that $\overline{U} \subset V \subset \overline{V} \subset \BB{H}$. 
By comparing the Green functions associated to the whole-plane and zero boundary GFFs, we see that
\[
h^{\op{IG}}_1|_{\overline{V}} \stackrel{\mcl{L}}{=} h^{0}|_{\overline{V}} + f_1,
\]
where $h^{0}$ is a zero boundary GFF on $\BB{H}$ and $f_1$ is an independent random harmonic function on $\overline{V}$ which is a centered Gaussian process on $\overline{U}$ with covariances $\text{Cov}(f_1(x),f_1(y)) = - 2 \log|x - \bar{y}|$.
By the definition of $h^{\op{IG}}_2$, 
\[
h^{\op{IG}}_2|_{\overline{V}} \stackrel{\mcl{L}}{=} h^{0}|_{\overline{V}} + f_2,
\]
where $h^{0}$ is a zero boundary GFF on $\BB{H}$  and $f_2$ is a bounded (deterministic) harmonic function on $\overline{V}$.

Assume we have coupled $h_1^{\op{IG}}$ and $h_2^{\op{IG}}$ so that they are independent from one another (equivalently, $\eta$ and $\eta^{\op{whole}}$ are independent from one another). 
Let $\psi$ be a smooth compactly supported ``bump function" which equals 1 on $\ol U$ and vanishes outside of $V$. Then on $\ol U$, 
\[
h_1^{\op{IG}} = h^0 + f_1 \psi \qquad \text{and} \qquad h_2^{\op{IG}} = h^0 + f_2 \psi .
\] 
Set $f = f_1 - f_2$. 
By a standard Gaussian estimate---see, e.g., the proof of~\cite[Proposition 3.4]{ig1}---the law of $h_2^{\op{IG}}|_U$ is absolutely continuous w.r.t.\ the law of $h_1^{\op{IG}}|_U$, with Radon-Nikodym derivative 
\eqbn
M := \BB E\left[ \exp\left( (h_1^{\op{IG}} , f \psi)_\nabla - \frac12 (\psi f , \psi f)_\nabla \right) \,\big|\, h_1^{\op{IG}}|_U \right] ,
\eqen
where $(\cdot,\cdot)_\nabla$ denotes the Dirichlet inner product. 

By H\"older's inequality and since $(h_1^{\op{IG}} , f\psi)_\nabla$ is Gaussian with variance $(f\psi,f\psi)_\nabla$, for $q>1$,
\eqb
\BB E \left(M^q \right) 
\leq \BB{E} \exp\left( q (h_1^{\op{IG}}, f\psi)_{\nabla} - q (f\psi, f\psi)_{\nabla}/2 \right)
=
 \BB{E} \exp\left( \frac{q^2-q}{2} (f\psi, f\psi)_{\nabla} \right) \label{PE}
\eqe
A short computation using integration by parts (Green's identities) shows that
\[
 (f\psi, f\psi)_{\nabla} = \int_V \left( \frac{1}{2} \Delta ( \psi(w)^2) - \psi(w) \Delta \psi(w) \right) f(w)^2 dw.
\]
Therefore, using the notation $\| \varphi\|_D = \sup_{z\in D} |\varphi(z)|$ for a function $\varphi$ on a domain $D \subset \BB{C}$,~\eqref{PE} is bounded above by
\[
 \BB{E} \exp\left( c\frac{q^2 -q}{2} \left(\| f_1\|_{\overline{V}}^2 + \|f_2 \|_{\overline{V}}^2 \right) \right)
\]
for some constant $c$ depending only on $\psi$. 

Recall that $f_2$ is deterministic and uniformly bounded.  Moreover, since  $\|f_1\|_{\overline{V}} < \infty$,  the Borell-TIS inequality~\cite{borell-tis1,borell-tis2} (see, e.g.,~\cite[Theorem 2.1.1]{adler-taylor-fields}) gives $\BB{E}\|f_1\|_{\overline{V}} < \infty$,  $\sigma^2 := \sup_{v \in \overline{V}} \BB{E} |f_1|^2 < \infty$, and
\[
\BB{P}\left[ \|f_1\|_{\overline{V}} - \BB{E}\|f_1\|_{\overline{V}} > u \right] \leq \exp(-u^2/(2\sigma^2))
\]
for each $u > 0$.   Hence, we can choose $q>1$ sufficiently small such that
\[
\BB{E} \exp\left( c\frac{q^2-q}{2} \left(\| f_1\|_{\overline{V}}^2 + \|f_2 \|_{\overline{V}}^2 \right) \right)
\]
is finite, as desired.
\end{proof}

By the discussion before the lemma, this implies

\begin{cor}
\label{nikodym}
Let $\eta$ be a chordal SLE$_8$ from $0$ to $\infty$ in $\BB H$, and let $\eta^{whole}$ be a whole-plane SLE$_8$ from $\infty$ to $\infty$.  Let $U$ be an open subset of $\BB H$.
The law of the collection of segments of $\eta$ contained in $U$ (as defined, e.g., in~\eqref{eqn-curve-segments}) is absolutely continuous w.r.t.\ the law of the collection of segments of $\eta^{whole} $ contained in $U$, with the Radon-Nikodym derivative having a finite $q$-th moment for some $q>1$.  
\end{cor}

\subsection{The one-sided mated-CRT map: two equivalent definitions}
\label{sec-mated-crt}

\begin{figure}[t!]
 \begin{center}
\includegraphics[scale=.65]{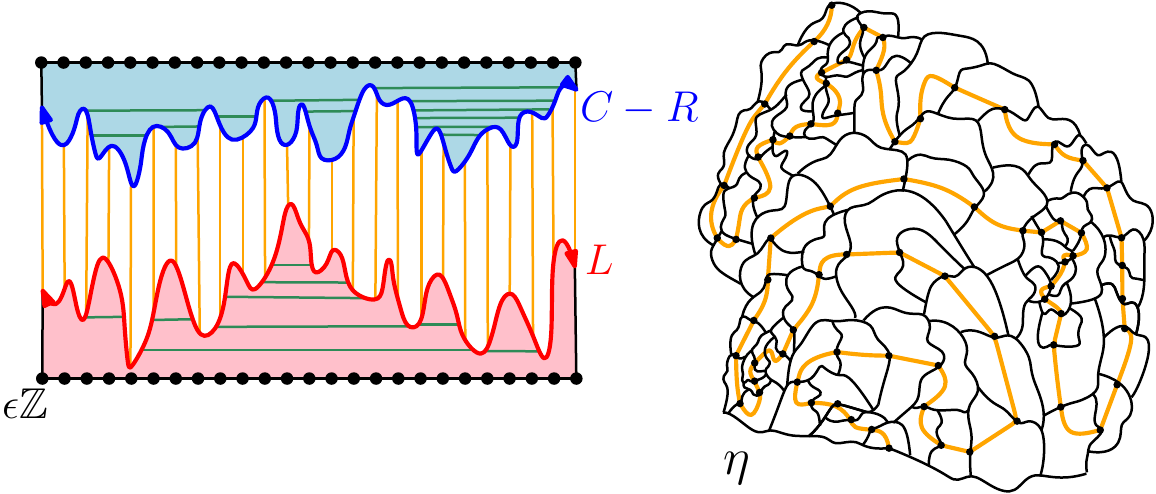}
\vspace{-0.01\textheight}
\caption{\textbf{Left:} We can construct the mated-CRT map $\mcl G^\ep$ restricted to some finite interval $I \subset \BB N$ geometrically as follows.  We can draw the graph of $L$ (red) and the graph of $C-R$ (blue) for some large constant $C > 0$ chosen so that the parts of the graphs over the time interval $I$ do not intersect. We then divide the region between the graphs into vertical strips (boundaries shown in orange) and identify each strip with the horizontal coordinate $x\in \ep \BB N \cap I$ of its rightmost point. A pair of vertices $x_1,x_2\in \ep \BB N \cap I$ is connected by an edge if and only if the corresponding strips are connected by a horizontal line segment which lies under the graph of $L$ or above the graph of $C-R$. For each pair of vertices in the figure that satisfy this condition, we have drawn one of these horizontal segments in green.
\textbf{Right:} The mated-CRT map can be realized as the adjacency graph of \emph{cells} $\eta([x-\ep,x])$ for $x\in \ep\BB Z$, where $\eta$ is a space-filling SLE$_8$ parametrized by $\sqrt 2$-LQG mass with respect to an independent $\sqrt 2$-quantum cone. Here, the cells are outlined in black and the order in which they are hit by the curve is shown in orange. 
Note that the two pictures do not correspond to the same mated-CRT map realization. 
Similar figures have appeared in~\cite{ghs-map-dist,gm-spec-dim,dg-lqg-dim,gh-displacement}. 
}\label{fig-mated-crt-map}
\end{center}
\vspace{-1em}
\end{figure} 

We now introduce a particularly useful random planar map with the half-plane topology called the \emph{one-sided mated-CRT map}.\footnote{It is possible to define a one-sided mated-CRT map associated to each $\gamma \in (0,2)$, but we will define the one-sided mated-CRT map as the one associated to $\gamma = \sqrt{2}$.}   Like the spanning-tree-weighted random planar map, the one-sided mated-CRT map is in the $\sqrt 2$-LQG universality class.  However, it can be defined directly in terms of continuum objects, namely, a $\sqrt{2}$-quantum surface and an independent space-filling SLE$_8$.  
 This allows us to use LQG and SLE techniques to prove quantitative estimates for graph distances in the one-sided mated-CRT map---estimates that we cannot prove directly for spanning-tree-weighted random planar maps.  The exponents $\chi$ and $d$ both arise in the context of the one-sided mated-CRT map, and it is this model that we will analyze in our proof of Theorem~\ref{thm-chi}.

There are actually two equivalent definitions of the one-sided mated-CRT map: one in terms of a pair of Brownian motions, and one in terms of LQG and SLE.  We will use both of these definitions in the paper since this will allow us to analyze distances in the one-sided mated-CRT map using both Brownian motion techniques and tools from LQG and SLE theory.  We first give the definition in terms of Brownian motions.

\begin{defn} \label{def-mated-CRT} 
For a pair $Z = (L,R)$ of independent one-sided Brownian motions, we define the \textit{one-sided mated-CRT map} $\mcl G^{\ep}$ as the graph with boundary whose vertex set is $\ep \BB{N}$, in which two vertices $x_1, x_2 \in \ep \BB{N}$ with $x_1 < x_2$ are considered to be adjacent iff there is an $s_1 \in [x_1-\ep,x_1]$ and an $s_2 \in [x_2 - \ep,x_2]$ such that either
\eqb \label{eqn-inf-adjacency}
\inf_{r \in [s_1,s_2]} L_r = L_{s_1} = L_{s_2} \quad \text{or} \quad \inf_{r \in [s_1,s_2]} R_r = R_{s_1} = R_{s_2}.
\eqe 
We define the boundary as the set of vertices $x\in \ep\BB N$ for which either $L$ or $R$ attains a running minimum in the time interval $[x-\ep,x]$.

Also, for an interval $I \subset (0,\infty)$, we denote by $\mcl G^{\ep}|_I$ the subgraph of $\mcl G^{\ep}$ induced by the set of vertices of $\mcl G^{\ep}$ contained in $I$. 
\end{defn}

See Figure~\ref{fig-mated-crt-map}, left, for an illustration of Definition~\ref{def-mated-CRT}. 
The graph is called the one-sided mated-CRT map because it can be viewed as a discretized mating of a pair of continuum random trees (CRTs) associated to the two one-sided Brownian motions. It can be constructed from the more commonly used two-sided mated-CRT map by restricting to vertices in $\ep\BB N$.
Definition~\ref{def-mated-CRT} defines the mated-CRT map as a graph, but it is easy to see that it also has a natural structure as a planar map with infinite boundary: this follows, e.g., from the discussion just after Definition~\ref{gdef}. 
One should think of Definition~\ref{def-mated-CRT} as a semi-continuous analog of the Mullin bijection. Indeed, it is easy to see that in the setting of the Mullin bijection, the condition for two triangles of $Q\cup T \cup T_*$ to share an edge is a discrete analog of~\eqref{eqn-inf-adjacency} with $(\mcl L , \mcl R)$ in place of $(L, R)$; see, e.g.,~\cite[Proposition 3.3]{ghs-map-dist}.

The connection between the graph just defined and the theory of LQG that will give us our second definition of the one-sided mated-CRT map is a special case of a framework developed in \cite{wedges} for identifying SLE-decorated LQG surfaces as canonical embeddings of mated CRT pairs.  It follows from the work of~\cite{wedges} that the one-sided mated-CRT map  can be coupled to the graph defined by a particular discretization of an SLE$_8$-decorated $Q$-quantum wedge, with the coupling defined so that it yields a natural graph isomorphism between the two graphs. To describe this equivalence, we first define a more general class of graphs obtained by discretizing SLE$_8$-decorated $\sqrt{2}$-LQG surfaces (along with a natural notion of distance in such graphs), since our proofs will sometimes require this more general context.

\begin{defn}
Let $(\mcl D,h)$ be a quantum surface, with $h$ some variant of a GFF on $\mcl D$; and let $\eta$ be an independent space-filling SLE$_{8}$-type curve\footnote{In this paper, we will take $\mcl D = \BB{H}$ and $\eta$ a chordal SLE$_8$ from $0$ to $\infty$ in $\BB H$ parametrized by $[0,\infty)$.  One could also consider, e.g., $\mcl D = \BB C$ and $\eta$ a whole-plane SLE$_8$ from $\infty$ to $\infty$ in $\BB C$ parametrized by $\BB R$.
} sampled independently from $h$ and then parametrized by $\sqrt{2}$-LQG mass with respect to $h$. For fixed $\ep > 0$, we define $\mcl{G}_{\mcl D,h,\eta}^{\ep}$ to be the graph of cells of the form  $\eta([x-\ep,x])$, where $x$ ranges over the elements of $\ep \BB{Z}$ for which $[x-\ep,x]$ lies in the domain of $\eta$. Two such cells are considered to be adjacent in the graph if they intersect.

Also, for $U \subset \mcl D$ and $z_1,z_2$ in the closure of $U$ in $\BB{C}$, let $D^{\ep}_{D,h,\eta}(z_1,z_2;U)$ be equal to the graph distance between the cells containing $z_1$ and $z_2$ in the adjacency graph of cells in $\mcl{G}^{\ep}_{\mcl D,h,\eta}$ that are contained in $\overline{U}$.
We abbreviate $D^{\ep}_{\mcl D,h,\eta}(\cdot,\cdot):= D^{\ep}_{\mcl D,h,\eta}(\cdot,\cdot;\mcl D)$. 
\label{gdef}
\end{defn}

We now consider the special case of Definition~\ref{gdef} which corresponds to the one-sided mated-CRT map.  
Suppose that $h$ is the circle average embedding of a $Q$-quantum wedge and our SLE$_8$ curve $\eta$ is sampled independently from $h$ and then parametrized by $\sqrt{2}$-quantum mass with respect to $h$. Consider for each $t > 0$ the closed region $\eta([0,t])$ of the upper-half plane, and let $x_t$ and $y_t$ denote the infimum and supremum, respectively, of the set of points where $\eta([0,t])$ intersects the real line.  We define the left boundary length $L_t$ of $\eta$ at time $t$ to be the $\sqrt{2}$-LQG length of the boundary arc of the closed region $\eta([0,t])$ from $\eta(t)$ to $x_t$, minus the $\sqrt{2}$-LQG length of  the segment $[x_t,0]$.   
Similarly, we define the right boundary length $R_t$ of $\eta$ at time $t$ to be the $\sqrt{2}$-LQG length of the boundary arc of  the closed region $\eta([0,t])$ from $\eta(t)$ to $y_t$, minus the $\sqrt{2}$-LQG length of  the segment $[0,y_t]$. See Figure~\ref{fig-bdy-process} for an illustration. We note that the definition of $(L,R)$ is the continuum analogue of the so-called horodistance process for peeling processes on random planar maps, as studied, e.g., in~\cite{curien-glimpse,gwynne-miller-perc}. 

\begin{figure}
\centering
\includegraphics[width=0.7\textwidth]{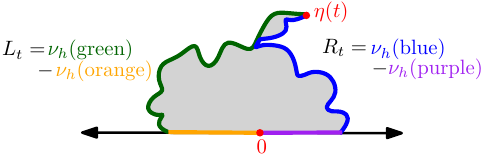}
\caption{The definitions of the processes $L$ and $R$.} \label{fig-bdy-process}
\end{figure}

It follows from \cite[Theorems 1.9]{wedges} (applied only for positive time and with $\gamma=\sqrt 2$) that the left and right boundary processes $L$ and $R$ are independent Brownian motions and the adjacency graph of cells $\mcl{G}_{\mcl D,h,\eta}^{\ep}$ is isomorphic to the one-sided mated-CRT map $\mcl{G}^{\ep}$ via $x\mapsto \eta([x-\ep,x])$. This gives us a second equivalent definition of the one-sided mated-CRT map. 

\begin{notation}
In the special case of Definition~\ref{gdef} which corresponds to the one-sided mated-CRT map discussed just above, we abbreviate $D^{\ep}_{D,h,\eta}$ by $D^{\ep}$.
\label{notation-D-ep}
\end{notation}

 We will express the law of $\mcl{G}^{\ep}$ in terms of either Definition~\ref{def-mated-CRT} or~\ref{gdef} depending on the context: the former definition will be more useful in Section~\ref{proofpart1}, and the latter in Section~\ref{proofpart2}.

\section{Proof of Theorem~\ref{thm-chi} conditional on two propositions}
\label{outline}

The proof of Theorem~\ref{thm-chi} is the most technically involved part of this paper. 

The starting point of the proof is the following two results from the existing literature. The first follows from \cite[Theorem 1.15]{ghs-dist-exponent} and the representation of the mated-CRT map as an adjacency graph of cells; the second is a special case of  \cite[Proposition 4.4]{dg-lqg-dim}.

\begin{lem}[\cite{ghs-dist-exponent}]
Let $(\BB H , h , 0 ,\infty)$ be a $Q$-quantum wedge, let $\eta$ be an independent space-filling SLE$_8$ parameterized by $\sqrt 2$-LQG mass with respect to $h$, and let $D^\ep_{\BB H , h,\eta}$ be as in Definition~\ref{gdef}. For each $\zeta \in (0,1)$,
\eqbn
\BB{P}\left[ D_{\BB H , h , \eta}^\ep \left( 0 , 1 ; \eta[0,1] \right) \geq \ep^{-\chi + \zeta} \right] \geq \ep^{o_{\ep}(1)} \quad \text{as $\ep\rta 0$}.
\eqen 
\label{thm1.15}
\end{lem}

We note that Lemma~\ref{thm1.15} does \emph{not} say that $D_{\BB H , h , \eta}^\ep \left( 0 , 1 ; \eta[0,1] \right) \geq \ep^{-\chi + \zeta} $ with probability tending to 1 as $\ep\rta 0$, but this will eventually follow from our results. 

\begin{lem}[\cite{dg-lqg-dim}]
Let $h^{\op{whole}}$ be a whole-plane GFF normalized so that its circle average over $\bdy \BB{D}$ is zero, let  $\eta^{\op{whole}}$ be an independent whole-plane SLE$_8$ from $\infty$ to $\infty$, and fix $\zeta \in (0,1)$.
For any open set $U \subset \BB{C}$ and any $z,w \in U$, it holds with polynomially high probability as $\ep \rightarrow 0$ that
\[
D^{\ep}_{\BB{C},h^{\op{whole}},\eta^{\op{whole}}}(z,w;U) \leq \ep^{-1/\left(d-\zeta\right)}
\]
\label{prop4.4}
\end{lem}

The first lemma bounds one quantity from below by $\ep^{-\chi + u}$ with probability $\geq \ep^{o_{\ep}(1)}$; the second bounds another quantity from above by $\ep^{-1/(d-\zeta)}$ with polynomially high probability as $\ep \rightarrow 0$.  If the quantities being bounded in the two lemmas were the same, then since $\ep^{o_{\ep}(1)} \leq 1 - \ep^p$ for any fixed $p$ and $\ep$ sufficiently small, we could deduce that $\ep^{-\chi + u} \leq \ep^{-1/(d-\zeta)}$, which would imply that $\chi \leq 1/d$. Together with Lemma~\ref{lem-obvious-direction}, this proves $\chi = 1/d$.

However, the quantities being bounded in the two lemmas are not the same.  Both are distances are of the form defined in Definition \ref{gdef}, but the underlying fields are different. More importantly, Lemma~\ref{prop4.4} considers the distance between two interior points of a region of $\BB{C}$, while Lemma~\ref{thm1.15} bounds the distance between  a pair of boundary points of a random region (one of which is the origin, where the field has a log singularity). So we cannot directly combine Lemmas \ref{prop4.4} and \ref{thm1.15} to prove Theorem~\ref{thm-chi}.  Instead, we will use Lemmas~\ref{prop4.4} and~\ref{thm1.15} to prove upper and lower bounds for a third distance quantity that we will be able to relate to the two different distances considered in Lemmas~\ref{prop4.4} and~\ref{thm1.15}. 

To formulate this third distance quantity, we define, for an interval $I  = (a,b) \subset [0,\infty)$, a notion of upper and lower boundary associated to the submap $\mcl G^\ep|_I$.  We can view the submap $\mcl G^\ep|_I$ as the adjacency graph of cells $\eta([x-\ep,x])$ for $x\in \ep\BB Z \cap I$.  We define the upper boundary as the set of cells adjacent to cells in $\mcl G^\ep|_{(b,\infty)}$, and the lower boundary as the set of cells either adjacent to cells in $\mcl G^\ep|_{[0,a)}$ \emph{or} adjacent to the boundary $\BB R$ of the upper-half plane.  (The definitions of the upper and lower boundaries are not symmetric with respect to each other because we are looking at the \emph{one-sided} mated-CRT map.) We state these definitions in terms of the Brownian motions $L$ and $R$.

\begin{defn}
\label{def-boundary}
For an interval $I \subset [0,\infty)$, we define the lower (resp. upper) boundary $\ul{\bdy}_{\ep}I$ (resp. $\ol{\bdy}_{\ep}I$) of the one-sided mated-CRT map $\mcl G^{\ep}$ restricted to $I$ to be the set of points in $(\ep \BB N) \cap I$ such that either of the Brownian motions $L,R$ (resp. their time reversals) achieves a running infimum in $[x-\ep,x]$ relative to the left (resp. right) endpoint of $I$. We define $\bdy \mcl G^\ep|_I = \ul{\bdy}_\ep I \cup \ol{\bdy}_\ep I$. 
\end{defn}

In terms of the left side of Figure~\ref{fig-mated-crt-map}, a point $x \in (\ep\BB N) \cap I$ is on the lower (resp. upper) boundary if we can connect the corresponding vertical strip to the left (resp. right) boundary of the leftmost (rightmost) vertical strip, by a horizontal line that lies either under the graph of $L$ or above the graph of $R$.  

If we write $I = (a,b)$, then from Definition~\ref{def-mated-CRT}, we see that Definition~\ref{def-boundary} is simply saying that
\begin{itemize}
\item
$\ul\bdy_\ep I$  is the set of vertices of $\mcl G^\ep|_I$ that either are adjacent to vertices of $\mcl G^\ep|_{[0,a)}$ or lie in $\bdy\mcl G^\ep$; and
\item
$\ol\bdy_\ep I$ is the set of vertices of $\mcl G^\ep|_I$ that are adjacent to vertices of $\mcl G^\ep|_{[b,\infty)}$.
\end{itemize}
So Definition~\ref{def-mated-CRT} is indeed equivalent to our first definition of the upper and lower boundaries in terms of SLE cells.

We will prove the following upper and lower bounds. For the statements, we recall the notation for graph distance from the end of Section~\ref{sec-notation}.  


\begin{prop}
In the notation of Definitions~\ref{def-mated-CRT} and~\ref{def-boundary}, for each $\zeta \in (0,1)$,  
\eqbn
 \BB P\left[ \op{dist}\left( \ep , \ol{\bdy}_{\ep}(0,1]; \mcl G^\ep|_{(0,1]} \right) \geq \ep^{-\chi + \zeta} \right] \geq \ep^{o_\ep(1)} ,\quad \text{as $\ep\rta 0$}.
\eqen 
\label{thm1.15modified}
\end{prop}

\begin{prop}
For each $\zeta \in (0,1)$, it holds with polynomially high probability as $\ep \rta 0$ that
\[
\op{dist}\left( \ep , \ol{\bdy}_{\ep}(0,1]; \mcl G^\ep|_{(0,1]} \right) \leq  \ep^{- 1/(d - \zeta)} .
\]
\label{prop4.4modified}
\end{prop}

\begin{proof}[Proof of Theorem~\ref{thm-chi} assuming Propositions~\ref{thm1.15modified} and~\ref{prop4.4modified}]
Propositions~\ref{thm1.15modified} and~\ref{prop4.4modified} imply that, for $\ep$ sufficiently small,
\[
\ep^{-\chi+\zeta} \leq \op{dist}\left( \ep , \ol{\bdy}_{\ep}(0,1]; \mcl G^\ep|_{(0,1]} \right) \leq  \ep^{- 1/(d - \zeta)}
\]
on some positive probability event. We deduce that $\chi-\zeta \leq 1/(d-\zeta)$ for each $\zeta \in (0,1)$, whence $\chi  \leq 1/d$.   Together with Lemma~\ref{lem-obvious-direction}, this gives $\chi = 1/d$.
The upper bound in~\eqref{eqn-chi} is immediate from the fact that $\chi = 1/d$ and~\eqref{eqn-dist-upper0}. 

We now need to prove the lower bound in~\eqref{eqn-chi}. 
By the comparison between the mated-CRT map and the spanning-tree-weighted random planar map established in~\cite[Theorem 1.5]{ghs-map-dist},
to prove the lower bound in~\eqref{eqn-chi} it suffices to show that for the one-sided mated-CRT map with $\ep=1$, one has
\eqb \label{eqn-mated-crt-map-tail}
\BB P\left[ \op{diam}  \left(\mcl G^1|_{(0,n]} \right) < n^{1/d - \zeta} \right]  = O_n(n^{-p}) , \quad \forall p > 0 .
\eqe 

To this end, we first note that $\eta([0,1])$ has positive probability to contain a Euclidean ball centered at $\eta(1/2)$ which lies at positive distance from zero. By this, the lower bound for distances in the whole-plane setting from~\cite[Proposition 4.6]{dg-lqg-dim}, and local absolute continuity away from the boundary (see Lemma~\ref{nikodym}), there is a constant $p > 0$ (independent of $\ep$) such that for each $\ep \in (0,1)$, 
\eqbn
\BB P\left[ \max_{x \in (\ep\BB N) \cap (0,1]} \op{dist}\left( x , \bdy \mcl G^\ep|_{(0,1]} ; \mcl G^\ep|_{(0,1]} \right) \geq \ep^{-1/d - \zeta } \right] \geq p .
\eqen
It is immediate from the Brownian motion definition~\ref{def-mated-CRT} of $\mcl G^\ep$ that for each $k \in \BB N_0$ and $n\in\BB N$, the map $\mcl G^1|_{[(k-1) n^{1-\zeta} , k n^{1-\zeta}] }$ agrees in law with $\mcl G^{n^{-(1-\zeta)} }|_{(0,1]} $. 
Hence, $\mcl G^1|_{[(k-1) n^{1-\zeta} , k n^{1-\zeta}] }$ has uniformly positive probability (independent of $n$ and $k$) to contain a vertex which lies at graph distance at least $n^{  1/d - 2\zeta}$ from its boundary. 
These maps for distinct choices of $k$ are i.i.d.\ (since the Brownian motion $Z$ has independent increments).
Hence with superpolynomially high probability there is at least one value of $k\in [1,n^\zeta]_{\BB Z}$ for which $\mcl G^1|_{[(k-1) n^{1-\zeta} , k n^{1-\zeta}] }$ contains a vertex which lies at graph distance at least $n^{  1/d - 2\zeta}$ from its boundary. If this is the case, then $\op{diam}   \left(\mcl G^1|_{(0,n]} \right) \geq n^{ 1/d - 2\zeta}$. Since $\zeta$ can be made arbitrarily small, this gives~\eqref{eqn-mated-crt-map-tail}. 
\end{proof}

The rest of the paper is devoted to proving Propositions~\ref{thm1.15modified} and~\ref{prop4.4modified}. 

\begin{remark} 
We note that our proofs of Propositions~\ref{thm1.15modified} and~\ref{prop4.4modified} require different formulations of the law of $\mcl G^{\ep}$. To prove Proposition~\ref{thm1.15modified}, we want to define $\mcl G^{\ep}$ in terms of Definition~\ref{def-mated-CRT}, because our proof involves looking at variants of the Brownian motion $(L,R)$ in Definition~\ref{def-mated-CRT} and analyzing the corresponding maps and their relationship to $\mcl G^{\ep}$.  On the other hand, to prove Proposition~\ref{prop4.4modified},  we  want to define $\mcl G^{\ep}$ in terms of Definition~\ref{gdef}, because we want to compare distances in $\mcl G^{\ep}$ with distances in other maps of the form defined in Definition~\ref{gdef}, such as that considered in Proposition~\ref{prop4.4} above. 
\end{remark}

\section{Proof of Proposition \ref{thm1.15modified} from Lemma \ref{thm1.15}}
\label{proofpart1}

In this section we prove Proposition \ref{thm1.15modified} from Lemma \ref{thm1.15}. To make the proof a bit easier to read, we slightly modify the statement of Proposition~\ref{thm1.15modified} to the following equivalent assertion (the equivalence follows from the reversibility of Brownian motion). Note that $\ep \lfloor \ep^{-1} \rfloor$ is the rightmost vertex of $\mcl G^\ep|_{(0,1]}$.

\begin{prop}
In the notation of Definitions~\ref{def-mated-CRT} and~\ref{def-boundary}, for each $\zeta \in (0,1)$,  
\eqbn
 \BB P\left[ \op{dist}\left(   \ul{\bdy}_{\ep}(0,1] , \ep \lfloor \ep^{-1} \rfloor ; \mcl G^\ep|_{(0,1]} \right) \geq \ep^{-\chi + \zeta} \right] \geq \ep^{o_\ep(1)} ,\quad \text{as $\ep\rta 0$}.
\eqen 
\label{thm1.15modified'}
\end{prop}

We first analyze the case when $ \ul{\bdy}_{\ep}(0,1] = \{\ep\}$  is a single point, corresponding to the first $\ep$-length interval $[0,\ep]$ of the positive real axis.  This happens in particular when $L$ and $R$ are each non-negative on $[0,1]$.  To this end, we define $\wh{\mcl{G}}^\ep$ as we defined the one-sided mated-CRT map in Definition~\ref{def-mated-CRT} but with $(L,R)$ replaced by a pair $\wh Z = (\wh L, \wh R)$ of independent Brownian meanders started from $0$---i.e., $\wh Z$ is a standard two-dimensional Brownian motion conditioned to stay in the first quadrant until time $1$. We also define the subgraphs $\wh{\mcl{G}}^\ep|_I$ for intervals $I \subset (0,\infty)$ as we did for $\mcl{G}^\ep$ in Definition~\ref{def-mated-CRT}.
One has the following miracle (which is not true for any $\gamma \not=\sqrt 2$, in which case the coordinates of $Z$ are correlated).
 
\begin{prop}
For $\gamma=\sqrt 2$, the law of the mated-CRT map $\wh{\mcl{G}}^{\ep}|_{(0,1]}$ associated with $\wh Z$ is absolutely continuous with respect to the law of $\mcl{G}^{\ep}|_{(0,1]}$. The Radon-Nikodym derivative of the law of the former map with respect to the law of the latter map agrees in law with $\frac{\pi}{2} (X_1^L X_1^R)^{-1} $, where $X^L$ and $X^R$ are independent three-dimensional Bessel processes.
\label{structure}
\end{prop}
\begin{proof}
Let $ (L,R)$ be a pair of independent standard linear Brownian motions and for $t\geq 0$ define $S_t^L = \max_{0 \leq s \leq t} L_s$ and $S_t^R :=  \max_{0 \leq s \leq t} R_s$.
By~\cite[Theorem 1.3]{pitman-bessel}, the processes $X^L := 2S^L - L$ and $X^R = 2S^R - R$ are independent 3-dimensional Bessel processes. 
By the equation just before~\cite[Corollary 2]{imhof-factorization}, applied to each of $\wh L$ and $\wh R$, the law of the pair of meanders $(\wh L ,\wh R)|_{[0,1]}$ is absolutely continuous with respect to the law of $(X^L , X^R)|_{[0,1]}$ with Radon-Nikodym derivative $\frac{\pi}{2} (X_1^L X_1^R)^{-1} $. 
Hence, if we define $\mcl G^X$ as in Definition~\ref{def-mated-CRT} with $(X^L,X^R)$ in place of $(L,R)$, then the law of $\wh{\mcl G}|_{(0,1]}$ is absolutely continuous with respect to the law of $\mcl G^X|_{(0,1]}$, with Radon-Nikodym derivative $\frac{\pi}{2} (X_1^L X_1^R)^{-1} $.

We will now complete the proof by showing that $\mcl G^X$ agrees in law with $\mcl G$. 
In fact, we will show that $\mcl G^X$ is a.s.\ identical to the mated-CRT map associated with $(-L,-R)$. 
By Definition~\ref{def-mated-CRT}, it is enough to show that, for $x_1,x_2 \in \ep \BB{N}$ with $x_1 < x_2$ and $s_j \in [x_j-\ep,x_j]$, for $j\in \{1,2\}$, the conditions
\begin{equation}
\inf_{r \in [s_1,s_2]} (2S_r^L - L_r) = 2S_{s_1}^L - L_{s_1} = 2 S_{s_2}^L - L_{s_2}
\label{2s-b}
\end{equation}
and
\begin{equation}
\sup_{r \in [s_1,s_2]} L_r =  L_{s_1} = L_{s_2}
\label{b}
\end{equation}
are equivalent (by symmetry, the same is true with $R$ in place of $L$). 
 Indeed, we claim that if either~\eqref{2s-b} and~\eqref{b} holds, then $S_{s_1}^L$ and $S_{s_2}^L$ must be equal.  This is clearly true if~\eqref{b} holds.  Moreover, if $S_{s_1}^L < S_{s_2}^L$, then, choosing $r \in [s_1,s_2]$ such that $L_r = S_r^L < S_{s_2}^L$, we have $2 S_r^L - L_r = S^L_r < 2 S^L_{s_2} - L_{s_2}$; so~\eqref{2s-b} fails. Hence~\eqref{2s-b} and~\eqref{b} are equivalent.
 \end{proof}

\begin{cor}
For each $\zeta \in (0,1)$, the $\wh{\mcl G}^\ep|_{(0,1]}$-distance between the leftmost and rightmost vertices satisfies
\eqb \label{eqn-ghat}
\BB{P}\left[\op{dist}\left(\ep, \ep \lfloor \ep^{-1} \rfloor ; \wh{\mcl{G}}^{\ep}|_{(0,1]}\right) \geq \ep^{-\chi + \zeta} \right] \geq \ep^{o_{\ep}(1)} ,\quad \text{as $\ep\rta 0$} .
\eqe
\label{cor-ghat}
\end{cor} 
\begin{proof}
Let 
\eqbn
E^\ep := \left\{ \op{dist}\left(\ep, \ep \lfloor \ep^{-1} \rfloor ;  \mcl{G}^{\ep}|_{(0,1]}\right) \geq \ep^{-\chi + \zeta} \right\} .
\eqen
By Lemma~\ref{thm1.15} and the equivalence of the Brownian and SLE/LQG representations of the mated-CRT map, we have $\BB P[E^\ep] \geq \ep^{o_\ep(1)}$. 
By Proposition~\ref{structure}, the probability of the event in~\eqref{eqn-ghat} is given by $\frac{\pi}{2} \BB E[\BB 1_{E^\ep} (X_1^L X_1^R)^{-1} ]$ for a certain pair of independent 3-dimensional Bessel processes $(X^L,X^R)$. 
Since a 3-dimensional Bessel process has the law of the modulus of a 3-dimensional Brownian motion, for $\delta >0$ one has $\BB P[(X_1^L X_1^R)^{-1} < \ep^\delta] \leq O_\ep(\ep^p)$ for all $p > 0$. Therefore,
\eqb
\BB E\left[\BB 1_{E^\ep} (X_1^L X_1^R)^{-1} \right] \geq \ep^\delta \BB P[E^\ep] - O_\ep(\ep^p) \geq \ep^{\delta + o_\ep(1)} .
\eqe
Since $\delta$ can be made arbitrarily small, this concludes the proof. 
\end{proof}

We now deduce Proposition~\ref{thm1.15modified'} from Corollary~\ref{cor-ghat}. Roughly, the idea of the proof is as follows.\footnote{The rest of this section is very similar to an argument which appeared in an earlier version of~\cite{ghs-dist-exponent}, but was removed in the final, published version.} If we condition on $\wh Z|_{[0, \ep^\zeta]}$ for a small but fixed $\zeta>0$, then the conditional law of $(\wh Z_{\ep^\zeta+\cdot} - \wh Z_{\ep^\zeta})|_{[0,1-\ep^\zeta]}$ is the same as the conditional law of $  Z|_{[0 , 1-\ep^\zeta ]}$ given that 
\eqb \label{eqn-thm1.15-cond}
\inf_{t\in [0,1-\ep^\zeta]} L_t \geq -\wh L_{\ep^\zeta} \quad \text{and} \quad  \inf_{t\in [0,1-\ep^\zeta]} R_t \geq -\wh R_{\ep^\zeta} .
\eqe
With high probability, $ \wh L_{\ep^\zeta}$ and $\wh R_{\ep^\zeta}$ are at least $\ep^{\zeta(1-\zeta)/2}$. This enables us to lower-bound the probability of the event in~\eqref{eqn-thm1.15-cond} and thereby compare the conditional law of $(\wh Z_{\ep^\zeta+\cdot} - \wh Z_{\ep^\zeta})|_{[0,1-\ep^\zeta]}$ given $\wh Z|_{[0, \ep^\zeta]}$ to the unconditional law of $  Z|_{[0 , 1-\ep^\zeta ]}$ up to multiplicative errors of the form $\ep^\zeta$ (note that $\zeta$ is small, so we can think of these as subpolynomial errors). This, in turn, allows us to compare distances in $\wh{\mcl G}^\ep|_{[\ep^\zeta,1]}$ and $\mcl G^\ep|_{[\ep^\zeta,1]}$. 
Furthermore, by~\cite[Theorem 1.15]{ghs-dist-exponent}, it is likely that every point on the lower boundary of $\wh{\mcl G}^\ep|_{[\ep^\zeta , 1]}$ (defined as in Definition~\ref{def-boundary}) is close (in the sense of $\wh{\mcl G}^\ep$-graph distances) to the initial vertex $\ep$. By the triangle inequality, this means that the distance in $\wh{\mcl G}^\ep$ from this lower boundary to 1 is close to the distance in $\wh{\mcl G}^\ep$ from $\ep$ to 1. 
This leads to a lower bound for the distance in $\mcl G^\ep|_{[\ep^\zeta,1]}$ from the lower boundary of $\mcl G^\ep|_{[\ep^\zeta,1]}$ to 1  in terms of the distance in $\wh{\mcl G}^\ep$ from $\ep$ to $1$ which holds with probability at least $\ep^{o_\ep(1)}$, as required. 

For the proof we will need two elementary Brownian motion lemmas.

\begin{lem} \label{prop-bm-inf}
Let $B$ be a standard linear Brownian motion and let $\ul t$ be the time at which $B$ attains its minimum on the interval $[0,1]$. For each $\delta > 0$ and each $\zeta \in (0,1)$,  
\eqbn
\BB P\left[ \ul t > \delta^{2-\zeta} \,|\, \inf_{s\in [0,1]} B_s \geq -\delta  \right] \preceq \delta^{\zeta/2} ,
\eqen
with the implicit constant depending only on $\zeta$. 
\end{lem}  
\begin{proof} 
We have $\BB P\left[ \inf_{t\in [0,1]} B_t \geq -\delta \right] \asymp \delta$, so it suffices to show that 
\eqb
\BB P\left[  \inf_{t\in [0,1]} B_t \geq -\delta , \ul t > \delta^{2-\zeta} \right]  \preceq \delta^{1+\zeta/2} .
\eqe
Using the Markov property of Brownian motion, it is easily seen that \[\BB P\left[B_t \in [-\delta,0] \,: \forall t \in [\delta^{2-\zeta} , 1] \right]\] decays faster than any positive power of $\delta$. 
It therefore suffices to show that 
\eqb
\BB P\left[  \inf_{t\in [0,1]} B_t \geq -\delta , \sup_{t\in [\delta^{2-\zeta} ,1]} B_t > 0 , \ul t > \delta^{2-\zeta} \right]  \preceq \delta^{1+\zeta/2} .
\eqe

To this end, let  $\sigma$ be the first time $t\geq \delta^{2-\zeta}$ for which $B_t = 0$. We observe that $B_{\ul t} \leq 0$, so if $\ul t \geq \delta^{2-\zeta}$ and $B_t > 0$ for some $t\in [\delta^{2-\zeta},1]$, then there is some $t \in [\delta^{2-\zeta} , 1)$ for which $B_{t} = 0$ and hence $\sigma  < 1$. 
Therefore,
\alb
&\BB P\left[  \inf_{t\in [0,1]} B_t \geq -\delta , \sup_{t\in [\delta^{2-\zeta} ,1]} B_t > 0 , \ul t > \delta^{2-\zeta} \right]  \notag \\
&\qquad \leq \BB P\left[  \inf_{t\in [0,\delta^{2-\zeta} ]} B_t \geq -\delta   \right] 
 \BB P\left[ \inf_{t\in [\sigma,1]} B_t \geq -\delta , \, \sigma < 1 \,|\,  \inf_{t\in [0,\delta^{2-\zeta} ]} B_t \geq -\delta  \right]  .
\ale
We have $ \BB P\left[  \inf_{t\in [0,\delta^{2-\zeta} ]} B_t \geq -\delta   \right] \preceq \delta^{\zeta/2}$. 
It therefore suffices to show that
\eqb \label{eqn-bm-show}
\BB P\left[ \inf_{t\in [\sigma,1]} B_t \geq -\delta , \, \sigma < 1  \,|\,  \inf_{t\in [0,\delta^{2-\zeta} ]} B_t \geq -\delta \right] \preceq \delta .
\eqe 

Intuitively,~\eqref{eqn-bm-show} follows from the fact that $\sigma$ is typically of order $\delta^{2-\zeta}$, and the probability that a Brownian motion stays positive on $[0,1-\delta^{2-\zeta}]$ is comparable to the probability that it stays positive on $[0,1]$. To be more precise, the Markov property implies that on the event $\{\sigma <1\}$, 
\eqb \label{eqn-cond-on-sigma}
\BB P\left[ \inf_{t\in [\sigma,1]} B_t \geq -\delta \,|\, B|_{[0,\sigma]} \right] \preceq \delta (1-\sigma)^{-1/2} .
\eqe 
The conditional law of $\sigma$ given that $\{\inf_{t\in [0,\delta^{2-\zeta} ]} B_t \geq -\delta\}$ has a continuous density with respect to Lebesgue measure which is bounded above by a $\delta$-independent constant in $[1/2,\infty)$. 
Consequently, for $k \in \BB N_0$, 
\eqb \label{eqn-sigma-law}
\BB P\left[ \sigma \in [1 - 2^{-k } , 1 - 2^{-k-1} ] \,|\,  \inf_{t\in [0,\delta^{2-\zeta} ]} B_t \geq -\delta \right]  \preceq 2^{-k } .
\eqe 
Combining~\eqref{eqn-cond-on-sigma} with~\eqref{eqn-sigma-law} and summing over $k$ shows that
\eqbn
\BB P\left[ \inf_{t\in [\sigma,1]} B_t \geq -\delta , \, \sigma < 1  \,|\,  \inf_{t\in [0,\delta^{2-\zeta} ]} B_t \geq -\delta \right] 
\preceq \delta \sum_{k=0}^\infty 2^{-k/2} \preceq \delta ,
\eqen
as required. 
\end{proof}
 
Recall that $\wh L$ and $\wh R$ are the independent 3d Bessel processes used to define $\wh{\mcl G}^\ep$.
 
\begin{lem} \label{prop-inf-times-poly}
For $\delta \in (0,1)$, let $\wh t^L_\delta$ (resp. $\wh t^R_\delta$) be the time at which $\wh L $ (resp. $\wh R $) attains its minimum value on the interval $[\delta,1]$. For each $\beta \in (0,1)$, it holds with polynomially high probability as $\delta\rta 0$ that $\wh t^L_\delta \leq \delta^\beta$ and $ \wh t^R_\delta \leq \delta^\beta$. 
\end{lem}  
\begin{proof} 
Since $\wh L$ and $\wh R$ are independent, it suffices to prove that with polynomially high probability as $\delta \rta 0$, one has $\wh t^L_\delta \leq \delta^\beta$.  
For $\zeta>0$, the probability that $\wh L_\delta$ does not belong to $[\delta^{1/2+\zeta} , \delta^{1/2 - \zeta}]$ decays polynomially in $\delta$. The conditional law of $(\wh L - \wh L_\delta)|_{[\delta ,1]}$ given $\wh L_\delta$ is that of a Brownian motion conditioned to stay above $-\wh L_\delta$. Therefore, the statement of the lemma follows from Lemma~\ref{prop-bm-inf} (applied with $\wh L_\delta$ in place of $\delta$) upon making an appropriate choice of $\zeta$ (depending on $\beta$).  
\end{proof}

\begin{proof}[Proof of Proposition~\ref{thm1.15modified'}] 
Let $\ep\in (0,1)$ and assume without loss of generality that $1/\ep\in\BB Z$, so that $\ep \lfloor \ep^{-1} \rfloor = 1$.
Also fix $u , v\in (0,1)$ to be chosen later, depending on $\zeta$. 
 Let $y_\ep := \ep \lfloor \ep^{u-1} \rfloor$, so that $y_\ep \in (0,1]_{\ep\BB Z}$. Let $\wt Z^\ep = (\wt L^\ep , \wt R^\ep)$ be a standard two-dimensional Brownian motion conditioned to stay in the first quadrant until time $1 + y_\ep$.  Note that $\wt Z^\ep$ is defined in a similar manner to $\wh Z$ but conditioned to stay in the first quadrant for slightly more than one unit of time. Define $\wt{\mcl G}^\ep$ according to Definition~\ref{def-mated-CRT} but with $Z$ replaced by $\wt Z^{\ep}$; also, define $\wt{\mcl G}^\ep|_I$ for intervals $I \subset (0,\infty)$ as in Definition~\ref{def-mated-CRT}.  The idea of the proof is to first deduce a lower bound for the distance from $1+y_\ep$ to the lower boundary of the graph $\wt{\mcl G}^\ep|_{(y_\ep , 1+y_\ep]}$ (essentially, this follows from Corollary~\ref{cor-ghat}, the upper bound of \cite[Theorem 1.15]{ghs-dist-exponent}, and the triangle inequality); then compare this latter object to $\mcl G^\ep|_{(0,1]}$ conditioned on the positive probability event that its lower boundary length is small. 
\medskip
 
\noindent\textit{Step 1: lower bound for distances in $\wt{\mcl G}^\ep|_{(y_\ep , 1+y_\ep]}$.}
Let $\wt t^L_\ep$ and $\wt t^R_\ep$ be the times at which $\wt L^\ep$ and $\wt R^\ep$, respectively, attain their minimum values on $[y_\ep   , 1 + y_\ep]$. Letting $\ul{\wt\bdy}_\ep(y_\ep , 1+y_\ep]$ denote the lower boundary of $\wt{\mcl G}^\ep|_{(y_\ep , 1+y_\ep]}$ (in the sense of Definition~\ref{def-boundary}),  
\eqb \label{eqn-lower-bdy-contain}
\ul{\wt\bdy}_\ep(y_\ep ,1+y_\ep] \subset (y_\ep ,  \wt t^L_\ep \vee \wt t^R_\ep]_{\ep\BB Z} .
\eqe 

Fix $\beta \in (0,1)$ and let
\alb
E^\ep &:= \left\{ \op{dist}\left( \ep , 1+ y_\ep   ; \wt{\mcl G}^\ep|_{(0 , 1+y_\ep]} \right) \geq  \ep^{-\chi + v } \right\} 
\cap \left\{ \wt t^L_\ep \vee \wt t^R_\ep \leq \ep^{ \beta u }   \right\}  \notag\\
&\qquad\qquad \cap \left\{\wt Z^\ep_{y_\ep} \in \left[ \ep^{u (1+u)/2} , \ep^{u(1-u)/2} \right]^2 \right\} .
\ale
By Corollary~\ref{cor-ghat} and scale invariance, the probability of the first event in the definition of $E^\ep$ is at least $\ep^{o_\ep(1)}$. By Lemma~\ref{prop-inf-times-poly} and a standard estimate for linear Brownian motion (recall that the coordinates of $\wt Z^\ep$ are independent), the probability that each of the other two events in the definition of $E^\ep$ fails to occur decays polynomially in $\ep$. Therefore,
\eqbn
\BB P\left[E^\ep \right] \geq \ep^{o_\ep(1)} .
\eqen

By applying the upper bound for distance from~\cite[Theorem 1.15]{ghs-dist-exponent} (i.e., the analog of~\eqref{eqn-dist-upper0} for $\mcl G^\ep$) to each of the graphs $\wt{\mcl G}^\ep|_{(2^{-k} y_\ep , 2^{-k+1} y_\ep]}$ for $k\in\BB N$ such that $\ep \leq 2^{-k} y_\ep \leq \ep^{\beta u}$ and summing over all such $k$, we find that with superpolynomially high probability as $\ep\rta 0$, 
\eqb \label{eqn-cond-short-dist}
 \op{dist}\left( \ep , x ;  \wt{\mcl G}^\ep|_{(0 ,  y_\ep]} \right)  \leq \ep^{-(1- \beta u) \left(\chi + v\right)} ,\quad \forall x \in ( 0 , \ep^{ \beta u } ]_{\ep\BB Z} .
\eqe 
By~\eqref{eqn-lower-bdy-contain}, if $E^\ep$ occurs and~\eqref{eqn-cond-short-dist} holds then  
\eqb \label{eqn-bdy-short-dist'}
 \op{dist}\left( \ep , x ; \wt{\mcl G}^\ep|_{(0 , 1+y_\ep]} \right) \leq \ep^{-(1-\beta u) \left(\chi + v \right)} ,\quad \forall x \in \ul{\wt\bdy}_\ep(y_\ep ,1+y_\ep] .
\eqe  
Henceforth assume that $v$ is chosen sufficiently small (depending on $u$) that $(1-\beta u) \left(\chi +v\right) < \chi - v$. 

By~\eqref{eqn-bdy-short-dist'} and the triangle inequality, if $\ep$ is chosen sufficiently small (depending on $u$ and $v$), then whenever $E^\ep$ occurs and~\eqref{eqn-cond-short-dist} holds we have 
\eqb \label{eqn-bdy-dist-event}
\op{dist}\left( 1+y_\ep , \ul{\wt\bdy}_\ep(y_\ep ,1+y_\ep]  ; \wt{\mcl G}^\ep|_{(y_\ep , 1+y_\ep]} \right) \geq \frac12 \ep^{-\chi + v}  \quad \op{and} \quad \wt Z^\ep_{y_\ep} \in \left[ \ep^{u (1+u)/2} , \ep^{u(1-u)/2} \right]^2  .
\eqe  
\medskip

\noindent\textit{Step 2: comparison of $Z$ and $\wt Z^\ep$.}
Let $\wt E^\ep$ be the event that~\eqref{eqn-bdy-dist-event} holds. Then the above discussion shows that $\BB P[\wt E^\ep] \geq \ep^{o_\ep(1)} $. Hence for each $\delta>0$, 
\alb
\ep^{o_\ep(1)}  \leq \BB P[\wt E^\ep] 
 =  \BB E\left[ \BB P\left[\wt E^\ep \,|\, \wt Z^\ep|_{[0,y_\ep]} \right] \right]  
\leq \ep^\delta  +  \BB P\left[ \BB P\left[\wt E^\ep \,|\, \wt Z^\ep |_{[0,y_\ep]} \right]  \geq \ep^\delta \right]  .
\ale
By re-arranging, we obtain
\eqb \label{eqn-cond-bm-bdy-prob}
\BB P\left[ \BB P\left[\wt E^\ep \,|\, \wt Z^\ep |_{[0,y_\ep]} \right]  \geq \ep^\delta \right]  \geq \ep^{o_\ep(1)}  , \quad\forall \delta >0 .
\eqe 

The conditional law of $(\wt Z^\ep - \wt Z^\ep_{y_\ep}) |_{[y_\ep , 1 + y_\ep]}$ given $\wt Z^\ep|_{[0,y_\ep]}$ is the same as the law of $Z|_{[0,1]}$ conditioned on the event that 
\eqbn
\inf_{s\in [0,1]} L_s \geq -\wt L^\ep_{y_\ep} \quad \op{and} \quad \inf_{s\in [0,1]} R_s \geq -\wt R^\ep_{y_\ep}  .
\eqen
By this,~\eqref{eqn-cond-bm-bdy-prob}, and the definition~\eqref{eqn-bdy-dist-event} of $\wt E^\ep$, it follows that for each $\delta >0$ there exists $a_L , a_R \in \left[ \ep^{u (1+u)/2} , \ep^{u(1-u)/2} \right]$ such that
\eqbn
\BB P\left[\op{dist}\left(  \ul\bdy_\ep(0,1]  , 1; \mcl G^\ep|_{(0,1]} \right) \geq \ep^{-\chi + \zeta} \,|\, \inf_{s\in [0,1]} L_s \geq - a_L \: \op{and} \: \inf_{s\in [0,1]} R_s \geq -a_R \right] \geq \ep^\delta .
\eqen
Since $a_L , a_R \geq \ep^{u(1+u)/2}$ and $L$ and $R$ are independent,
\eqbn
\BB P\left[\inf_{s\in [0,1]} L_s \geq - a_L \: \op{and} \: \inf_{s\in [0,1]} R_s \geq -a_R \right] \geq \ep^{ u(1+u)} .
\eqen
Therefore,
\eqbn
\BB P\left[\op{dist}\left(  \ul\bdy_\ep(0,1] , 1 ; \mcl G^\ep|_{(0,1]} \right) \geq \ep^{-\chi + \zeta} \right] \geq \ep^{\delta +   u (1+ u) }.
\eqen
Sending $\delta\rta0$ and $u \rta 0$ concludes the proof of the proposition.  
\end{proof}


\section{Proof of Proposition \ref{prop4.4modified} from Lemma \ref{prop4.4}}
\label{proofpart2}

We now proceed to the proof of of Proposition \ref{prop4.4modified} from Lemma \ref{prop4.4}. In this section, we express the law of $\mcl G^{\ep}$ in the form of Definition~\ref{gdef}, in which the vertices of the graph are subsets of $\BB H$ (``cells'') traced by a chordal SLE$_8$.  

To prove Proposition~\ref{prop4.4modified}, we construct, for some deterministic $\delta > 0$ sufficiently small and depending on $\zeta$, a path in $\mcl{G}^\ep$ from the cell containing the imaginary number $\ep^{-\delta} i/2$ to the cell containing $0$ with $D^{\ep}$-length $\ep^{-1/(d - \zeta)}$ with polynomially high probability as $\ep \rta 0$.   We then show that this path yields a path from the cell containing $0$ to the upper boundary $\ol{\bdy}_{\ep}(0,1]$ with the desired probability by showing that, with polynomially high probability as $\ep \rightarrow 0$, the curve $\eta$ will not absorb the point $\ep^{-\delta} i/2$ before time $1$.  

Our construction of the path in $\mcl{G}^\ep$ from $\ep^{-\delta} i/2$ to $0$ proceeds in three steps roughly as follows.
\begin{enumerate}
\item We convert the bound of Lemma \ref{prop4.4} to a bound on the distance between two points with the distance defined in terms of a chordal SLE$_8$ from $0$ to $\infty$ in $\BB{H}$ and the scale-invariant component  $h_2$ of a quantum wedge field in $\BB{H}$ (Proposition~\ref{firststep}).
\item We apply  this bound to the pairs of points  $\ep^{-\delta} i/2^{j-1}$ and $\ep^{-\delta} i/2^{j}$ for each integer $1 < j < \ep^{-\xi}$ for  fixed but small $\delta, \xi > 0$. If we make $\delta$ and $\xi$ sufficiently small, we can use the scale invariance of $h_2$ to bound distances between each of these pairs of points from above simultaneously, and then combine these bound to bound  the $D^{\ep}$-distance from $\ep^{-\delta} i/2$ to $\ep^{-\delta} i/2^{-\lfloor \ep^{-\xi} \rfloor}$, all with polynomially high probability as $\ep \rightarrow 0$ (Proposition \ref{polysteps}). 
\item We show that the point $\ep^{-\delta} i/2^{-\lfloor \ep^{-\xi} \rfloor}$ is contained in the origin-containing cell of  $\mcl{G}^{\ep}$ with polynomially high probability as $\ep \rightarrow 0$ (Proposition \ref{lastcell}).
\end{enumerate}

\begin{figure}[ht!] \centering
\begin{tabular}{ccc} 
\includegraphics[width=0.28\textwidth]{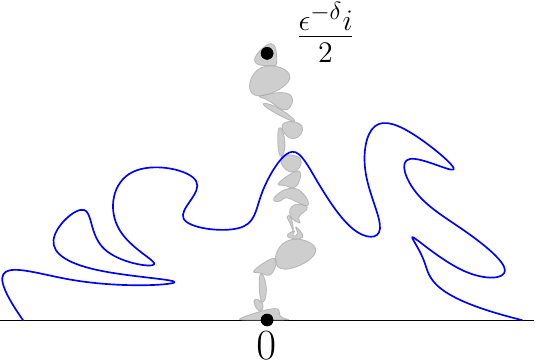}
&
\includegraphics[width=0.28\textwidth]{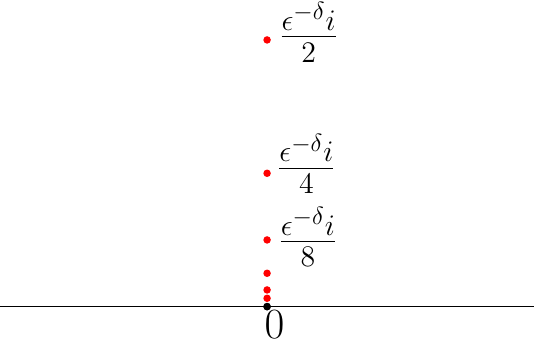}
&
\includegraphics[width=0.28\textwidth]{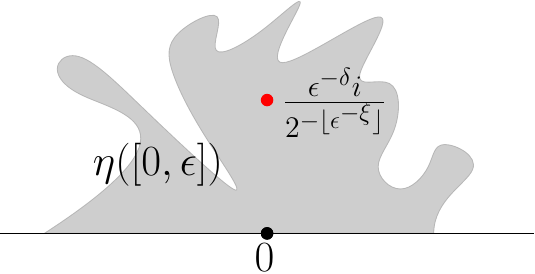}
\end{tabular}
\caption{A sketch of the proof of Proposition~\ref{prop4.4modified}. (For brevity, we omit the qualifier ``with polynomially high probability as $\ep \rightarrow 0$'' which applies to each of these steps.) \textbf{Left:} To upper-bound the $D^{\ep}$-distance between $\ol{\bdy}_{\ep}(0,1]$ and the origin, we construct a path of cells of $\mcl G^{\ep}$ (colored in gray) from $\ep^{-\delta} i/2$ to $0$ of the desired length. This path will contain a path from $\ol{\bdy}_{\ep}(0,1]$ to the cell containing $0$ as long as the point $\ep^{-\delta} i/2$ lies outside $\eta(0,1)$, which we show separately.  \textbf{Center:} We construct the path from the cell containing $\ep^{-\delta} i/2$ to the cell containing $0$  using a multiscale argument:  starting with the point $\ep^{-\delta} i/2$, we consider a sequence of polynomially (in $\ep$) many  points approaching the origin geometrically, and we bound distances between consecutive pairs of points.  \textbf{Right:}  (Zoomed in near the origin.) We show that the last point in the sequence is contained in the origin-containing cell of  $\mcl{G}^{\ep}$.}
\label{fig-multiscale}
\end{figure}

\subsection{Two notions of distance in terms of SLE cells}
\label{sec-metric-def}

To execute the construction just outlined, we first need a notation of distance that differs slightly from the metric $D^{\ep}_{\mcl D,h,\eta}(\cdot,\cdot;U)$ defined in Definition~\ref{gdef}. Though the latter is quite natural, it has the crucial drawback that it does not just depend on the restriction of $h$ and $\eta$ to $U$, since we cannot ``see" which times for $\eta$ are elements of $\ep\BB Z$ if we just see $h|_U$ and the segments of $\eta$ contained in $U$. We therefore define a second notion of distance that possesses this useful property.

\begin{defn} \label{def-cell-dist-tilde}
Let $\mcl D$, $h$ and $\eta$ be as in Definition \ref{gdef}. For fixed $\ep > 0$, $U \subset \mcl D$, and $z_1,z_2$ in the closure of $U$ in $\BB{C}$, let $\wt{D}^{\ep}_{\mcl D,h,\eta}(z_1,z_2;U)$ be equal to the minimum number of SLE segments of the form $\eta([a,b])$ for $0 < b-a \leq \ep$ which are contained in $\overline{U}$ and whose union contains a Euclidean path from $z_1$ to $z_2$. 
We abbreviate $\wt{D}^{\ep}_{\mcl D,h,\eta}(\cdot,\cdot):= \wt{D}^{\ep}_{\mcl D,h,\eta}(\cdot,\cdot;\mcl D)$.
\end{defn} 
 
We state two useful properties of the metric $\wt{D}^{\ep}_{\mcl D,h,\eta}(\cdot,\cdot;U)$ that are easy to verify.

\begin{lem} \noindent \label{lem}
Let $h_1$ and $h_2$ be two GFF type distributions defined on a domain $\mcl D$,  let  $U \subset \mcl D$, let $\eta$ be as in Definition \ref{gdef}.  Let $z_1,z_2 \in U$.
\begin{enumerate}[label={(\roman*)}, ref={\thelem(\roman*)},itemindent=1em]

       \item
If $h_1 =  h_2 + c$ a.s. for some constant $c$, then, almost surely,
\[
\wt{D}_{\mcl D,h_1,\eta}^{\ep}(z_1,z_2;U) = \wt{D}_{\mcl D,h_2,\eta}^{e^{-\sqrt{2} c} \ep}(z_1,z_2;U) .
\]
\label{lem:distcomparison} 
       \item 
If $h_1$ and $h_2$ are two GFF type distributions defined on a domain $\mcl D$ such that $(h_2 - h_1)|_U$ is a non-negative function on $U$ and if $\ep_1 \geq \ep_2$ then, almost surely,
\[
\wt{D}_{\mcl D,h_1,\eta}^{\ep_1}(z_1,z_2;U) \leq \wt{D}_{\mcl D,h_2,\eta}^{\ep_2}(z_1,z_2;U) .
\]
\label{lem:monotonicity}
     \end{enumerate}
\end{lem}
\begin{proof}
The first statement follows since adding $c$ to the field scales the $\sqrt 2$-LQG area of each segment of $\eta$ by $e^{\sqrt 2 c}$. The second statement follows since each segment of $\eta$ with $\mu_{h_2}$-area at most $\ep_2$ has $\mu_{h_1}$ area at most $\ep_1$. 
\end{proof}

Moreover, we can relate $\wt{D}^{\ep}_{\mcl D,h,\eta}(\cdot,\cdot;U)$ to the  distance $D^{\ep}_{\mcl D,h,\eta}(\cdot,\cdot;U)$ of Definition~\ref{gdef}, which is what we really care about. See~\cite[Lemma 4.2]{dg-lqg-dim} for a closely related statement. 

\begin{lem}
Let $\mcl D$, $h$ and $\eta$ be as in Definition \ref{gdef}, and let $U \subset \mcl D$.  
\begin{enumerate}[label={(\roman*)}, ref={\thelem(\roman*)},itemindent=1em]
\item We have 
\eqb
\wt{D}_{\mcl D,h,\eta}^{\ep}(z_1,z_2;U) \leq D_{\mcl D,h,\eta}^{\ep}(z_1,z_2;U) +1   \qquad \forall z_1,z_2 \in U
\eqe
\label{dist-comparison-a}
\item
On the event that each cell $\eta([x-\ep,x])$ for $x \in \ep \BB{Z}$ which intersects $\overline{U}$ is contained in $\overline{V}$ (where $V \supset U$ is an open subset of $\mcl D$),
\eqb
D_{\mcl D,h,\eta}^{\ep}(z_1,z_2;V) \leq 2 \wt{D}_{\mcl D,h,\eta}^{\ep}(z_1,z_2;U)  \qquad \forall z_1,z_2 \in U 
\label{dist-inequality}
\eqe
\label{dist-comparison-b}
\end{enumerate}
\end{lem} 
\begin{proof}
Part (i) is obvious (the $+1$ is needed since the $D_{\mcl D,h,\eta}^{\ep}$-distance between two points in the same segment $\eta([a,b])$ with $b-a \leq \ep$ is 1, whereas the $D_{\mcl D,h,\eta}^\ep$-distance between two such points can be zero). Part (ii) follows from the observation that each segment $\eta([a,b])$ with $0 < b-a \leq \ep$ which is contained in $U$ is contained in the union of at most two segments of the form $\eta([x-\ep,x])$ for $x \in \ep \BB{Z}$ which intersect $\ol{U}$, hence are contained in $\ol{V}$.
\end{proof}

If $h$ is taken to be a whole-plane GFF normalized to have mean zero on the unit circle, and $\eta$ a whole-plane SLE$_8$ from $\infty$ to $\infty$, the inequality~\eqref{dist-inequality} holds with polynomially high probability as $\ep \rta 0$.  This is a consequence of the following lemma, which follows from basic SLE/LQG estimates; see, e.g., the proof of~\cite[Lemma 2.4]{ghs-map-dist}, which applies verbatim in our setting. 

\begin{lem} \label{whole_lemma}
Let $h^{\op{whole}}$ be a whole-plane GFF normalized to have mean zero on the unit circle and $\eta^{\op{whole}}$  a whole-plane SLE$_8$ from $\infty$ to $\infty$. Then there is a constant $q>0$ such that for each bounded open set $U\subset\BB C$, it holds with polynomially high probability as $\ep \rta 0$ that the maximal diameter of the segments $\eta([a,b])$ with $0 < b-a \leq \ep$ and which intersect $\overline{U}$ is at most $\ep^q$.
\end{lem}

\subsection{Transferring from the whole plane to the half plane}
\label{sec-whole-half}

We are now ready to begin the proof of Proposition~\ref{prop4.4modified}.

We begin by proving the analog of Lemma~\ref{prop4.4} with $\eta^{\op{whole}}$ replaced by a chordal SLE$_8$ from $0$ to infinity in $\BB{H}$ and with $h^{\op{whole}}$ replaced by the field $h_2$ associated with the $Q$-quantum wedge (as defined in Definition~\ref{wedgedef}). Ultimately, we are interested in distances w.r.t.\ the metric $D^{\ep}_{\BB{H},h, \eta}$. However, at this stage of the proof of Proposition \ref{prop4.4modified}, we want to consider the field $h_2$ because of its scale invariance property. We will use this property in the proof of the Proposition \ref{polysteps} below to bound the $D^{\ep}$-lengths between polynomially (in $\ep$) many pairs of points.

\begin{prop}
Let $U \subset \BB{H}$, let $h_2$ be the mean-zero part of a quantum wedge as in Definition~\ref{wedgedef}, and let $\eta$ be an independent chordal $SLE_8$ from $0$ to infinity in $\BB{H}$. Also define the approximate metric $\wt D^\ep_{\BB H , h_2,\eta}$ as in Definition~\ref{def-cell-dist-tilde}. Then, for any fixed $z,w \in U$ and $\zeta \in (0,1)$,
\eqb \label{eqn-firststep}
\wt{D}^{\ep}_{\BB{H},h_2, \eta}(z, w ; U)\leq \ep^{- 1/(d - \zeta)}
\eqe
with polynomially high probability as $\ep \rightarrow 0$.
\label{firststep}
\end{prop}

\begin{proof}
By possibly replacing $U$ by a smaller open subset of $\BB H$ containing $z$ and $w$, we can assume without loss of generality that $\ol U \subset \BB H$; note that shrinking $U$ can only increase $\wt D^\ep_{\BB{H},h_2, \eta}(\cdot,\cdot;U)$. 
By Lemma~\ref{prop4.4} and Lemma~\ref{dist-comparison-a} (with $h^{\op{whole}}$ and $\eta^{\op{whole}}$ defined as in Lemma~\ref{prop4.4}), we have, for each fixed $\zeta \in (0,1)$,
\begin{equation}
\wt{D}^{\ep}_{\BB{H},h^{\op{whole}}, \eta^{\op{whole}}}(z,w ; U )\leq \ep^{- 1/(d - \zeta)}
\label{4.4}
\end{equation}
with polynomially high probability  as $\ep \rightarrow 0$. We need to transfer from $(h^{\op{whole}},\eta^{\op{whole}})$ to $(h_2 ,\eta)$.  

\medskip
\noindent\textit{Step 1: replacing $\eta^{\op{whole}}$ by $\eta$.} 
We first show that \eqref{4.4} still holds with $\eta^{\op{whole}}$ replaced by $\eta$.
Choose an open set $V \supset \overline{U}$ whose closure is  contained in $\BB{H}$. By Lemma~\ref{lem-gms-harmonic-2.4}, the event that~\eqref{4.4} is satisfied is measurable with respect to $h^{\op{whole}}$ and the restriction to $\overline{V}$ of the imaginary geometry field associated to  $\eta^{\op{whole}}$. Thus, by Lemma~\ref{nikodym} and H\"{o}lder's inequality, 
\begin{equation}
\wt{D}^{\ep}_{\BB{H},h^{\op{whole}}, \eta}(z,w ; U )\leq \ep^{- 1/(d - \zeta')}
\label{4.4eta}
\end{equation}
with polynomially high probability  as $\ep \rightarrow 0$. 

\medskip
\noindent\textit{Step 2: replacing $h^{\op{whole}}$ by $h_2$.} We now show that we can replace $h^{\op{whole}}$ by $h_2$ in~\eqref{4.4eta}. Define $V$ as in Step 1.
By comparing the Green functions associated to the whole-plane and free boundary GFFs, we see that we can write
\[
h^{\op{whole}}|_{\ol{V}} = h^{\op{free}}|_{\ol{V}} + f_1,
\]
where $h^{\op{free}}$ is a free boundary GFF on $\BB{H}$ (normalized so that its circle average over $\bdy B_1(0) \cap \BB{H}$ is zero)  and $f_1$ is an independent random harmonic function on $\ol{V}$ which is a centered Gaussian process on $\ol{V}$ with covariances $\text{Cov}(f_1(x),f_1(y)) = - \log|x - \bar{y}|$.
Furthermore, we can couple $h^{\op{free}}$ with $h_2$ so that
\[
h^{\op{free}} = h_2 + f_2
\]
where $f_2$ is the function in $\mcl{H}_1(\BB{H})$ whose common value on $\bdy B(0,e^{-t})$ is equal to $B_{t}$, where $B_t$ is a standard Brownian motion.  

The function $g:= f_1 + f_2$ on $\ol{V}$ is a centered Gaussian process with $\|g\|_V = \sup_{v \in V} |g| < \infty$ almost surely; therefore, the Borell-TIS inequality gives $\BB{E} \|g\|_V  < \infty$,  $\sigma^2 := \sup_{\ol{V}} \BB{E} |g|^2 < \infty$, and
\eqb
\BB{P}\left[ \|g\|_V  - \BB{E} \|g\|_V > u \right] \leq \exp(-u^2/(2\sigma^2))
\label{delta_super_tail}
\eqe
for each $u > 0$.  

We deduce that, for any $\delta >0$, it holds with superpolynomially high probability as $\ep\rta 0$ that 
\eqb
\max_{z \in \ol U} |(h_2 -  h^{\op{whole}})(z)| \leq \BB{E}\|g\|_V - \delta \log \ep.
\label{delta_super}
\eqe
On the event that~\eqref{delta_super} holds, we can apply Lemma~\ref{lem:monotonicity} to the pair of fields $h_2$ and $h^{\op{whole}} + \BB{E}\|g\|_V - \delta \log \ep$. Since, by~\eqref{delta_super}, the difference $(h^{\op{whole}} + \BB{E}\|g\|_V - \delta \log \ep) - h_2$ is nonnegative, Lemma~\ref{lem:monotonicity} implies that
\eqb
\wt{D}_{\BB{H},h_2,\eta}^{\ep}(z,w;V)  \leq \wt{D}^{\ep^{1+2\delta}}_{\BB{H}, h^{\op{whole}}, \eta}(z,w ; U )
\label{2}
\eqe
for $\ep$ sufficiently small. Combining with~\eqref{4.4eta} and choosing $\delta$ and $\zeta'$ sufficiently small, depending on $\zeta$, yields the desired result.
\end{proof}

\subsection{Distance to the origin in the $Q$-quantum wedge}
\label{sec-wedge-dist}

In this subsection we will establish the following proposition, which in particular implies the existence of the path from $\ep^{-\delta} i/2$ to $0$ in $\mcl{G}^\ep$ discussed at the beginning of this section and depicted in Figure~\ref{fig-multiscale}. 

\begin{prop}
For each fixed $\zeta \in (0,1)$, we can choose $\delta > 0$ sufficiently small so that
\[
D^{\ep}\left( \frac{\ep^{-\delta} i}{2} ,0\right) \leq \ep^{- 1/(d - \zeta)}
\]
with polynomially high probability as $\ep \rightarrow 0$.
\label{allsteps}
\end{prop}

The proof of Proposition~\ref{allsteps} has two main steps; see Figure~\ref{fig-multiscale} for an illustration. The first step is to apply Proposition~\ref{firststep} $\ep^{-\xi}$ times to prove an upper bound for the $D^{\ep}$-distance between the points $\ep^{-\delta} i/2$ and $\ep^{-\delta} i/2^{\lfloor \ep^{-\xi} \rfloor} $ for a small but fixed $\xi > 0$ (Proposition~\ref{polysteps} below). As noted before the statement of Proposition~\ref{firststep} above, the key tool used in this step is the scale invariance of $h_2$. 
The second step is to show that the point $\ep^{-\delta} i/2^{ \lfloor \ep^{-\xi} \rfloor}$ is contained in the origin-containing cell of  $\mcl{G}^{\ep}$ with polynomially high probability as $\ep \rightarrow 0$ (Proposition~\ref{lastcell}).

We note that, for this second step, it is important that we are working with $i/2^{ \lfloor \ep^{-\xi} \rfloor}$ instead of $ \ep^p i $ for some $p > 0$.   Because of the $Q$-log singularity at the origin, the quantum mass of a ball of radius $r$ centered at the origin decays \emph{slower} than any polynomial in $r$ as $r \rta 0$.  This indicates that the point   $ \ep^p i $ will lie outside the cell $\eta([0,\ep])$ with probability tending to $1$ as $\ep \rta 0$.

\begin{lem}
\label{hh2bound}
Let $h_3$ denote the field whose common value on  $\bdy B_{e^{-t}}(0) \cap \BB{H}$ is equal to $Q t$.  Then, for each $\delta , \omega > 0$, we have that $h_2  + h_3 + \log{\ep^{-\omega}} - h$ is equal to a non-negative continuous function on $B_{\ep^{-\delta}}(0) \cap \BB{H}$ with polynomially high probability as $\ep \rta 0$.
\end{lem}

\begin{proof}
By Definition~\ref{wedgedef}, this follows from the fact that the probability that $\sup_{t \in [0, \log{\ep^{-\delta}}]} B_{2t}$ is greater than $\log{\ep^{-\omega}}$ is bounded from above by a positive power of $\ep$.
\end{proof}

\begin{prop} 
For each $\zeta \in (0,1)$, we have, for $\delta, \xi > 0$ sufficiently small (depending on $\zeta$),
\begin{equation}
D^{\ep}\left( \frac{\ep^{-\delta} i}{2} , \frac{\ep^{-\delta}  i}{ 2^{\left\lfloor \ep^{-\xi} \right\rfloor}} \right) \leq \ep^{- 1/(d - \zeta)}
\label{xipolysteps}
\end{equation}
with polynomially high probability as $\ep \rightarrow 0$.
\label{polysteps}
\end{prop}

\begin{proof} 
We have
\allb
  D^{\ep}\left( \frac{\ep^{-\delta} i}{2} , \frac{\ep^{-\delta}  i}{2^{\left\lfloor \ep^{-\xi} \right\rfloor}}\right) 
 &\leq 2   \wt{D}_{\BB H , h , \eta}^{\ep}\left( \frac{\ep^{-\delta}  i}{2} , \frac{\ep^{-\delta}  i}{2^{\left\lfloor \ep^{-\xi} \right\rfloor}}\right) 
\quad \text{by Lemma~\ref{dist-comparison-b} } \notag \\ 
 &\leq 2   \wt{D}_{\BB H , h , \eta}^{\ep}\left( \frac{\ep^{-\delta}  i}{2} , \frac{\ep^{-\delta}  i}{2^{\left\lfloor \ep^{-\xi} \right\rfloor}}; B_{\ep^{-\delta}}(0) \cap \BB{H} \right) .\notag
\alle
For any fixed $\omega>0$, Lemmas~\ref{lem:monotonicity} and~\ref{hh2bound} imply that this latter quantity is at most  
 \allb
&2   \wt{D}_{\BB H , h_2 + h_3 + \log \ep^{-\omega} , \eta}^{\ep}\left( \frac{\ep^{-\delta}  i}{2} , \frac{\ep^{-\delta}  i}{2^{\left\lfloor \ep^{-\xi} \right\rfloor}}; B_{\ep^{-\delta}}(0) \cap \BB{H} \right) 
\notag \\
&\qquad2   \wt{D}_{\BB H , h_2 + h_3 , \eta}^{\ep^{1 + \sqrt{2} \omega}}\left( \frac{\ep^{-\delta}  i}{2} , \frac{\ep^{-\delta}  i}{2^{\left\lfloor \ep^{-\xi} \right\rfloor}}; B_{\ep^{-\delta}}(0) \cap \BB{H} \right) 
\notag \qquad \text{by Lemma~\ref{lem:distcomparison}} \\
&\qquad  \leq 2 \sum_{j=1}^{\left\lfloor \ep^{-\xi}-1 \right\rfloor} \wt{D}_{\BB H , h_2 + h_3 , \eta}^{\ep^{1 + \sqrt{2} \omega}}\left( \frac{\ep^{-\delta}  i}{2^j} , \frac{\ep^{-\delta} i}{2^{j+1}} ;B_{\ep^{-\delta}}(0) \cap \BB{H}\right) \notag \\
&\qquad  \leq 2 \sum_{j=1}^{\left\lfloor \ep^{-\xi}-1 \right\rfloor} \wt{D}_{\BB H , h_2 + (j+2) Q \log{2} , \eta}^{\ep^{1 + \sqrt{2} \omega}}\left( \frac{\ep^{-\delta}  i}{2^j} , \frac{\ep^{-\delta} i}{2^{j+1}} ; \left( B_{\ep^{-\delta} 2^{-j+1}}(0) \setminus B_{\ep^{-\delta} 2^{-j-2}}(0) \right) \cap \BB{H}\right) 
\label{sumscales}
\alle
with polynomially high probability as $\ep \rta 0$.
By the scale invariance of the law of $h_2$, the LQG coordinate change formula~\cite[Proposition 2.1]{shef-kpz}, and the scale invariance of the law of $\eta$, for $j\in\BB N$ the distance
\[
\wt{D}_{\BB H , h_2  + Q (j+2) \log{2}, \eta}^{\ep^{1 + \sqrt{2} \omega}}\left( \frac{\ep^{-\delta} i}{2^j} , \frac{\ep^{-\delta} i}{2^{j+1}} ;\left( B_{\ep^{-\delta} 2^{-j+1}}(0) \setminus B_{\ep^{-\delta} 2^{-j-2}}(0) \right) \cap \BB{H}\right)  \]
is equal in distribution to
\[ 
\wt{D}_{\BB H , h_2 + Q \log{8} + Q\log{\ep^{-\delta}}, \eta}^{\ep^{1 + \sqrt{2} \omega}}\left( \frac{ i}{2 } , \frac{i}{4} ; \left( B_{1}(0) \setminus B_{1/8}(0) \right) \cap \BB{H}\right) , 
\]
which by Lemma \ref{lem:distcomparison} is equal to 
\[ 
\wt{D}_{\BB H , h_2, \eta}^{ 8^{-\sqrt 2 Q} \ep^{1 + \delta'} }\left( \frac{ i}{2 } , \frac{i}{4} ; \left( B_{1}(0) \setminus B_{1/8}(0) \right) \cap \BB{H}\right) 
\]
for some $\delta' > 0$ which can be made arbitrarily small by choosing $\delta,\omega>0$ sufficiently small. 
Thus, Proposition \ref{firststep} and a union bound over all $j\in [1,\left\lfloor \ep^{-\xi}-1 \right\rfloor]_{\BB Z}$ shows that for any $\zeta'\in (0,1)$ and $\delta > 0$, if we choose $\xi$ sufficiently small (depending on $\zeta'$) then~\eqref{sumscales} is bounded by $ \ep^{ - \xi -\frac{1+ \delta'}{d - \zeta'}}$ with polynomially high probability as $\ep \rightarrow 0$. In particular, we can choose $\xi$, $\delta$, $\omega$ and $\zeta'$ sufficiently small that  $\xi + (1+\delta)/(d-\zeta')$ is less than $1/(d-\zeta)$. 
\end{proof}

We now turn our attention to showing that $ \eta([0,\ep])$ contains $\ep^{-\delta}  i/2^{\lfloor \ep^{-\xi} \rfloor}$ with polynomially high probability as $\ep \rta 0$. We will prove the following slightly stronger statement.  
 
\begin{prop}
For any $\xi>0$, the set  $B_{2^{-\ep^{-\xi}}}(0) \cap \BB{H}$ is contained in $\eta([0,\ep])$ with polynomially high probability as $\ep \rightarrow 0$.  
\label{lastcell}
\end{prop}

Note that we dropped the factor of $\ep^{-\delta}$ since, for any $\xi' < \xi$, the ball of radius $\ep^{-\delta} 2^{-\ep^{-\xi}}$ is contained in the ball of radius $2^{-\ep^{-\xi'}}$ for $\ep$ sufficiently small.
 
To prove Proposition~\ref{lastcell}, we will show that, with polynomially high probability as $\ep \rightarrow 0$, two events occur: 
\begin{enumerate}[label=(\alph*)]
\item  \label{item-lastcell-area}
 the set  $B_{\ep^{-1} 2^{-\ep^{-\xi}}}(0) \cap \BB{H}$  has $\mu_h$-area $\leq \ep$, and 
\item  \label{item-lastcell-absorb}
$\eta$ absorbs $B_{2^{-\ep^{-\xi}}}(0) \cap \BB{H}$ before hitting $\bdy B_{\ep^{-1} 2^{-\ep^{-\xi}}}(0) \cap \BB{H}$. 
\end{enumerate}
This implies the proposition since, on the intersection of the events~\ref{item-lastcell-area} and~\ref{item-lastcell-absorb},  $B_{2^{-\ep^{-\xi}}}(0) \cap \BB{H}$ is necessarily contained in $\eta([0,\ep])$.
First, we prove that the event~\ref{item-lastcell-area} occurs with polynomially high probability as $\ep \rightarrow 0$.

\begin{lem}
For any fixed $\xi>0$, it holds with polynomially high probability as $\ep\rta 0$ that 
$\mu_h\left( B_{\ep^{-1} 2^{-\ep^{-\xi}}}(0) \cap \BB{H} \right) \leq \ep$.
\label{arealastcell}
\end{lem}

We first need the following elementary estimate for the LQG area measure.

\begin{lem} \label{lem-free-bdy-moment}
Let $h $ be a free-boundary GFF on $\BB H$, let $h_2 = h - h_{|\cdot|}(0)$ be its mean-zero part, and let $\mu_{h_2}$ be the associated $\gamma$-LQG measure for some $\gamma \in (0,2)$. 
There exists $p > 0$ such that for every $r \in (0,1)$, 
\eqb
\BB E\left( \left[ \mu_{h_2}\left( (\BB D\setminus B_r(0)) \cap \BB H \right) \right]^p \right) < \infty .
\eqe
\end{lem}
\begin{proof}
The field $h_2$ does not depend on the choice of additive constant for $h$, so we can normalize so that $h_1(0) = 0$, where here $h_s(0)$ is the semicircle average of $h$ over $\bdy B_s(0) \cap\BB H$.  
By H\"older's inequality, for $p>0$ and $q > 1$, 
\allb \label{eqn-free-bdy-moment-split}
&\BB E\left( \left[ \mu_{h_2}\left( (\BB D\setminus B_r(0)) \cap \BB H \right) \right]^p \right) \notag \\
&\qquad \leq \BB E\left( \left[ \mu_h\left( (\BB D\setminus B_r(0)) \cap \BB H \right) \right]^{q p} \right)^{1/q} \BB E\left( \exp\left( \sqrt 2 \frac{qp}{q-1} \sup_{s \in [r,1]} |h_s(0)|  \right) \right)^{1-1/q}
\alle
The process $t\mapsto h_{e^{-t}}(0)$ is a standard linear Brownian motion, so the second factor on the right of~\eqref{eqn-free-bdy-moment-split} is finite for any $p > 0$ and $q>1$. 
The proof of~\cite[Proposition 2.3]{rhodes-vargas-review} shows that the total mass of the $\gamma$-LQG measure associated with a free-boundary GFF on $\BB D$ has a finite moment of some positive order. Changing coordinates to $\BB H$ and using the LQG coordinate change formula~\cite[Proposition 2.1]{shef-kpz} shows that $\mu_h(\BB D\cap \BB H)$ has a finite moment of some positive order. Hence the right side of~\eqref{eqn-free-bdy-moment-split} is finite for small enough $p$. 
\end{proof}

\begin{proof}[Proof of Lemma~\ref{arealastcell}]
Recall the 3-dimensional Bessel process $X_t $ from Definition \ref{wedgedef}, defined so that $h_{e^{-t}}(0)   = -  X_{2t} +  Q t$ for $t\geq 0$. Using that $X_t$ has the law of the modulus of a three-dimensional Brownian motion, we get that for each $c \in (0,e^{k/2})$,
\begin{align*}
\BB{P}\left[ \min_{e^k \leq t \leq e^{k+1}} X_t < c \right]
  \preceq  \BB{P}\left[ \min_{1 \leq t \leq e} X_t < e^{-k/2} c \right]^3  \preceq
 e^{-3k/2} c^3 
\end{align*}
where here $\preceq$ denotes inequality up to a constant that does not depend on $k$ or $c$.  
Therefore, for $k_0 \in \BB N$,
\allb \label{eqn-bessel-max}
\BB{P}[\text{$X_t < t^{1/3}$ for some $t \geq e^{k_0}$}] 
 &\leq \sum_{k = k_0}^{\infty} \BB{P}\left[ \min_{e^k \leq t \leq e^{k+1}} X_t < (k+1)^{1/3} \right]  \notag \\
 &\preceq \sum_{k = k_0}^{\infty} e^{-3k/2} (k+1)
\preceq e^{-k_0/2} .
\alle
For $p\in (0,1)$, the function $x\mapsto x^p$ is concave, hence subadditive. Combining this with~\eqref{eqn-bessel-max} (applied with $k_0$ proportional to $-\log\log(\ep^{-1} 2^{-\ep^{-\zeta}})$) shows that with polynomially high probability as $\ep\rta 0$,  
\allb \label{eqn-bessel-sum}
 &\left[\mu_{h}\left(   B_{\ep^{-1} 2^{-\ep^{-\xi}}}(0) \cap \BB{H}  \right)\right]^p \notag \\
 &\qquad \leq \sum_{\ell = \lfloor \ep^{-\xi} \log{2} +  \log{\ep} \rfloor}^{\infty} \left[\mu_{h}\left( \left(   B_{e^{-\ell}}(0) \backslash B_{e^{-\ell-1}}(0) \right) \cap \BB{H}  \right)\right]^p \notag  \\
 &\qquad \leq \sum_{\ell = \lfloor \ep^{-\xi} \log{2} +  \log{\ep} \rfloor}^{\infty} e^{-\sqrt{2} p \min_{\ell \leq t \leq \ell + 1} (X_{2t} - Q t)} \left[\mu_{h_2}\left( \left(   B_{e^{-\ell}}(0) \backslash B_{e^{-\ell-1}}(0) \right) \cap \BB{H}  \right)\right]^p \qquad \text{(Def.~\ref{wedgedef})} \notag \\
 &\qquad \leq \sum_{\ell = \lfloor \ep^{-\xi} \log{2} +  \log{\ep} \rfloor}^{\infty} e^{-\sqrt{2} p \min_{\ell \leq t \leq \ell + 1} X_{2t} } \left[\mu_{h_2 + Q \ell}\left( \left(   B_{e^{-\ell}}(0) \backslash B_{e^{-\ell-1}}(0) \right) \cap \BB{H}  \right)\right]^p\notag   \\
 &\qquad \leq \sum_{\ell = \lfloor \ep^{-\xi} \log{2} +  \log{\ep} \rfloor}^{\infty} e^{-\sqrt{2} p (2\ell)^{1/3}} \left[\mu_{h_2 + Q \ell}\left( \left(   B_{e^{-\ell}}(0) \backslash B_{e^{-\ell-1}}(0) \right) \cap \BB{H}  \right)\right]^p \quad \text{(by~\eqref{eqn-bessel-max})} .
\alle

By the LQG coordinate change formula~\cite[Proposition 2.1]{shef-kpz} and the scale invariance of the law of $h_2$, 
\eqb \label{eqn-h2-scale}
\mu_{h_2 + Q \ell}\left( \left(   B_{e^{-\ell}}(0) \backslash B_{e^{-\ell-1}}(0) \right) \cap \BB{H}  \right) \eqD \mu_{h_2  }\left( \left(   B_1(0) \backslash B_{e^{-1}}(0) \right) \cap \BB{H}  \right) .
\eqe 
By Lemma~\ref{lem-free-bdy-moment}, we can choose $p  \in (0,1)$ such that  
\eqb \label{eqn-h2-moment}
\BB E\left( \left[ \mu_{h_2  }\left( \left(   B_1(0) \backslash B_{e^{-1}}(0) \right) \cap \BB{H}  \right) \right]^p \right) < \infty . 
\eqe  

Since~\eqref{eqn-bessel-sum} holds with probability at least $1-O_\ep(\ep^\alpha)$ for some $\alpha > 0$, the Chebyshev inequality shows that  
\begin{align*}
&\BB P\left(  \mu_{h}\left(   B_{\ep^{-1} 2^{-\ep^{-\xi}}}(0) \cap \BB{H}  \right)  > \ep \right) \notag\\
&\qquad \leq \ep^{-p} 
\sum_{\ell = \ep^{-\xi} \log{2} +  \log{\ep}}^{\infty} e^{-\sqrt{2} p (2\ell)^{1/3}} \BB{E}\left( \left[\mu_{h_2 + Q \ell}\left( \left(   B_{e^{-\ell}}(0) \backslash B_{e^{-\ell-1}}(0) \right) \cap \BB{H}  \right)\right]^p  \right) + O_\ep(\ep^\alpha) \\
&\qquad \preceq \ep^{-p} \sum_{\ell = \lfloor \ep^{-\xi} \log{2} +  \log{\ep} \rfloor}^{\infty} e^{-\sqrt{2} p (2\ell)^{1/3}} + O_\ep(\ep^\alpha)  \quad \text{by \eqref{eqn-h2-scale} and \eqref{eqn-h2-moment}} .
\end{align*}
Since this last sum decays faster than any power of $\ep$, we conclude the proof. 
\end{proof} 

Next, we turn to the event~\ref{item-lastcell-absorb} above.  That this event has polynomially high probability as $\ep \rightarrow 0$ follows from the following basic SLE estimate.
 
\begin{prop}
Let $\kappa \geq 8$ and let $\eta$ be a chordal SLE$_\kappa$ from 0 to $\infty$ in $\BB H$. It holds with polynomially high probability as $\ep \rta 0$, uniformly over all $R > 0$, that $\eta$ hits every point of $B_R(0) \cap \BB{H}$ before exiting $ B_{  R / \ep}(0) $.
\label{slecor}
\end{prop}

Proposition~\ref{slecor} will follow from the following lemma together with the Markov property of SLE. 

\begin{lem}
For $r >0$, let $\sigma_r$ be the first time $\eta$ exits $B_r(0)$. Almost surely, the conditional probability given $\eta|_{[0,\sigma_r]}$ that $\eta$ hits every point of $B_{r}(0) \cap \BB H$ before leaving $B_{8r}(0)$ is at least some $p > 0$ which does not depend on $r$.   
\label{slelemma}\end{lem}

\begin{proof} 
We parametrize $\eta$ so that the half-plane capacity of $\eta([0,t])$ is $2t$.
Let $\{g_t\}_{t\geq 0}$ be the Loewner maps associated with $\eta$, with the hydrodynamic normalization. 
By \cite[Corollary 3.44]{lawler-book}, we have $|g_{\sigma_r}(z) - z| \leq 3r$ for each $z\in \BB H\setminus \eta([0,\sigma_r])$. Thus, the semicircle $\bdy B_{r}(0) \cap \BB{H}$ is mapped by $g_{\sigma_r}$ into  $B_{4r}(0) \cap \BB{H}$, and the semicircle $\bdy B_{8r}(0) \cap \BB{H}$ is mapped by $g_t$ into the complement of $B_{5r}(0) \cap \BB{H}$.  The result now follows from conformal invariance and the domain Markov property since $\eta$ has positive probability to hit every point of $B_4(0)\cap \BB H$ before leaving $B_5(0)$. 
\end{proof}

\begin{proof}[Proof of Proposition~\ref{slecor}]
Applying Lemma~\ref{slelemma} to $r = 8^j R$ for $0 \leq j \leq \left\lfloor \frac{\log{\ep^{-1}}}{\log 8} \right\rfloor$, we deduce that the probability in question is at least $1 - (1-p)^{\left\lfloor \frac{\log{\ep^{-1}}}{\log 8} \right\rfloor} \geq 1 - \ep^q$ for some constant $q$, as desired.
\end{proof}

\begin{proof}[Proof of Proposition~\ref{lastcell}]
Applying Proposition~\ref{slecor} with $R = 2^{-\ep^{-\xi}}$, we deduce that with polynomially high probability as $\ep\rta 0$, the SLE$_8$ curve $\eta$ absorbs $B_{2^{-\ep^{-\xi}}}(0) \cap \BB{H}$ before hitting $\bdy B_{\ep^{-1} 2^{-\ep^{-\xi}}}(0) \cap \BB{H}$. Combining this with Lemma~\ref{arealastcell} yields the desired result.
\end{proof}

\begin{proof}[Proof of Proposition~\ref{allsteps}]
Combine Propositions~\ref{polysteps} and~\ref{lastcell}.
\end{proof}

\subsection{Finishing the proof}

To finish the proof of Proposition~\ref{prop4.4modified}, we show that $\eta(0,1]$ is unlikely to contain the point  $\ep^{-\delta} i/2$, so that the path from $0$ to $\ep^{-\delta} i/2$ in $\BB{H}$ yields a path from $0$ to the upper boundary $\overline{\partial}_{\ep}(0,1]$ of $\eta(0,1]$.  We will deduce this fact from Proposition~\ref{slecor} and the following probablistic LQG-area lower bound.

\begin{lem} \label{lem-area}
It holds with polynomially high probability as $\ep \rta 0$ that $\mu_h(B_{\ep^{-\delta/2}}(0)\cap \BB H) \geq 1$.
\end{lem}
\begin{proof}
We can write
\allb \label{eqn-area-split}
&\mu_h(B_{\ep^{-\delta/2}}(0)\cap \BB H)  \notag \\
&\qquad\geq \mu_h\left( \left( B_{\ep^{-\delta/2}}(0) \setminus B_{\ep^{-\delta/2}/2}(0)\right) \cap \BB H \right) \notag\\
&\qquad \geq \exp\left( \sqrt 2 \inf_{t\in [\log(\ep^{\delta/2}) , \log(2\ep^{\delta/2})]} h_{e^{-t}}(0) \right) \mu_{h_2}\left( \left( B_{\ep^{-\delta/2}}(0) \setminus B_{\ep^{-\delta/2}/2}(0)\right) \cap \BB H \right) . 
\alle
Recall from Definition~\ref{wedgedef} that $h_{e^{-t}}(0) = B_{-2t} + Q t$ for $t<0$, where $B$ is a standard Brownian motion independent from $h_2$. By the Gaussian tail bound, for each $\alpha > 0$, it holds with polynomially high probability as $\ep\rta 0$ that 
\eqb \label{eqn-area-radial}
\inf_{t\in  [\log(\ep^{\delta/2}) , \log(2\ep^{\delta/2})]} h_{e^{-t}}(0) \geq  (Q - \alpha)\log(\ep^{\delta/2}) .
\eqe
By the LQG coordinate change formula~\cite[Proposition 2.1]{shef-kpz} and the scale invariance of the law of $h_2$,   
\eqbn
\mu_{h_2}\left( \left( B_{\ep^{-\delta/2}}(0) \setminus B_{\ep^{-\delta/2}/2}(0)\right) \cap \BB H \right)
 \eqD (\ep^{-\delta/2})^{\sqrt 2 Q} \mu_{h_2}\left( ( B_1(0)\setminus B_{1/2}(0) ) \cap \BB H     \right)  .
\eqen
By a standard lower tail estimate for the LQG measure (using, e.g.,~\cite[Lemma 4.6]{shef-kpz} and a comparison between $h_2$ and a zero-boundary GFF on $\BB H$ away from $\bdy\BB H$), for each $\alpha' > 0$, it holds with polynomially high probability as $\ep \rta 0$ that 
\eqb \label{eqn-area-meanzero}
\mu_{h_2}\left( \left( B_{\ep^{-\delta/2}}(0) \setminus B_{\ep^{-\delta/2}/2}(0)\right) \cap \BB H \right) \geq  \ep^{\alpha'-\sqrt 2 Q \delta/2}
\eqe
Plugging~\eqref{eqn-area-radial} and~\eqref{eqn-area-meanzero} into~\eqref{eqn-area-split} concludes the proof.
\end{proof}

\begin{proof}[Proof of Proposition~\ref{prop4.4modified}]
By Proposition~\ref{slecor}, it holds with polynomially high probability as $\ep \rta 0$ that $\eta$ absorbs $B_{\ep^{-\delta/2}}(0) \cap \BB{H}$ before hitting $\bdy B_{\ep^{-\delta}/2}(0) \cap \BB{H}$. By Lemma~\ref{lem-area},  the set $B_{\ep^{-\delta/2}}(0) \cap \BB{H}$ has $\mu_h$-mass at least $1$ with polynomially high probability as $\ep \rta 0$. Since $\eta$ is parametrized by $\mu_h$-mass, it holds with polynomially high probability as $\ep  \rta 0$ that $\eta[0,1]$ does not contain the point $\ep^{-\delta} i/2$. This means that each path of cells from 0 to $\ep^{-\delta} i/2$ must include a cell corresponding to a vertex of the upper boundary $\overline{\partial}_{\ep}(0,1]$.
Proposition~\ref{allsteps} therefore implies that with polynomially high probability as $\ep\rta 0$, one has $\op{dist}(\ep , \ol\bdy_\ep(0,1] ; \mcl G^\ep|_{(0,1]}) \leq \ep^{-1/(d-\zeta)}$, as desired.  
\end{proof}
 
\section{Open problems}
\label{sec-open-problems}

Here we list a few of the many interesting open problems related to the results of this paper. 
Suppose we are in the setting of Theorem~\ref{main3}, so that $\{X_m\}_{m\in\BB N}$ are the sequence of growing clusters of external DLA on the infinite spanning-tree-weighted random planar map $(M,v_0)$, started from $v_0$ and targeted at $\infty$. 
Theorem~\ref{main3} gives the growth exponent for the diameter of external DLA clusters w.r.t.\ the ambient graph distance on $M$. 
It is also of interest to understand the diameter of the cluster $X_m$ with respect to its \emph{internal} graph distance, i.e., we view $X_m$ as a tree (without regard for its embedding into $M$) and consider the graph distance in this tree. 
 
\begin{prob}[Internal diameter of DLA clusters] \label{prob-internal-diameter}
Show that the following limit exists and compute its value:
\eqbn
\lim_{m\rta\infty} \frac{\log \op{diam} (X_m)}{\log m} .
\eqen
\end{prob}

Currently, we do not even know that $\op{diam} (X_m)$ grows sublinearly in $m$. To prove this, one would need to show that the maximum over all $v\in X_m$ of the harmonic measure from $\infty$ of $v$ in $M\setminus X_m$ tends to zero as $m\rta\infty$ (see, e.g., the arguments in~\cite{benjamini-yadin-dla}). 

Another interesting question about the geometry of DLA on $M$ is the following. 

\begin{prob}[One-endedness] \label{prob-one-ended}
Is external DLA on the UITM almost surely one-ended? That is, if we let $X_\infty := \bigcup_{m=1}^\infty X_m$, is it almost surely the case that $X_\infty \setminus B_r^{X_\infty}(v_0)$ has exactly one infinite connected component for every $r>0$? 
\end{prob}

Our arguments are very specific to the spanning-tree-weighted random planar map due to the relationship between DLA and LERW in this setting. 
However, it is also of interest to study DLA on other random planar maps.

\begin{prob}[DLA on other maps] \label{prob-other-maps}
What can be said about external DLA on other natural infinite random planar maps, like the uniform infinite planar quadrangulation/triangulation~\cite{angel-schramm-uipt,krikun-maps} or the infinite-volume limits of planar maps decorated by bipolar-orientations~\cite{kmsw-bipolar} or the Ising model~\cite{ams-ising-local}?
\end{prob}

\bibliography{cibib,cibiblong,dlasupplementarybib}
\bibliographystyle{hmralphaabbrv}

\end{document}